\documentclass[final]{siamltex}

\usepackage{multirow}

\usepackage{graphicx}
\usepackage{amsmath}
\usepackage{amssymb}
\usepackage[square,numbers]{natbib}

\newtheorem{xxxexample}{Example}[section]
\newenvironment{example}{\begin{xxxexample}}{\end{xxxexample}\vspace{-0.5cm}\rightline{$\Diamond$}}

\title{Computation of Balanced Equivalence Relations  and their Lattice for a Coupled Cell Network}

\author{Hiroko Kamei\thanks{Division of Mathematics, University of Dundee, Dundee, DD1 4HN, UK ({\tt hiroko@maths.dundee.ac.uk}).}
        \and Peter J. A. Cock\thanks{James Hutton Institute, Dundee, DD2 5DA, UK ({\tt Peter.Cock@hutton.ac.uk})}}
\date{Submitted December 28, 2010; revision submitted November 2011}

\begin{document}

\maketitle

\begin{abstract}
A coupled cell network describes interacting (\textit{coupled}) individual systems (\textit{cells}). As in networks from real applications, coupled cell networks can represent \textit{inhomogeneous} networks where different types of cells interact with each other in different ways, which can be represented graphically by different symbols, or abstractly by equivalence relations.

Various synchronous behaviors, from full synchrony to partial synchrony,
can be observed for a given network. Patterns of synchrony, which do not depend on specific dynamics of the network, but only on the network structure, are associated with a special type of partition of cells, termed \textit{balanced equivalence relations}. Algorithms in Aldis (2008) and Belykh and Hasler (2011) find the unique pattern of synchrony with the least clusters. In this paper, we compute the set of all possible patterns of synchrony and show their hierarchy structure as a \textit{complete lattice}.

We represent the network structure of a given coupled cell network by a symbolic adjacency matrix encoding the different coupling types. We show that balanced equivalence relations can be determined by a matrix computation on the adjacency matrix which forms a block structure for each balanced equivalence relation. This leads to a computer algorithm to search for all possible balanced equivalence relations.
Our computer program outputs the balanced equivalence relations, quotient matrices, and a complete lattice for user specified coupled cell networks.
Finding the balanced equivalence relations of any network of up to $15$ nodes is tractable, but for larger networks this depends on the pattern of synchrony with least clusters.
\end{abstract}



\pagestyle{myheadings}
\thispagestyle{plain}
\markboth{HIROKO KAMEI AND PETER J. A. COCK}{COMPUTATION OF BALANCED EQUIVALENCE RELATIONS}

\section{Introduction}
In many areas of science, interacting objects can be represented as a network. Examples can be found in biological, chemical, physical, technological and social systems~\cite{Newman-2003,Strogatz-2001}. One important dynamical feature of networks is the possibility of synchrony, which occurs when distinct individuals exhibit identical dynamics. Synchronization of initially distinct dynamics, such as that appearing in fireflies, coupled lasers and coupled chaotic systems have been  extensively studied (see \cite{Boccaletti,Arenas,Wang} for reviews). Many studies have investigated the role of synchrony in a wide range of cognitive and information processing, including recently the possible relevance of neural synchrony in pathological brain states, such as epilepsy and Alzheimer's disease (reviewed in \cite{Uhlhaas}).

In this paper however we are interested in not only full synchronization, but also partial synchronization where a network breaks into sub-networks, called clusters, such that all individual systems within one cluster are perfectly synchronised.  In coupled chaotic systems, such partial synchronization (or clustering) has been attracting great interest \cite{Belykh2000,Zhang2001,Belykh2008,Wang2010}. Partial synchronies can also appear as a result of synchrony breaking. Suppose all individuals of a network are initially synchronized, but at some point lose coherence -- giving synchronized sub-networks or even differently behaving individuals. Differentiation of (biological) cells \cite{Koseska2010}, speciation~\cite{Stewart-speciation}, desynchronization of coupled oscillators~\cite{Kuramoto}, or the loss of coherence in lasers can be interpreted in this way. We are interested in computing those potential partial synchronies which are solely determined by the network architecture (topology), rather that any specific details of the network dynamics (e.g. parameter values or function forms).

Mathematically such interaction networks are described as a directed graph~\cite{Tutte,Wilson}. Nodes correspond to the individuals, and arrows (edges) denote their interactions. Coupled cell networks are a general formalism using a directed graph to describe such interacting individuals (see~\cite{Stewart1,Golubitsky1,Golubitsky-and-Stewart-2006}). In this settings cells correspond to the individuals (graph nodes), and there can be multiple types of cell, and similarly multiple types of arrow (graph edges). The topology of the network is described by an adjacency matrix, using symbolic entries for different arrow types.

In this paper we describe how to compute all possible partial synchronies which are a consequence of the topology (i.e. the adjacency matrix for the graph) of a given coupled cell network.
Such partial synchronies are represented as a partition of the network cells, termed a \textit{balanced equivalence relation} (also referred to as a balanced coloring).
These impose a block structure on the adjacency matrix, leading to a computer algorithm which determines all possible balanced equivalence relations using matrix computations.
The set of all possible balanced equivalence relations is partially ordered, and forms a complete lattice (see \cite{Stewart}). 
In this paper, we compute all possible balanced equivalence relations and their complete lattice.
Existing algorithms in \citep{Aldis,BelykhHasler2011} can find the top lattice node, i.e. the maximal balanced equivalence relations from a given network topology.
We use the top lattice node in order to reduce the search space for finding all possible balanced equivalence relations, and so speed up constructing the complete lattice.

The supplementary material includes an implementation of the algorithms described,
and a hybrid of the top lattice node algorithms from \cite{Aldis,BelykhHasler2011},
using the freely available programming language Python (http://www.python.org) and Numerical Python library (http://numpy.scipy.org) for matrix support. We hope that interested researchers will be able to take this script and run it on networks of interest by entering the adjacency matrices. The code prints out balanced equivalence relations, their quotient network adjacency matrices, and the associated lattice structure (as text). Additionally, provided GraphViz \cite{Graphviz} and associated Python libraries for calling it are installed, figures of the network and lattice are also produced.

Our implementation effectively finds all the balanced equivalence relations for any coupled cell network of up to $15$ cells.For larger networks, the speed of computation depends on the total number of possible partitions to check, which depends on the clustering pattern of the maximal balanced equivalence relation, but not directly the size of the network. Inhomogeneous networks are most tractable, and an example of a $30$ cell network is discussed.

This paper is organized as follows. 
In Section~\ref{sec:preliminaries}, we recall some basic features of the coupled cell formalism, then review the basics of lattice theory. 
In Section~\ref{sec:matrix_computation}, we show that finding a balanced equivalence relation is equivalent to finding a particular type of invariant subspace of a linear map, which is represented by the adjacency matrix of a given coupled cell network. We then show that an adjacency matrix, which leaves such subspaces invariant, has a block structure. This matrix property leads to the computer algorithm discussed in Section~\ref{sec:algorithm} which determines all possible balanced equivalence relations of a given network, allowing display of the corresponding lattice of the set of all balanced equivalence relations.
Finally in Section \ref{sec:examples}, we demonstrate how the algorithm can be applied to several examples from the literature, with conclusions in Section \ref{sec:conclusions}.

\section{Preliminaries}
\label{sec:preliminaries}
\subsection{Coupled cell network and associated coupled cell system} 

A \textit{coupled cell network} describes interacting individual systems schematically by a finite directed graph $\mathcal{G}=(\mathcal{C},\mathcal{E},\sim_{C},\sim_{E})$. Here $\mathcal{C}=\{c_{1},c_{2}\ldots,c_{n}\}$ is the set of nodes (\textit{cells}), $\mathcal{E}=\{e_{1},e_{2},\ldots,e_{m}\}$ is the arrows between them (\textit{couplings}), and equivalence relations $\sim_{C}$ and $\sim_{E}$ describe different types of cells and couplings.

Different cell types (equivalence relation $\sim_{C}$) can be \textit{labelled} with symbols such as circles, squares, or triangles. Similarly different arrow types ($\sim_{E}$) can be shown using different kinds of lines (solid, dashed, dotted).

Each cell $c$ is a dynamical system with variables $x_c$ in \textit{cell phase space} $P_{c}$, for simplicity a finite-dimensional real vector space $\mathbb{R}^{r}$, where $r$ may depend on $c$.
Cells of the same type must have the same phase space, for $c, d \in \mathcal{C}$, $c\sim_{c}d \Rightarrow  P_{c}=P_{d}$.

Each arrow $e\in\mathcal{E}$ connects a tail node $\mathcal{T}(e)$ to a head node $\mathcal{H}(e)$, expressed using maps $\mathcal{H}:\mathcal{E}\rightarrow\mathcal{C}$ and $\mathcal{T}:\mathcal{E}\rightarrow\mathcal{C}$. Arrows of the same type must have matching tail and head cell types, $e_1, e_2 \in \mathcal{E}$, $e_{1}\sim_{E}e_{2} \Rightarrow \mathcal{H}(e_{1})\sim_{C}\mathcal{H}(e_{2}) \text{ and } \mathcal{T}(e_{1})\sim_{C}\mathcal{T}(e_{2})$, ensuring $\sim_{E}$ maintains similar input/output characteristics.

For each cell $c\in\mathcal{C}$ the \textit{input set} of $c$ is defined as $I(c)=\{e\in\mathcal{E}: \mathcal{H}(e)=c\}$, where $e\in I(c)$ is called an \textit{input arrow} of $c$. 
\begin{definition}
The relation $\sim_{I}$ of \textit{input equivalence} on
$\mathcal{C}$ is defined by $c\sim_{I}d$ if and only if there
exists an arrow-type preserving bijection $\beta: I(c)\rightarrow I(d)$ such that $i\sim_{E}\beta(i)$ $\forall i\in I(c)$, and two cells $c$ and $d$ are said to be input isomorphic.
\end{definition}

The \textit{input equivalence} on $\mathcal{C}$ identifies equivalent cells whose input couplings are also equivalent. As a consequence, if $c\sim_{I}d$ then they should have similar dynamics defined in a \textit{coupled cell system}, which is a system of ordinary differential equations (ODEs) associated with a given coupled cell network. Define the \textit{total phase space} to a given $n$-cell coupled cell network $\mathcal{G}$ to be $\displaystyle P=P_{c_{1}}\times\cdots\times P_{c_{n}}$ and employ the coordinate system $(x_{1}, \ldots , x_{n})\in P$, where $x_{i}\in P_{c_{i}}$. The system associated with the cell $c_{i}$ has the form
\begin{displaymath}
\dot{x}_{i} = f_{i}(x_{i},x_{j_{1}},\ldots,x_{j_{q}}),
\end{displaymath}
where the first variable of $f_{i}$ represents the internal dynamics of the cell $c_{i}$ and the remaining $q$ variables $\{x_{j_{1}},\ldots,x_{j_{q}}\}=\mathcal{T}(I(c_{i}))$  represent coupling. In this paper, we employ definitions which permit multiple arrows (some subsets of indices $j_{k}$ are equal) and self-coupling (some $j_{k}$ equal $i$). 

The function $f_{i}$ corresponds to the $i$-th component of an \textit{admissible vector field} $F=(f_{1},\ldots,f_{n})$, which is compatible with the network structure, and depends on a fixed choice of the total phase space $P$. It follows that different components of $F$ are identical if the corresponding cells are input isomorphic, i.e., $f_{c}=f_{d}$ for $c\sim_{I}d$. As a consequence, the number of distinct functions in a coupled cell system corresponds to the number of input equivalence classes. We now define types of coupled cell networks.
\begin{definition}
A \textit{homogeneous} network is a coupled cell network such
that all cells are input isomorphic or identical. If a coupled cell network is not homogeneous, we call it an \textit{inhomogeneous} network. A homogeneous network that has one equivalence class of arrows is said to be \textit{regular}. The \textit{valency} of a homogeneous network is the number of arrows into each cell.
\end{definition}

Table~\ref{tab:coupled_cell_networks} shows various types of coupled cell networks and the corresponding coupled cell systems. Note that, by definition, $c\sim_{I}d\Longrightarrow c\sim_{C}d$ holds, but the converse fails in general. If $I(c)=\emptyset$ and $I(d)=\emptyset$, we say that $c\sim_{I}d$ iff $c\sim_{C}d$ (see for example the network $\mathcal{G}_{1}$ in Table \ref{tab:coupled_cell_networks}).

\begin{table}[p]
\begin{center}
\begin{tabular}{cclc}
\multicolumn{4}{c}{Inhomogeneous networks} \\
\multirow{5}{*}{$\mathcal{G}_{1}$} &
\multirow{5}{*}{\includegraphics[scale=0.2]{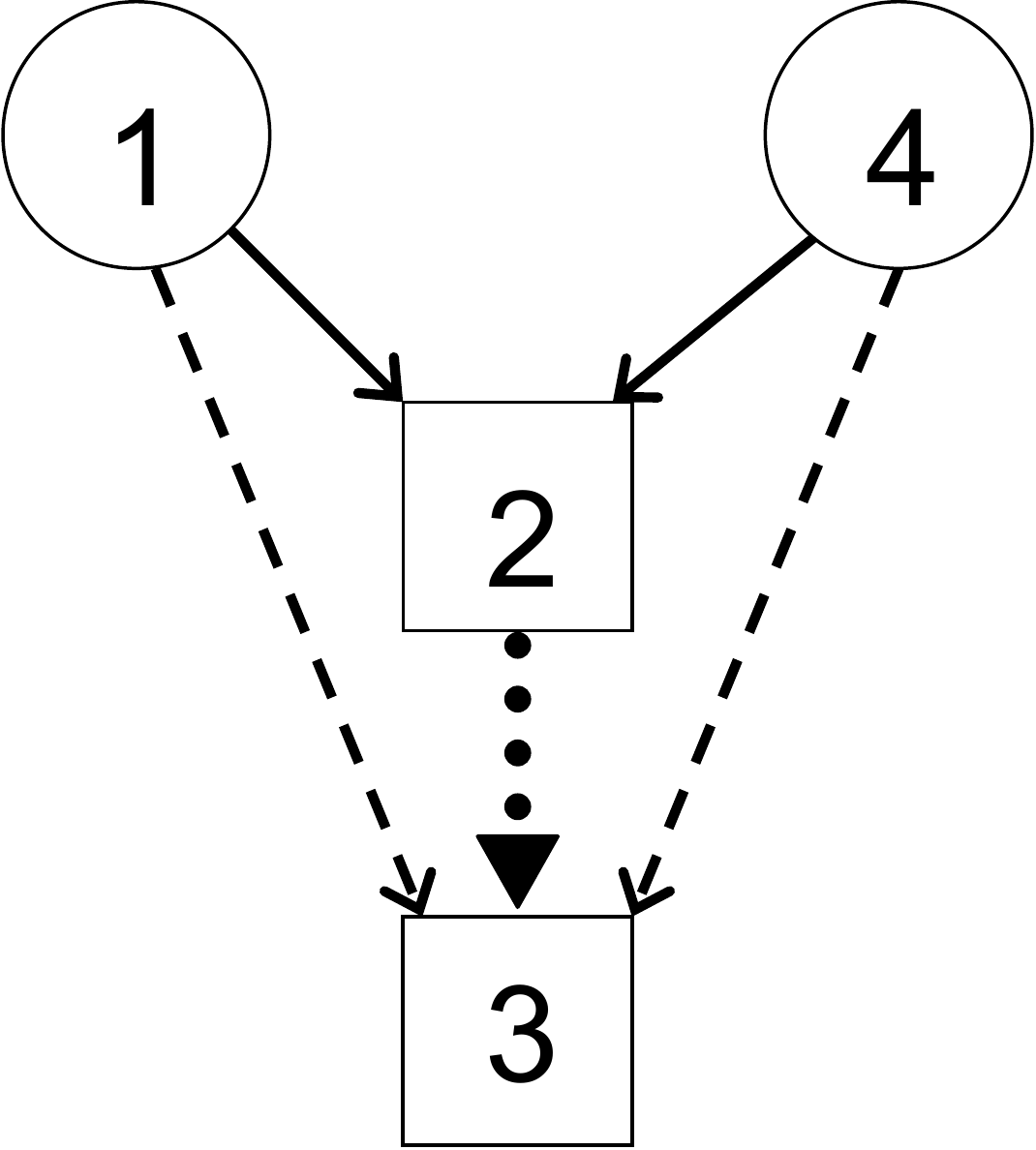}} &
\multirow{5}{*}{$\begin{array}{rl}
\dot{x}_{1}&= f_{1}(x_{1})\\
\dot{x}_{2}&= g_{1}(x_{2},\overline{x_{1},x_{4}})\\
\dot{x}_{3}&= h_{1}(x_{3},x_{2},\overline{x_{1},x_{4}})\\
\dot{x}_{4}&= f_{1}(x_{4})
\end{array}$} &
\multirow{5}{*}{$\left(\begin{array}{cccc}
0 & 0 & 0 & 0\\
e_{1} & 0 & 0 & e_{1}\\
e_{2} & e_{3} & 0 & e_{2}\\
0 & 0 & 0 & 0
\end{array}\right)$} \\
& & & \\
& & & \\
& & & \\
& & & \\
& & & \\
& & & \\
\multirow{5}{*}{$\mathcal{G}_{2}$} &
\multirow{5}{*}{\includegraphics[scale=0.2]{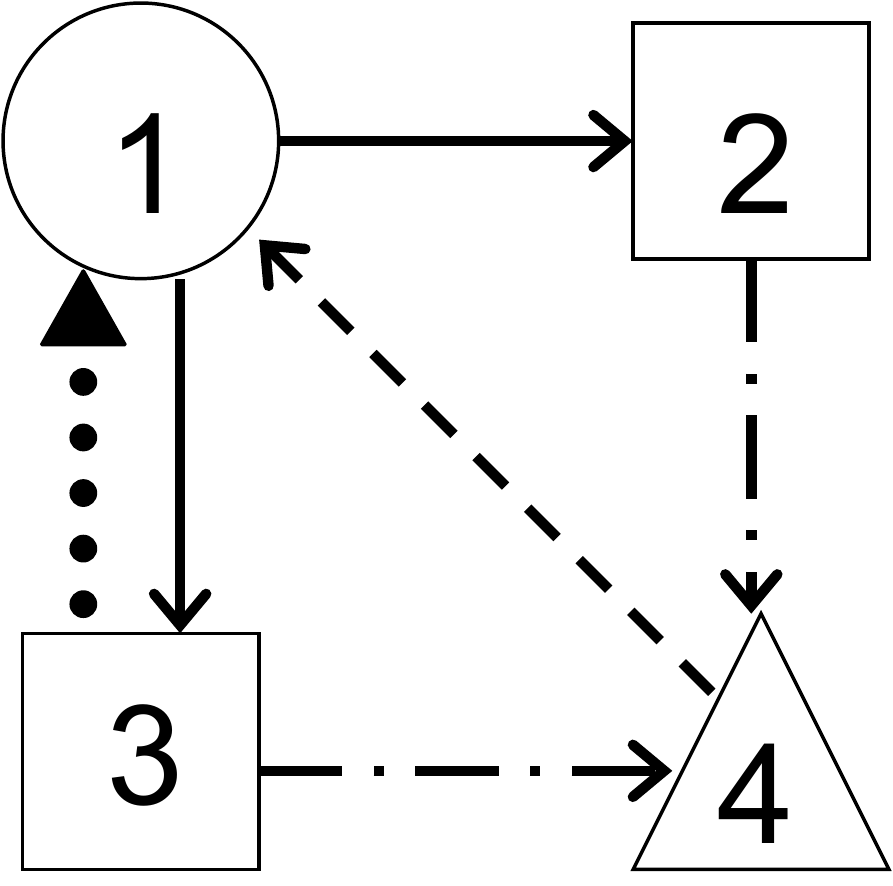}} &
\multirow{5}{*}{$\begin{array}{rl}
\dot{x}_{1}&= f_{2}(x_{1},x_{3},x_{4})\\
\dot{x}_{2}&= g_{2}(x_{2},x_{1})\\
\dot{x}_{3}&= g_{2}(x_{3},x_{1})\\
\dot{x}_{4}&= h_{2}(x_{4},\overline{x_{2},x_{3}})
\end{array}$} &
\multirow{5}{*}{$\left(\begin{array}{cccc}
0 & 0 & e_{3} & e_{2}\\
e_{1} & 0 & 0 & 0\\
e_{1} & 0 & 0 & 0\\
0 & e_{4} & e_{4} & 0
\end{array}\right)$} \\
& & & \\
& & & \\
& & & \\
& & & \\
\multirow{7}{*}{$\mathcal{G}_{3}$} &
\multirow{7}{*}{\includegraphics[scale=0.2]{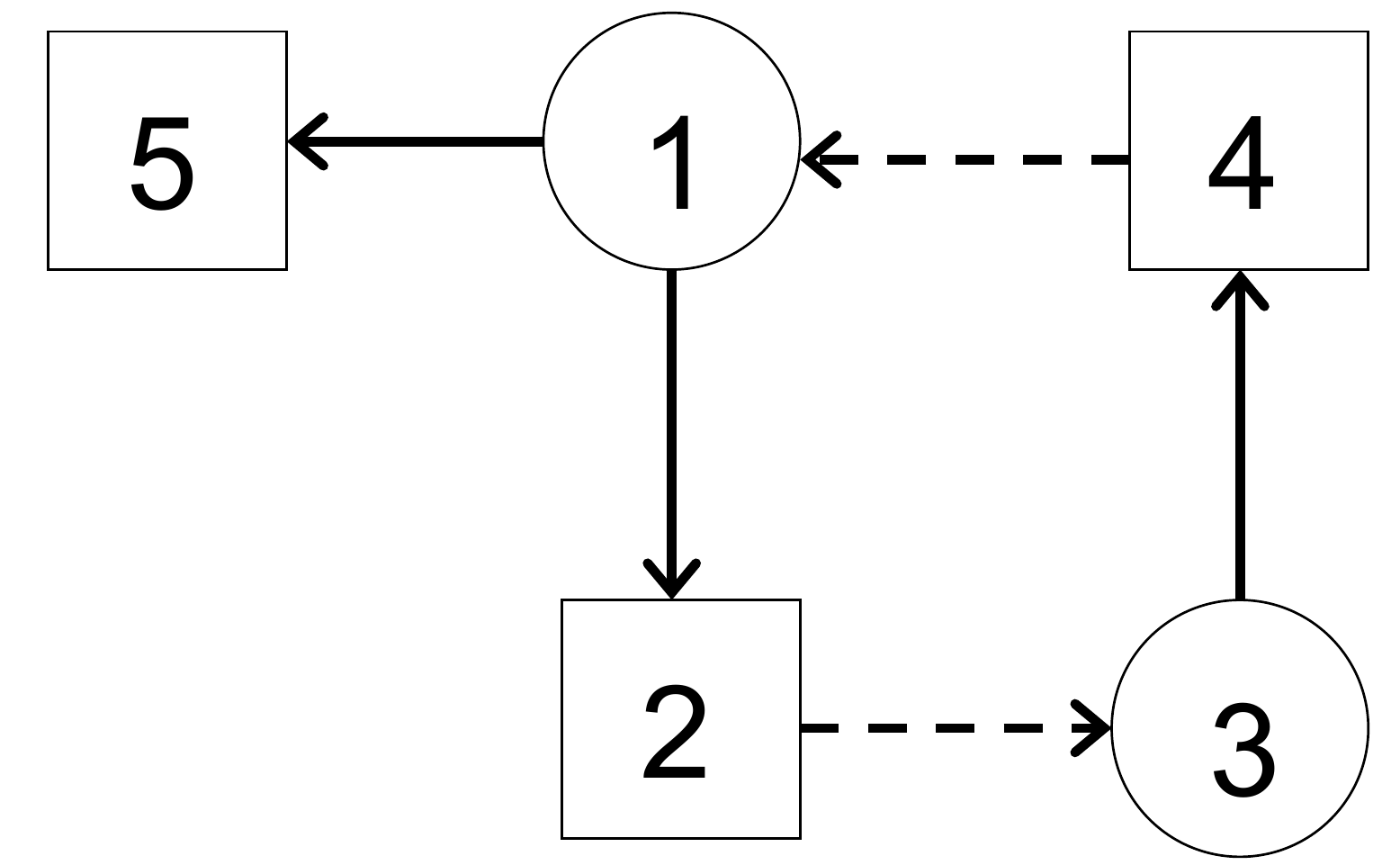}} &
\multirow{7}{*}{$\begin{array}{rl}
\dot{x}_{1}&= f_{3}(x_{1},x_{4})\\
\dot{x}_{2}&= g_{3}(x_{2},x_{1})\\
\dot{x}_{3}&= f_{3}(x_{3},x_{2})\\
\dot{x}_{4}&= g_{3}(x_{4},x_{3})\\
\dot{x}_{5}&= g_{3}(x_{5},x_{1})
\end{array}$} &
\multirow{7}{*}{$\left(\begin{array}{ccccc}
0 & 0 & 0 & e_{2} & 0\\
e_{1} & 0 & 0 & 0 & 0\\
0 & e_{2} & 0 & 0 & 0\\
0 & 0 & e_{1} & 0 & 0\\
e_{1} & 0 & 0 & 0 & 0
\end{array}\right)$} \\
& & & \\
& & & \\
& & & \\
& & &\\
\\
\\
\\
\multicolumn{4}{c}{Homogeneous networks} \\
\multirow{5}{*}{$\mathcal{G}_{4}$} &
\multirow{5}{*}{\includegraphics[scale=0.2]{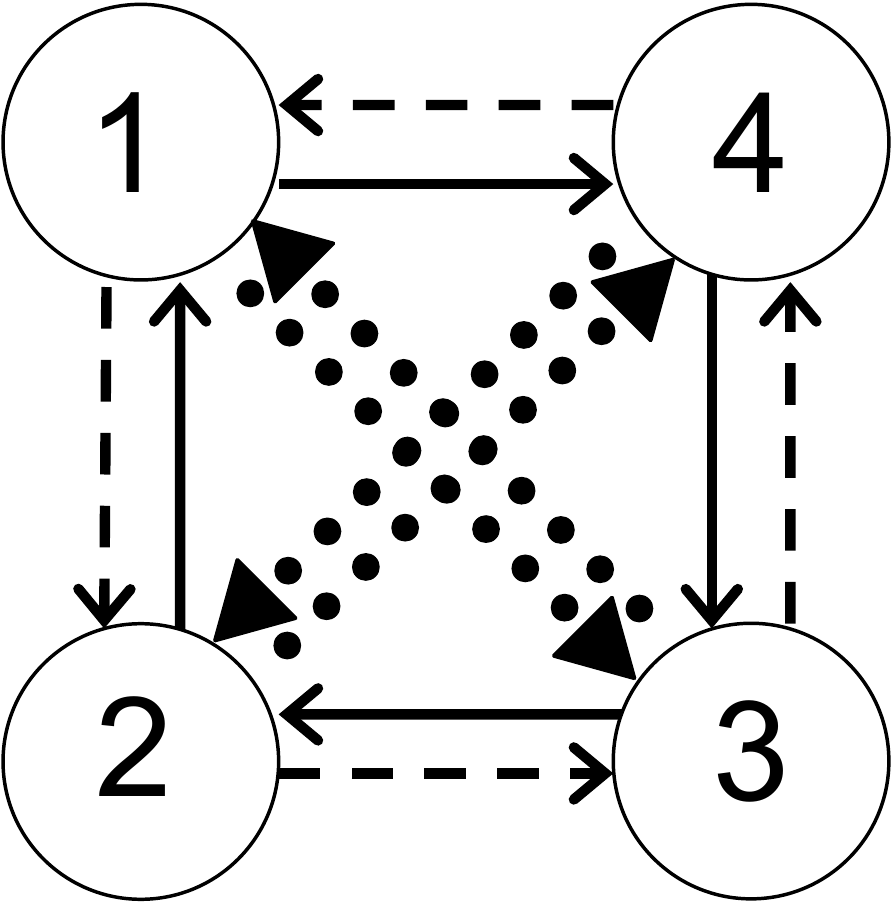}} &
\multirow{5}{*}{$\begin{array}{rl}
\dot{x}_{1}&= f_{4}(x_{1},x_{2},x_{3},x_{4})\\
\dot{x}_{2}&= f_{4}(x_{2},x_{1},x_{3},x_{4})\\
\dot{x}_{3}&= f_{4}(x_{3},x_{2},x_{1},x_{4})\\
\dot{x}_{4}&= f_{4}(x_{4},x_{2},x_{3},x_{1})
\end{array}$} &
\multirow{5}{*}{$\left(\begin{array}{cccc}
0 & e_{1} & e_{3} & e_{2}\\
e_{2} & 0 & e_{1} & e_{3}\\
e_{3} & e_{2} & 0 & e_{1}\\
e_{1} & e_{3} & e_{2} & 0
\end{array}\right)$} \\
& & & \\
& & & \\
& & & \\
& & & \\
& & & \\
\multirow{7}{*}{$\mathcal{G}_{5}$} &
\multirow{7}{*}{\includegraphics[scale=0.2]{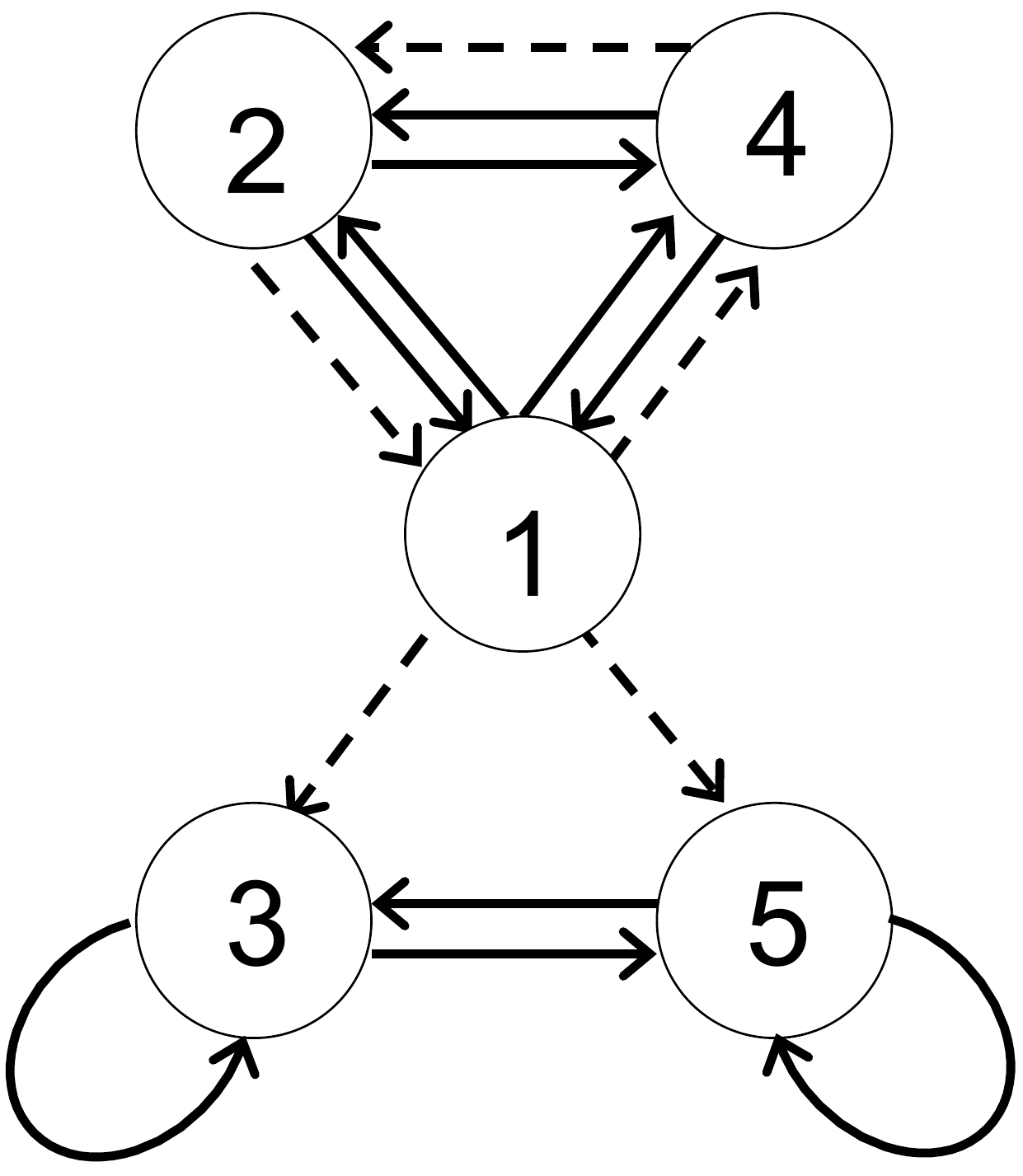}} &
\multirow{7}{*}{$\begin{array}{rl}
\dot{x}_{1}&= f_{5}(x_{1},x_{2},\overline{x_{2},x_{4}})\\
\dot{x}_{2}&= f_{5}(x_{2},x_{4},\overline{x_{1},x_{4}})\\
\dot{x}_{3}&= f_{5}(x_{3},x_{1},\overline{x_{3},x_{5}})\\
\dot{x}_{4}&= f_{5}(x_{4},x_{1},\overline{x_{1},x_{2}})\\
\dot{x}_{5}&= f_{5}(x_{5},x_{1},\overline{x_{3},x_{5}})
\end{array}$} &
\multirow{7}{*}{$\left(\begin{array}{ccccc}
0 & e_{1}+e_{2} & 0 & e_{1}   & 0 \\
e_{1} & 0   & 0 & e_{1}+e_{2} & 0 \\
e_{2} & 0   & e_{1} & 0   & e_{1} \\
e_{1}+e_{2} & e_{1} & 0 & 0 & 0 \\
e_{2} & 0   & e_{1} & 0   & e_{1} 
\end{array}\right)$} \\
& & & \\
& & & \\
& & & \\
& & &\\
\\
\\
\\
\multicolumn{4}{c}{Regular networks} \\
\multirow{7}{*}{$\mathcal{G}_{6}$} &
\multirow{7}{*}{\includegraphics[scale=0.2]{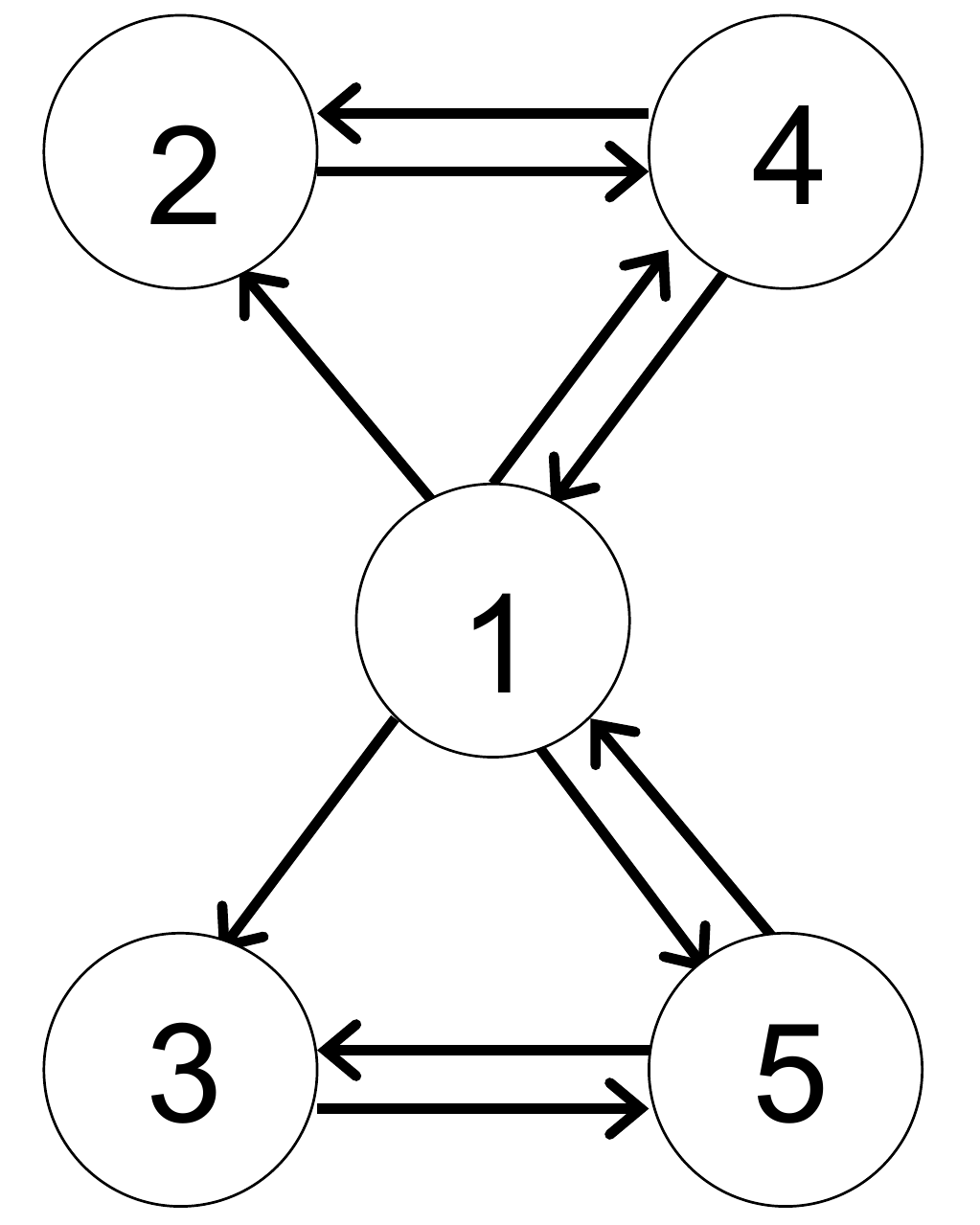}} &
\multirow{7}{*}{$\begin{array}{rl}
\dot{x}_{1}&= f_{6}(x_{1},\overline{x_{4},x_{5}})\\
\dot{x}_{2}&= f_{6}(x_{2},\overline{x_{1},x_{4}})\\
\dot{x}_{3}&= f_{6}(x_{3},\overline{x_{1},x_{5}})\\
\dot{x}_{4}&= f_{6}(x_{4},\overline{x_{1},x_{2}})\\
\dot{x}_{5}&= f_{6}(x_{5},\overline{x_{1},x_{3}})
\end{array}$} &
\multirow{7}{*}{$e_{1}\left(\begin{array}{ccccc}
0 & 0 & 0 & 1 & 1\\
1 & 0 & 0 & 1 & 0\\
1 & 0 & 0 & 0 & 1\\
1 & 1 & 0 & 0 & 0\\
1 & 0 & 1 & 0 & 0
\end{array}\right)$} \\
& & & \\
& & & \\
& & & \\
& & &\\
& & & \\
\multirow{7}{*}{$\mathcal{G}_{7}$} &
\multirow{7}{*}{\includegraphics[scale=0.2]{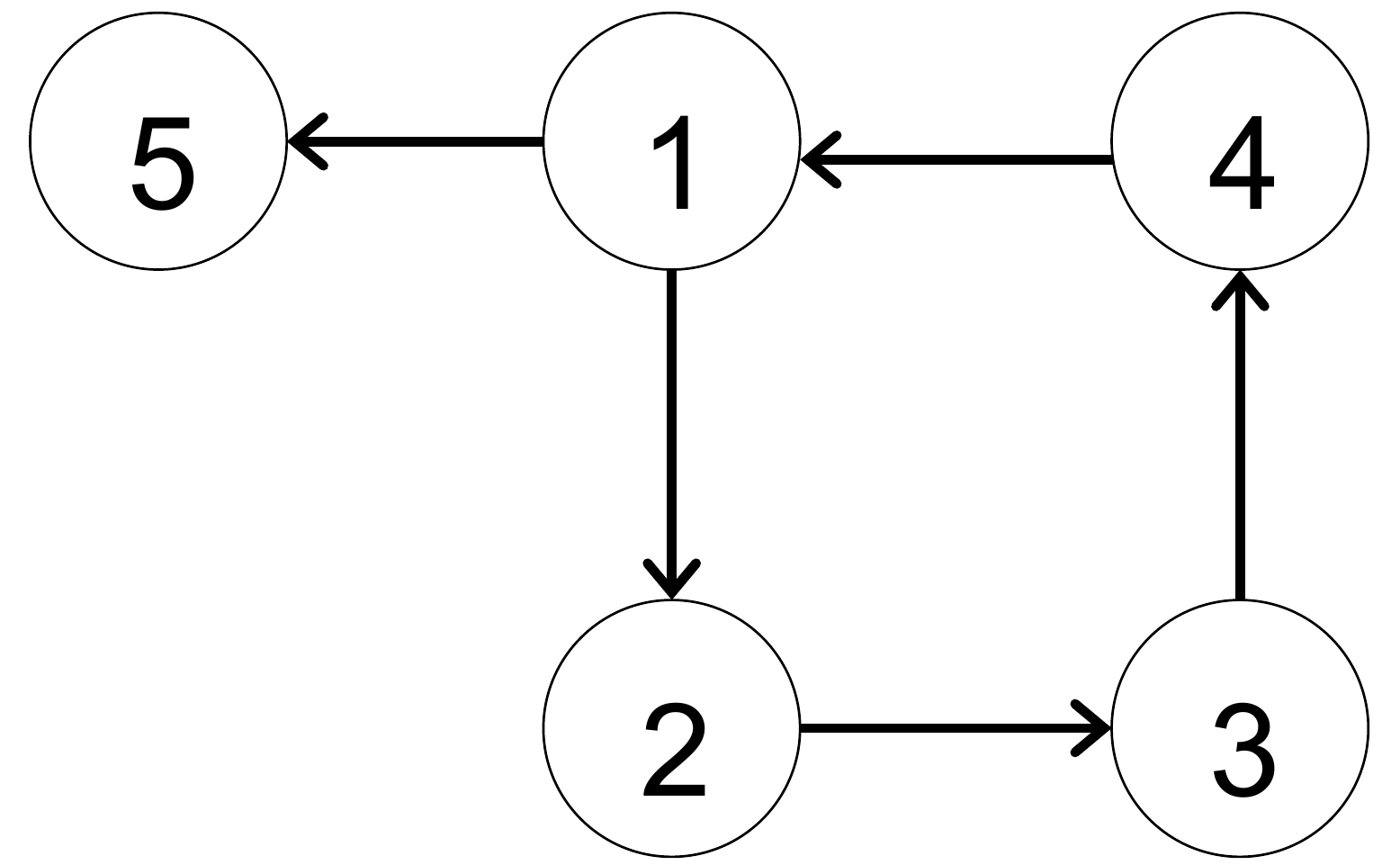}} &
\multirow{7}{*}{$\begin{array}{rl}
\dot{x}_{1}&= f_{7}(x_{1},x_{4})\\
\dot{x}_{2}&= f_{7}(x_{2},x_{1})\\
\dot{x}_{3}&= f_{7}(x_{3},x_{2})\\
\dot{x}_{4}&= f_{7}(x_{4},x_{3})\\
\dot{x}_{5}&= f_{7}(x_{5},x_{1})
\end{array}$} &
\multirow{7}{*}{$e_{1}\left(\begin{array}{ccccc}
0 & 0 & 0 & 1 & 0\\
1 & 0 & 0 & 0 & 0\\
0 & 1 & 0 & 0 & 0\\
0 & 0 & 1 & 0 & 0\\
1 & 0 & 0 & 0 & 0
\end{array}\right)$} \\
& & & \\
& & & \\
& & & \\
& & & \\
& & & \\
& & & \\
\end{tabular}
\end{center}
\caption{Coupled cell networks, and the corresponding coupled cell systems and symbolic adjacency matrices (defined in Section \ref{sec:matrix_computation}). The overline indicates that some couplings from other cells to that cell are identical, i.e., $f(x_{i},x_{j},\overline{x_{k},x_{l}})$ means $f(x_{i},x_{j},x_{k},x_{l})=f(x_{i},x_{j},x_{l},x_{k})$.}
\label{tab:coupled_cell_networks}
\end{table}

\subsection{Balanced equivalence relations and quotient networks}
We say a given coupled cell system has synchrony if (at least) two cells $c$ and $d$ have identical outputs, that is $x_{c}(t)=x_{d}(t)$ $\forall t\in\mathbb{R}$. Synchronous network dynamics solely determined by the network structure is associated with a special type of partitions of cells termed \textit{balanced equivalence relations}, which define a smaller network called the \textit{quotient network} describing synchronous dynamics of the original network.

Let $\bowtie$ be an equivalence relation on $\mathcal{C}$,
partitioning the cells into equivalence classes. For a given equivalence relation $\bowtie$, the corresponding subspace of the total phase space $P$ is defined by
$\triangle_{\bowtie}=\{x\in P: c\bowtie d\Rightarrow x_{c}=x_{d}\}$, which is called a \textit{polydiagonal subspace} of $P$.

Denote by  $\mathcal{F}_{\mathcal{G}}^{P}$ the class of admissible vector fields of a given coupled cell network $\mathcal{G}$ with the total phase space $P$. A polydiagonal subspace is called a \textit{synchrony subspace} (or \textit{balanced polydiagonal})
if it is flow-invariant for every admissible vector field with the
given network architecture. That is,
\begin{displaymath}
F(\triangle_{\bowtie})\subseteq\triangle_{\bowtie}\quad\forall
F\in\mathcal{F}_{\mathcal{G}}^{P}.
\end{displaymath}

Equivalently, if $x(t)$ is a trajectory of any $F\in\mathcal{F}_{\mathcal{G}}^{P}$, with initial condition $x(0)\in\triangle_{\bowtie}$, then $x(t)\in\triangle_{\bowtie}$
for all $t\in\mathbb{R}$. Patterns of such \textit{robust} synchrony are classified by a special type of equivalence relation defined in the following.
\begin{definition}
An equivalence relation on $\mathcal{C}$ is \textit{balanced} if
for every $c,d\in\mathcal{C}$ with $c\bowtie d$, there exists an
input isomorphism $\beta$ such that
$\mathcal{T}(i)\bowtie\mathcal{T}(\beta(i))$ for all $i\in I(c)$, where the map $\mathcal{T}(e)$ returns the tail node of an arrow $e\in\mathcal{E}$.
\end{definition}

In particular, the existence of $\beta$ implies $c\sim_{I}d$. Hence,
balanced equivalence relations can only occur between input isomorphic cells. A necessary and sufficient condition for a polydiagonal subspace to be a synchrony subspace is given by:
\begin{theorem}
\label{thm:balanced} 
An equivalence relation $\bowtie$ on $\mathcal{C}$ satisfies $F(\triangle_{\bowtie})\subseteq \triangle_{\bowtie}$ for any 
admissible vector field $F$ if and only if $\bowtie$ is balanced.
\end{theorem}
\begin{proof}
See \cite[Theorem 4.3]{Golubitsky1}.
\end{proof}

A balanced equivalence relation $\bowtie$ on a network $\mathcal{G}$ induces a unique canonical coupled cell network $\mathcal{G}/_{\bowtie}$ on
$\triangle_{\bowtie}$, called the \textit{quotient network}. The set of cells of the quotient network $\mathcal{G}/_{\bowtie}$ is defined as 
$\mathcal{C}_{\bowtie}=\{\overline{c}:c\in\mathcal{C}\}$, where $\overline{c}$ denote the $\bowtie$-equivalence class of $c\in\mathcal{C}$, and the set of arrows is defined as  $\mathcal{E}_{\bowtie}=\dot{\bigcup}_{c\in\mathcal{S}}I(c)$, where $\mathcal{S}$ is a set of cells consisting of precisely one cell $c$ from each $\bowtie$-equivalence class. 

For example, Figure~\ref{fig:balanced_quotient} shows all balanced equivalence relations of the inhomogeneous network $\mathcal{G}_{3}$ in Table~\ref{tab:coupled_cell_networks} and the corresponding quotient networks. Any dynamics on the quotient \textit{lifts} to a synchronous dynamic on the original network $\mathcal{G}$.

\begin{figure}[h]
\begin{center}
\begin{tabular}{cccc}
\includegraphics[scale=0.2]{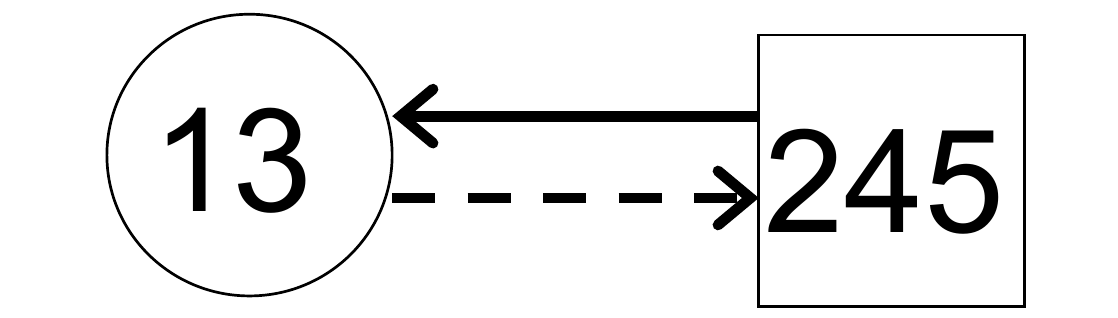} &
\includegraphics[scale=0.2]{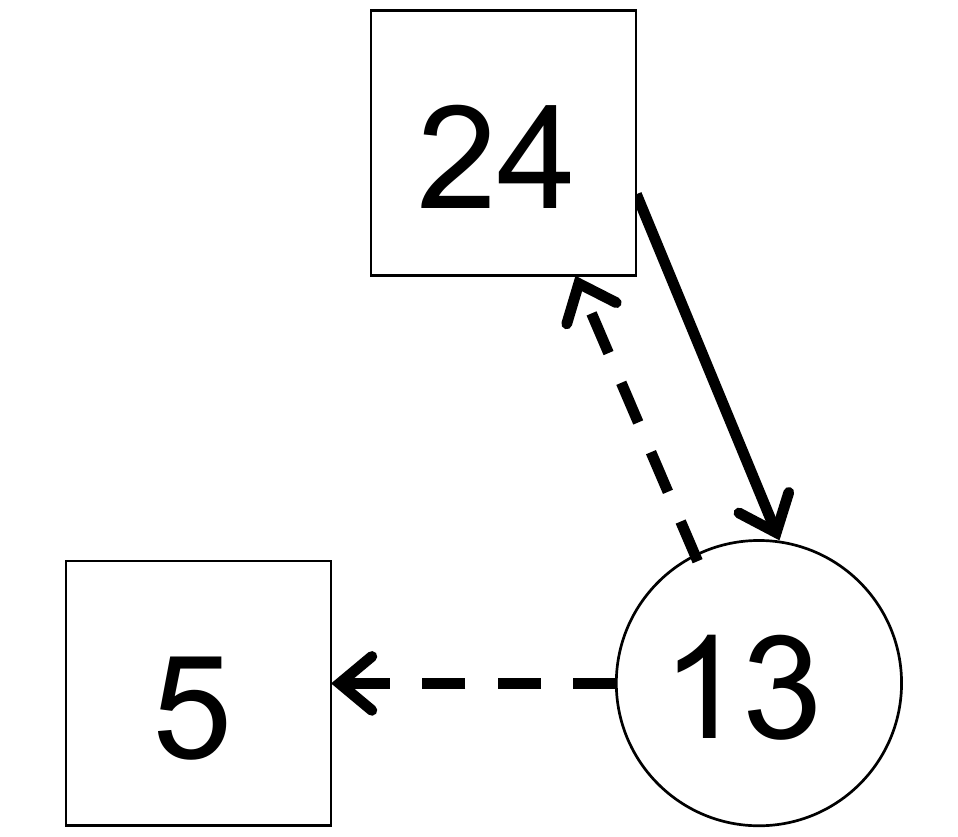} &
\includegraphics[scale=0.2]{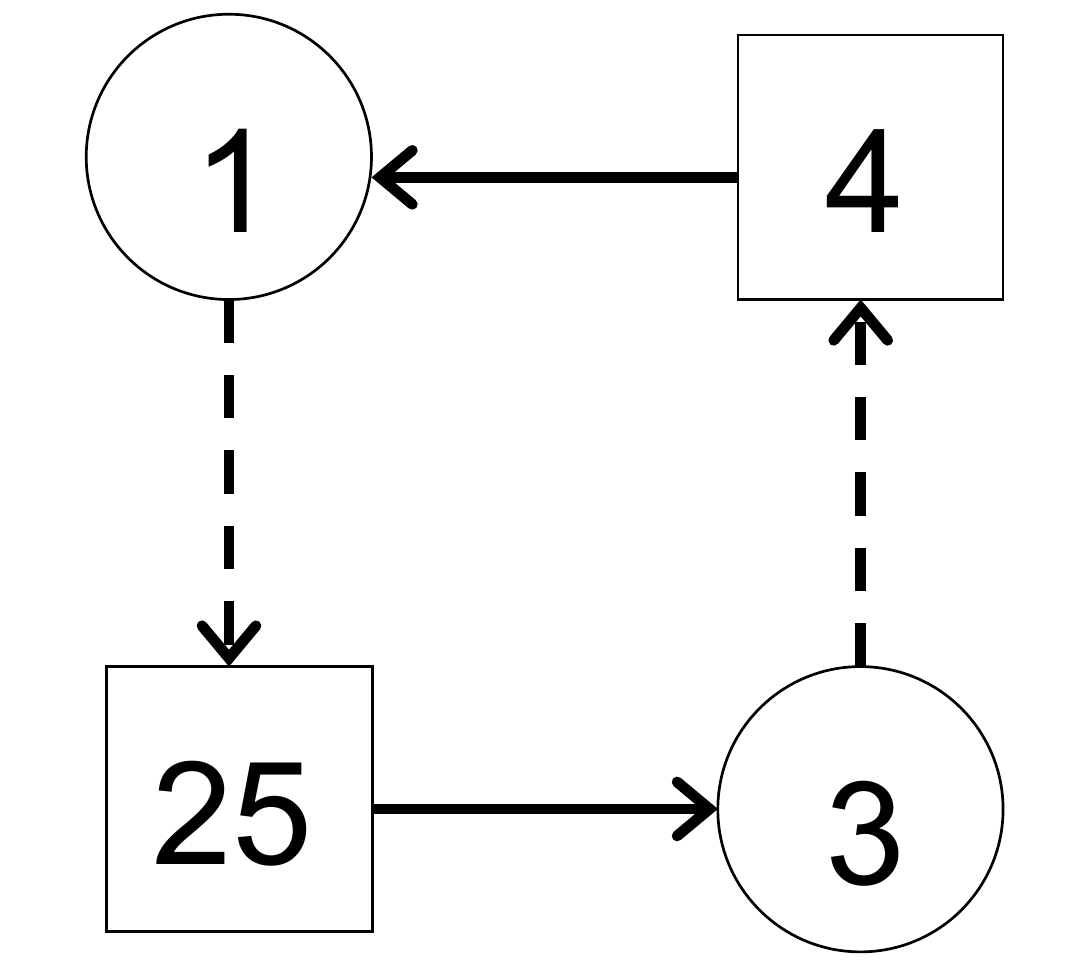} & 
\includegraphics[scale=0.2]{inhomogeneous_3.pdf}\\
$\bowtie_{1}=(13)(245)$ &
$\bowtie_{2}=(13)(24)(5)$ &
$\bowtie_{3}=(1)(25)(3)(4)$ &
$\bowtie_{4}=(1)(2)(3)(4)(5)$ \\
\end{tabular}
\end{center}
\caption{All balanced equivalence relations of the inhomogeneous network $\mathcal{G}_{3}$ (see Table~\ref{tab:coupled_cell_networks}) and the corresponding quotient networks. Note that the balanced equivalence relation $\bowtie_{4}=(1)(2)(3)(4)(5)$ is trivial and this gives the original $5$-cell network.}
\label{fig:balanced_quotient}
\end{figure}

\subsection{Lattice theory}
All possible partial synchronies (balanced equivalence relations) have a hierarchy structure represented as a \textit{complete lattice}. We recall some basic facts about lattice theory using balanced equivalence relations as an example for some concepts. See \cite{Davey} for concepts in general and for more details.

The set of all balanced equivalence relations has a partially ordered structure, using the relation of refinement. 
Let $\bowtie_{i}$ and $\bowtie_{j}$ be balanced equivalence relations on the set $\mathcal{C}$. Recall that
$\bowtie_{i}$ refines $\bowtie_{j}$, denoted by
$\bowtie_{i}\prec\bowtie_{j}$, if and only if $c\bowtie_{i}d\Rightarrow c\bowtie_{j}d$ where $c,d\in\mathcal{C}$. That is, $[c]_{i}\subseteq[c]_{j}$ where $[c]_{k}$ is the $\bowtie_{k}$-equivalence class.
The set of all balanced equivalence relations of a (locally finite) network form a complete lattice in general~\cite{Stewart,Aldis}. 

A complete lattice has a \textit{top} (maximal) element, denoted $\top$, and \textit{bottom} (minimal) element, denoted $\bot$. For example, the top element of the complete lattice of balanced equivalence relations for any $n$-cell homogeneous network is trivial and given by $\bowtie_{\top}=(12\cdots n)$ (i.e., all cells are synchronous). Aldis \cite{Aldis} and Belykh and Hasler \cite{BelykhHasler2011} give algorithms to find a nontrivial maximal balanced equivalence relation (top). For any $n$-cell coupled cell network, the bottom element is $\bowtie_{\bot}=(1)(2)\cdots(n)$ (i.e., all cells are distinct). 

The structure of a lattice can be visualised by a \textit{diagram}. Let $\bowtie_{i}$, $\bowtie_{j}$ and $\bowtie_{k}$ be distinct balanced equivalence relations. We say $\bowtie_{i}$ is covered by $\bowtie_{j}$, denoted $\bowtie_{i}<\bowtie_{j}$, if and only if $\bowtie_{i}\prec\bowtie_{j}$ and $\bowtie_{i}\prec\bowtie_{k}\prec\bowtie_{j}$ holds for no $\bowtie_{k}$. In a diagram, circles represent elements of the ordered set, and two elements $\bowtie_{i}$, $\bowtie_{j}$ are connected by a straight line if and only if one covers the other: if $\bowtie_{i}$ is covered by $\bowtie_{j}$, then the circle representing $\bowtie_{j}$ is higher than the circle representing $\bowtie_{i}$. The \textit{rank} of an equivalence relation is the number of its equivalence classes (see \cite{Kamei-part1}). Figure~\ref{fig:lattice} shows the complete lattice of the partially ordered set of all $4$ balanced equivalence relations of the coupled cell network $\mathcal{G}_{3}$, which are listed in Figure~\ref{fig:balanced_quotient}.
\begin{figure}[h]
\begin{center}
\includegraphics[scale=0.2]{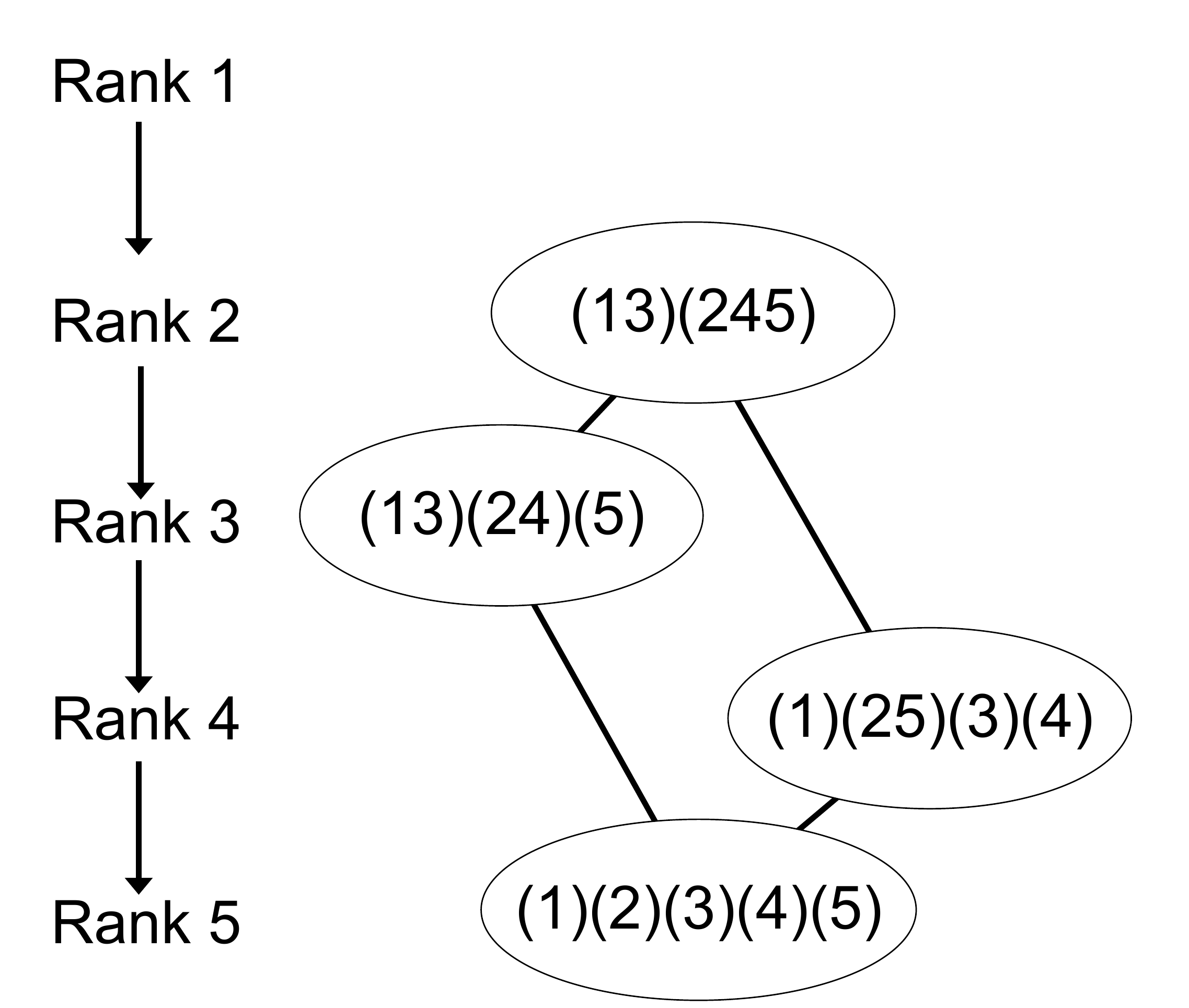}
\end{center}
\caption{The complete lattice of the partially ordered set of all possible balanced equivalence relations of the coupled cell network $\mathcal{G}_{3}$ along with their ranks (see also Figure \ref{fig:balanced_quotient}). For example, the balanced equivalence relation $\bowtie_{3}=(1)(25)(3)(4)$ has $4$ equivalence classes, hence its rank is $4$. Since this network is inhomogeneous, the top element is not of rank $1$, which requires all cells are identical.}
\label{fig:lattice}
\end{figure}

\section{Matrix computation for balanced equivalence relations}
\label{sec:matrix_computation}
We aim to determine robust patterns of synchrony of a coupled cell network solely
from the network structure, which is described by the corresponding symbolic adjacency matrix defined as follows. 
\begin{definition}
\label{def:adjacency-matrix}
Let $\mathcal{G}=(\mathcal{C},\mathcal{E},\sim_{C},\sim_{E})$ be an $n$-cell coupled cell network with $l$ cell-types and $m$ arrow-types with $[c_{1}]_{C},\ldots,[c_{l}]_{C}$, the $\sim_{C}$-equivalence classes for cells and $[e_{1}]_{E},\ldots,[e_{m}]_{E}$, the $\sim_{E}$-equivalence classes for arrows. We define the \textit{symbolic adjacency matrix} of $\mathcal{G}$ to be the $n\times n$ matrix $A=(a_{ij})$. The $(i,j)$-entry corresponds to the number of arrows of types  $[e_{1}]_{E},\ldots,[e_{m}]_{E}$ from cell $j$ to cell $i$, represented by the sum $\sum_{k=1}^{m}\beta_{k}e_{k}$, where $e_{k}$ is the type of arrow corresponding to the $[e_{k}]_{E}$-equivalence class and $\beta_{k}$ is the number of arrows of type $e_{k}$.
\end{definition}

Example coupled cell networks and their corresponding symbolic adjacency matrices are shown in Table~\ref{tab:coupled_cell_networks}.

Our main result determines all possible balanced equivalence relations combinatorially using matrix manipulations. In Lemma~\ref{lemma:block}, we show that the necessary and sufficient condition for a synchrony subspace imposes a matrix property defined as follows.

\begin{definition}
Let $B=(b_{ij})$ be a $p\times q$ symbolic matrix. We say $B$ is a homogenous block matrix if the sum $\sum_{j=1}^{q}b_{ij}$ is identical for all rows $i=1,\ldots, p$.
\end{definition}

The polydiagonal $\triangle_{\bowtie}$ is defined by a given equivalence relation $\bowtie$ which determines a unique partition of cells. We use normal form cycle notation which is obtained by writing the cell numbers $1,\ldots , n$ in increasing order in each cycle, starting with the $1$-cycle, then the $2$-cycles, and so on in increasing order of length. For example, the following polydiagonal subspace $\displaystyle \triangle_{\bowtie}=\{(x_{1},x_{2}, x_{3}, x_{4}, x_{5}, x_{6})|x_{2}=x_{4}, x_{3}=x_{5}=x_{6}\}$ corresponds to the equivalence relation $\bowtie=(1)(24)(356)$ and can be written as $\bowtie=[1^{1}2^{1}3^{1}]$. Let $A=(a_{ij})$, $i,j=1,\ldots,6$ be the adjacency matrix of a $6$-cell coupled cell network. For the above normal form cycle notation of the equivalence relation $\bowtie$, we arrange the columns and rows of the adjacency matrix of the network accordingly as follows:
\begin{displaymath}
A=\left(\begin{array}{c|cc|ccc}
a_{11} & a_{12} & a_{14} & a_{13} & a_{15} & a_{16}\\
\hline
a_{21} & a_{22} & a_{24} & a_{23} & a_{25} & a_{26}\\
a_{41} & a_{42} & a_{44} & a_{43} & a_{45} & a_{46}\\
\hline
a_{31} & a_{32} & a_{34} & a_{33} & a_{35} & a_{36}\\
a_{51} & a_{52} & a_{54} & a_{53} & a_{55} & a_{56}\\
a_{61} & a_{62} & a_{64} & a_{63} & a_{65} & a_{66}
\end{array}\right)
=
\left(\begin{array}{ccc}
A_{11} & A_{12} & A_{13} \\
A_{21} & A_{22} & A_{23} \\
A_{31} & A_{32} & A_{33} 
\end{array}\right)
\end{displaymath}
This is a block matrix, with $3 \times 3$ blocks, where $3$ is the number of equivalence classes.

More generally, let $\bowtie=[1^{\alpha_{1}}2^{\alpha_{2}}\cdots n^{\alpha_{n}}]$ and $k=\sum_{i=1}^{n}\alpha_{i}$ be the number of equivalence classes of $\bowtie$ which determines the polydiagonal $\triangle_{\bowtie}$. Interchanging the adjacency matrix rows and columns to match the permutation normal form gives a block matrix with $k \times k$ blocks.

Our main result (Theorem \ref{thm:main}) states that a polydiagonal subspace $\triangle_{\bowtie}$ is a synchrony subspace if and only if each block of the (reordered) adjacency matrix $A$ is a homogeneous block matrix. 

Alternatively, our result can be obtained by defining one (integer entry) adjacency matrix per arrow type (as in \cite{Aguiar} for regular networks with one arrow type), and finding the intersection of balanced equivalence relations for the arrow-specific adjacency matrices.

\subsection{Linear admissible vector fields}
We represent the matrix form of $\mathcal{G}$-admissible linear vector fields on an $n$-cell coupled cell network $\mathcal{G}$. 

\begin{definition}
Let $S=(s_{ij})$ and $T=(t_{ij})$ be symbolic matrices with the same size. Suppose $s_{ij}=s_{i'j'}$ for some $(i,j)$-th and $(i',j')$-th entries of $S$. If the corresponding entries in $T$  satisfy $t_{ij}=t_{i'j'}$ for all such indices $(i,j)$ and $(i',j')$, then we denote this relation by $S\sim T$.
\end{definition}

\begin{proposition}
\label{prop:higher-dimension}
Let $A=(a_{ij})$ be the $n\times n$ symbolic adjacency matrix of an $n$-cell coupled cell network $\mathcal{G}$. The $n\times n$ symbolic matrix $J=(J_{ij})$ of $\mathcal{G}$-admissible linear vector fields on $\mathcal{G}$ has the form:
\begin{equation}
\label{equ:linear-admissible}
J=D+\tilde{A},
\end{equation}
where $D=(d_{ij})$ is an $n\times n$ diagonal matrix with $d_{ii}=d_{kk}$ when $i\sim_{I} k$ for $i,k\in\mathcal{C}$, and $\tilde{A}=(\tilde{a}_{ij})$ is an $n\times n$ matrix such that $A\sim\tilde{A}$. 
\end{proposition}
\begin{proof}
Let $r_{i}$ be the dimension of internal dynamics of the $i$-th cell. Let $\bar{J}$ be the $n'\times n'$ matrix form of the $\mathcal{G}$-admissible linear vector field of the network $\mathcal{G}$ with total phase space $\displaystyle P'=\mathbb{R}^{n'}$ where $\displaystyle n'=\sum_{i=1}^{n}r_{i}$.

$\bar{J}$ can be described as a block matrix $\bar{J}=(\bar{J}_{ij})$, $i,j=1,\ldots,n$, and each block of $\bar{J}_{ij}$ is a $k_{i} \times k_{j}$ matrix with real entries. We can decompose $\bar{J}$ as
\begin{displaymath}
\bar{J}=\bar{D}+(-\bar{D}+\bar{J}).
\end{displaymath}

Each block of the diagonal matrix $\bar{D}_{ii}$ satisfies $D_{ii}=D_{kk}$ when $i\sim_{I}k$ for $i,k\in\mathcal{C}$. By representing each block $\bar{D}_{ii}$ with a symbol $d_{ii}$, we obtain an $n\times n$ symbolic diagonal matrix $D$. Similarly, by representing each block $\bar{J}_{ij}$ of $-\bar{D}+\bar{J}$ with a symbol $\tilde{a}_{ij}$, we obtain an $n\times n$ matrix $\tilde{A}$, which satisfies $A\sim\tilde{A}$ by the admissibility. Therefore, the $n\times n$ symbolic matrix $J=(J_{ij})$ of $\mathcal{G}$-admissible linear vector fields on an $n$-cell coupled cell network $\mathcal{G}$ has the form $J=D+\tilde{A}$.
\end{proof}

An equivalence relation $\bowtie$ is balanced if and only if the $\mathcal{G}$-admissible linear vector field $J$ satisfies $J(\triangle_{\bowtie})\subseteq\triangle_{\bowtie}$ \citep[Theorem 4.3]{Golubitsky1}. This fact leads to a generalization of Proposition $4.1.$ in \cite{Kamei-part1}, which consists of only one type of cell and one type of coupling. 

\begin{proposition}
\label{prop:linear-admissible-polysynchronous} Let $\mathcal{G}$ be a coupled cell network associated with the 
admissible vector field $F$. Let $A$ be the
adjacency matrix of $\mathcal{G}$. $\bowtie$ is balanced if and only if
$A(\triangle_{\bowtie})\subseteq \triangle_{\bowtie}$ where $\triangle_{\bowtie}$ is a synchrony subspace associated with $\bowtie$.
\end{proposition}
\begin{proof}
Using Equation~\ref{equ:linear-admissible}, the result follows immediately:
\begin{eqnarray*}
&                 & J(\triangle_{\bowtie})\subseteq\triangle_{\bowtie}\\
& \Leftrightarrow & \{D+\tilde{A}\}(\triangle_{\bowtie})\subseteq\triangle_{\bowtie}\\
& \Leftrightarrow & \{D(\triangle_{\bowtie})+\tilde{A}(\triangle_{\bowtie})\}\subseteq\triangle_{\bowtie}\\
& \Leftrightarrow & \tilde{A}(\triangle_{\bowtie})\subseteq\triangle_{\bowtie}\quad\textrm{since }\bowtie\textrm{ refines }\sim_{I}\textrm{, and therefore } D(\triangle_{\bowtie})\subseteq\triangle_{\bowtie}\\
& \Leftrightarrow & A(\triangle_{\bowtie})\subseteq\triangle_{\bowtie}\quad\textrm{since}\quad A\sim\tilde{A}.
\end{eqnarray*}
\end{proof}

Now, we aim to determine all possible invariant polydiagonals under a given adjacency matrix using a projection map. Since the determination of synchrony subspaces does not depend on the size of internal dynamics of the cells, without loss of generality, we assume the total phase space $P=\mathbb{R}^{n}$ for an $n$-cell coupled cell network for the remaining arguments.

\subsection{Projection onto a polydiagonal}
We construct a projection map on a given polydiagonal as follows.

Let $\triangle_{\bowtie}\subseteq \mathbb{R}^{n}$ be a polydiagonal subspace of $\mathbb{R}^{n}$, and $\triangle_{\bowtie}'$ denote its complement.
We define the projection map $P_{\bowtie}$ of $\mathbb{R}^{n}$ on $\triangle_{\bowtie}$ along $\triangle_{\bowtie}'$ by $\textrm{Im}(P_{\bowtie})=\triangle_{\bowtie}$ and $\textrm{Ker}(P_{\bowtie})=\triangle_{\bowtie}'$.

Let $\triangle_{\bowtie}$ be the polydiagonal determined by a given equivalence relation $\bowtie$. We use normal form cycle notation for an equivalence relation. For example, the equivalence relation $\bowtie$ corresponding to a polydiagonal subspace $\triangle_{\bowtie}=\{(x_{1},x_{2}, x_{3})|x_{2}=x_{3}\}$ is written as a product of disjoint cycles $\bowtie=(1)(23)$ in normal form, and also written as $\bowtie=[1^{1}2^{1}]$.

Now we define a map $\pi$ which maps each element to the first element of the cycle that they belong to when written in normal form. For example, the elements in the above product of disjoint cycles are mapped to $1\rightarrow 1$, $2\rightarrow 2$, and $3\rightarrow 2$ by a map $\pi$. We define the corresponding projection matrix $P_{\bowtie}=(p_{ij})$ on $\triangle_{\bowtie}$, which is written in normal form of a
partition using the map $\pi$ as follows:
\begin{displaymath}
p_{i,\pi(i)}=1
\end{displaymath}
with all other entries being $0$.

With the elements of $P_{\bowtie}$ defined in this way, the projection matrix has a block diagonal form:
\begin{displaymath}
P_{\bowtie}=\left( \begin{array}{cccc} P_{1} & 0 & \cdots & 0 \\
0 & P_{2} & \cdots & 0 \\
\vdots & \vdots & \ddots & \vdots \\
 0 & 0 & \cdots & P_{k}
\end{array} \right)
\end{displaymath}
where $k$ is the number of disjoint cycles and $P_{i}$,
$i=1,\cdots , k$ is a $t_i \times t_i$ square projection matrix
on $\triangle=\{(x_{1}, \ldots , x_{t_i})|x_{1}=\cdots
=x_{t_i}\}$ with $\textrm{rank}(P_{i})=1$ $\forall i$ and off-diagonal blocks are zero matrices.

\begin{lemma}
\label{lemma:A-invariant-subspaces} Let $P_{\bowtie}$ and $A$ be
linear mappings of $\mathbb{R}^{n}$ and let
$\mathbb{R}^{n}=\triangle_{\bowtie}\oplus\triangle_{\bowtie}'$.
$\triangle_{\bowtie}$ is $A$-invariant if and only if
$P_{\bowtie}AP_{\bowtie}=AP_{\bowtie}$, where $P_{\bowtie}$ is
the projection on $\triangle_{\bowtie}$ along
$\triangle_{\bowtie}'$ and $A$ is the adjacency matrix of a given
coupled cell network.
\end{lemma}
\begin{proof}
This result is well known, see for example~\cite{Lancaster} and~\cite{Kamei-part1}.
\end{proof}

For the rest of arguments, we arrange the columns and rows of the adjacency matrix of the network according to the normal form of a given equivalence relation $\bowtie$.

\begin{proposition}
\label{prop:papap} $\triangle_{\bowtie}$ is a synchrony subspace of a coupled cell network $\mathcal{G}$ if and only if
$P_{\bowtie}AP_{\bowtie}=AP_{\bowtie}$, where $P_{\bowtie}$ is the projection on $\triangle_{\bowtie}$ along
$\triangle_{\bowtie}'$ and $A$ is the adjacency matrix of $\mathcal{G}$.
\end{proposition}
\begin{proof}
If $\triangle_{\bowtie}$ is a synchrony subspace, then the corresponding equivalence relation $\bowtie$ is balanced. Hence, by Proposition~\ref{prop:linear-admissible-polysynchronous}, $\triangle_{\bowtie}$ is $A$-invariant and from Lemma~\ref{lemma:A-invariant-subspaces}, the corresponding projection map $P_{\bowtie}$ to $\triangle_{\bowtie}$ satisfies $P_{\bowtie}AP_{\bowtie}=AP_{\bowtie}$. Conversely, if
$P_{\bowtie}AP_{\bowtie}=AP_{\bowtie}$, then the corresponding subspace $\triangle_{\bowtie}$ is $A$-invariant by Lemma~\ref{lemma:A-invariant-subspaces} and Proposition~\ref{prop:linear-admissible-polysynchronous}, thus the corresponding polydiagonal is balanced.
\end{proof}

\subsection{Block structure of an adjacency matrix}
We show that the necessary and sufficient condition for a synchrony subspace in Proposition~\ref{prop:papap} imposes a block structure on the adjacency matrix. 

\begin{lemma}
\label{lemma:block} Let $A$ be the $n\times n$ adjacency matrix
of a given coupled cell network $\mathcal{G}$. Suppose a
synchrony subspace $\triangle_{\bowtie}$ is defined by a
partition $[1^{\alpha_{1}}2^{\alpha_{2}}\cdots n^{\alpha_{n}}]$
and the corresponding projection matrix $P_{\bowtie}$ is a block
diagonal matrix whose diagonal blocks $P_{i}$, $i=1, \ldots , k$,
where $k=\sum_{i=1}^{n}\alpha_{i}$ are projection matrices on
diagonal subspaces $\triangle$ in the corresponding dimensions.
Then $P_{\bowtie}AP_{\bowtie}=AP_{\bowtie}$ if and only if
corresponding blocks of $A$ to $P_{\bowtie}$ satisfy the
following condition:
\begin{itemize}
\item Let $A_{st}$, where $s, t=1,\ldots, k$ be blocks of $A$. For all blocks $A_{st}$, the sum of each row is identical.
\end{itemize}
\end{lemma}
\begin{proof}
Since $P_{\bowtie}=\left( \begin{array}{ccc} P_{1} &\cdots  & 0 \\
\vdots & \ddots & \vdots \\
0 & \cdots & P_{k}
\end{array} \right)$, we obtain:
\begin{eqnarray*}
AP_{\bowtie} & = & \left( \begin{array}{ccc} A_{11} & \cdots & A_{1k} \\
\vdots & \ddots & \vdots\\
A_{k1} & \cdots & A_{kk}
\end{array} \right)\left( \begin{array}{ccc} P_{1} &\cdots  & 0 \\
\vdots & \ddots &\vdots \\
0 &\cdots  & P_{k}
\end{array} \right)\\
& = & \left( \begin{array}{cccc} A_{11}P_{1} & A_{12}P_{2} & \cdots & A_{1k}P_{k} \\
A_{21}P_{1} & A_{22}P_{2} & \cdots & A_{2k}P_{k}\\
\vdots & \vdots &  \ddots & \vdots \\
A_{k1}P_{1} & A_{k2}P_{2} & \cdots & A_{kk}P_{k}
\end{array} \right) \\
P_{\bowtie}AP_{\bowtie} & = & \left( \begin{array}{ccc} P_{1} & \cdots & 0 \\
\vdots & \ddots &\vdots \\
0 & \cdots & P_{k}
\end{array} \right)\left( \begin{array}{cccc} A_{11}P_{1} & A_{12}P_{2} & \cdots & A_{1k}P_{k} \\
A_{21}P_{1} & A_{22}P_{2} & \cdots & A_{2k}P_{k}\\
\vdots & \vdots &  \ddots & \vdots \\
A_{k1}P_{1} & A_{k2}P_{2} & \cdots & A_{kk}P_{k}
\end{array} \right) \\
& = & \left( \begin{array}{cccc} P_{1}A_{11}P_{1} & P_{1}A_{12}P_{2} & \cdots & P_{1}A_{1k}P_{k} \\
P_{2}A_{21}P_{1} & P_{2}A_{22}P_{2} & \cdots & P_{2}A_{2k}P_{k}\\
\vdots & \vdots &  \ddots & \vdots \\
P_{k}A_{k1}P_{1} & P_{k}A_{k2}P_{2} & \cdots & P_{k}A_{kk}P_{k}
\end{array} \right) 
\end{eqnarray*}
Hence, $P_{\bowtie}AP_{\bowtie}=AP_{\bowtie}\Longleftrightarrow P_{s}A_{st}P_{t}=A_{st}P_{t}$, for all $s,t=1, \ldots , k$.

Let $A_{st}=(a^{st})_{ij}$ be an arbitrary $l\times m$ block matrix. Then $P_{t}$ is a $m\times m$ square matrix and $P_{s}$
is a $l\times l$ square matrix. Since $P_{t}$ and $P_{s}$ are projection matrices onto $m$-dimensional and $l$-dimensional diagonals, respectively,

\begin{eqnarray*}
A_{st}P_{t} & = & \left( \begin{array}{ccc} a^{st}_{11} & \cdots & a^{st}_{1m} \\
\vdots & \ddots & \vdots\\
a^{st}_{l1} & \cdots & a^{st}_{lm}
\end{array} \right)\left( \begin{array}{cccc} 1 & 0 & \cdots & 0 \\
1 & 0 & \cdots & 0 \\
\vdots & \vdots & \ddots & \vdots \\
1 & 0 & \cdots & 0
\end{array} \right)\\
& = & \left( \begin{array}{cccc} a^{st}_{11}+\cdots +a^{st}_{1m} & 0 & \cdots & 0 \\
a^{st}_{21}+\cdots +a^{st}_{2m} & 0 & \cdots & 0 \\
\vdots & \vdots &  \ddots & \vdots \\
a^{st}_{l1}+\cdots +a^{st}_{lm} & 0 & \cdots & 0
\end{array} \right) \\
P_{s}A_{st}P_{t} & = & \left( \begin{array}{cccc} 1 & 0 & \cdots & 0 \\
1 & 0 & \cdots & 0 \\
\vdots & \vdots & \ddots & \vdots \\
1 & 0 & \cdots & 0
\end{array} \right)\left( \begin{array}{cccc} a^{st}_{11}+\cdots +a^{st}_{1m} & 0 & \cdots & 0 \\
a^{st}_{21}+\cdots +a^{st}_{2m} & 0 & \cdots & 0 \\
\vdots & \vdots &  \ddots & \vdots \\
a^{st}_{l1}+\cdots +a^{st}_{lm} & 0 & \cdots & 0
\end{array} \right) \\
& = & \left( \begin{array}{cccc} a^{st}_{11}+\cdots +a^{st}_{1m} & 0 & \cdots & 0 \\
a^{st}_{11}+\cdots +a^{st}_{1m} & 0 & \cdots & 0 \\
\vdots & \vdots &  \ddots & \vdots \\
a^{st}_{11}+\cdots +a^{st}_{1m} & 0 & \cdots & 0
\end{array} \right) 
\end{eqnarray*}

Thus $P_{s}A_{st}P_{t}=A_{st}P_{t}$ if and only if $\sum_{j=1}^{m}a^{st}_{ij}$ is identical for all $i=1,\ldots , l$.

Therefore, $P_{s}A_{st}P_{t}=A_{st}P_{t}$ for all $s,t=1,\ldots,k$ if and only if, for all blocks $A_{st}$, the sum of each row is identical.  
\end{proof}

We call a block of this form a \textit{homogeneous block matrix}.
It now follows immediately that:
\begin{theorem}
\label{thm:main}
A polydiagonal subspace $\triangle_{\bowtie}$ is a synchrony subspace if and only if each block of the adjacency matrix $A$, which corresponds to a block of
$P_{\bowtie}$, is a homogeneous block matrix.
\end{theorem}

\begin{proof}
Each block of $A$ is a homogeneous block matrix\\
$\Longleftrightarrow\quad P_{\bowtie}AP_{\bowtie}=AP_{\bowtie}$
(by Lemma~\ref{lemma:block})\\
$\Longleftrightarrow\quad \triangle_{\bowtie}\quad\textrm{is a
synchrony subspace}$ (by Proposition~\ref{prop:papap}).
\end{proof}

\begin{example}
The projection mapping on
$\triangle_{\bowtie}=\{(x_{1},x_{2}, x_{3})|x_{2}=x_{3}\}$ has the form:
\begin{displaymath}
P_{\bowtie}=\left( \begin{array}{c|cc} 1 & 0 & 0 \\
\hline
0 & 1 & 0 \\
0 & 1 & 0
\end{array} \right)
\end{displaymath}

Let $A$ be the adjacency matrix of a $3$-cell coupled cell network of the form:
\begin{displaymath}
A=\left( \begin{array}{c|cc} a_{11} & a_{12} & a_{13} \\
\hline
a_{21} & a_{22} & a_{23} \\
a_{31} & a_{32} & a_{33}
\end{array} \right)
\end{displaymath}
Then,
\begin{eqnarray*}
P_{\bowtie}AP_{\bowtie} & = & \left( \begin{array}{ccc} 1 & 0 & 0 \\
0 & 1 & 0 \\
0 & 1 & 0
\end{array} \right)\left( \begin{array}{ccc} a_{11} & a_{12} & a_{13} \\
a_{21} & a_{22} & a_{23} \\
a_{31} & a_{32} & a_{33}
\end{array} \right)\left( \begin{array}{ccc} 1 & 0 & 0 \\
0 & 1 & 0 \\
0 & 1 & 0
\end{array} \right)\\
& = & \left( \begin{array}{ccc} 1 & 0 & 0 \\
0 & 1 & 0 \\
0 & 1 & 0
\end{array} \right)\left( \begin{array}{ccc} a_{11} & a_{12}+a_{13} & 0 \\
a_{21} & a_{22}+a_{23} & 0 \\
a_{31} & a_{32}+a_{33} & 0
\end{array} \right) \\
& = & \left( \begin{array}{ccc} a_{11} & a_{12}+a_{13} & 0 \\
a_{21} & a_{22}+a_{23} & 0 \\
a_{21} & a_{22}+a_{23} & 0
\end{array} \right)
\end{eqnarray*}
Therefore,
\begin{displaymath}
P_{\bowtie}AP_{\bowtie}=AP_{\bowtie}\quad\Longleftrightarrow\quad
a_{21}=a_{31}\quad\textrm{and}\quad a_{22}+a_{23}=a_{32}+a_{33}.
\end{displaymath}
Hence, $\triangle_{\bowtie}$ is a synchrony subspace if and only
if $A$ has the following block structure:
\begin{displaymath}
A=\left( \begin{array}{c|cc} a_{11} & a_{12} & a_{13} \\
\hline
a_{21} & a_{22} & a_{23} \\
a_{21} & a_{32} & a_{33}
\end{array} \right)
\end{displaymath}
where $a_{22}+a_{23}=a_{32}+a_{33}$.

\end{example}

Next we derive the adjacency matrix of the quotient network which is
defined as the adjacency matrix $A$ restricted on a synchrony subspace $\triangle_{\bowtie}$.

\begin{corollary}
Let $\triangle_{\bowtie}$ be a synchrony subspace
defined by a partition $[1^{\alpha_{1}}2^{\alpha_{2}}\cdots n^{\alpha_{n}}]$ with $k=\alpha_{1}+\cdots+\alpha_{n}$
equivalence classes, and $P_{\bowtie}$ be the corresponding block
projection matrix. Let $A_{st}$ for $s,t=1,\ldots,k$
corresponding to the blocks of $P_{\bowtie}$ be blocks of an
$n\times n$ adjacency matrix $A$. If blocks $A_{st}=(a^{st})_{ij}$ are
homogeneous block matrices such that:
\begin{displaymath}
A=\left( \begin{array}{c|c|c||c|c|c||c||c|c|c} A_{11} & \cdots &
A_{1\alpha_{1}} & A_{1(\alpha_{1}+1)} & \cdots &
A_{1(\alpha_{1}+\alpha_{2})} & \cdots &
A_{1(\alpha_{1}+\cdots+\alpha_{n-1}+1)} & \cdots &
A_{1k}\\
\hline A_{21} & \cdots & A_{2\alpha_{1}} & A_{2(\alpha_{1}+1)} &
\cdots & A_{2(\alpha_{1}+\alpha_{2})} & \cdots &
A_{2(\alpha_{1}+\cdots +\alpha_{n-1}+1)} & \cdots &
A_{2k} \\
\hline \vdots & \ddots & \vdots & \vdots & \ddots & \vdots &
\ddots & \vdots & \ddots & \vdots \\
\hline A_{k1} & \cdots & A_{k\alpha_{1}} & A_{k(\alpha_{1}+1)} &
\cdots & A_{k(\alpha_{1}+\alpha_{2})} & \cdots &
A_{k(\alpha_{1}+\cdots+\alpha_{n-1}+1)} & \cdots & A_{kk}
\end{array} \right)
\end{displaymath}

then the quotient network corresponding to $\bowtie$ has a
$k\times k$ adjacency matrix $A|_{\triangle_{\bowtie}}$, denoted
by $A_{\bowtie}$, of the form:
\begin{displaymath}
A_{\bowtie}=\left( \begin{array}{c|c|c||c||c|c|c}
\sum_{j=1}^{1}a^{11}_{1j} & \cdots &
\sum_{j=1}^{1}a^{1\alpha_{1}}_{1j}
& \cdots & \sum_{j=1}^{n}a^{1(\alpha_{1}+\cdots +\alpha_{n-1}+1)}_{1j} & \cdots & \sum_{j=1}^{n}a^{1k}_{1j}\\
\hline \sum_{j=1}^{1}a^{21}_{1j} & \cdots &
\sum_{j=1}^{1}a^{2\alpha_{1}}_{1j}
& \cdots & \sum_{j=1}^{n}a^{2(\alpha_{1}+\cdots +\alpha_{n-1}+1)}_{1j} & \cdots & \sum_{j=1}^{n}a^{2k}_{1j}\\
\hline \vdots & \ddots & \vdots
& \ddots & \vdots & \ddots & \vdots \\
\hline \sum_{j=1}^{1}a^{k1}_{1j} & \cdots &
\sum_{j=1}^{1}a^{k\alpha_{1}}_{1j}
& \cdots &
\sum_{j=1}^{n}a^{k(\alpha_{1}+\cdots+\alpha_{n-1}+1)}_{1j} &
\cdots & \sum_{j=1}^{n}a^{kk}_{1j}
\end{array} \right)
\end{displaymath}
\end{corollary}
\begin{proof}
Let $\{v_{1}, \ldots , v_{\alpha_{1}}, v_{\alpha_{1}+1}, \ldots ,
v_{k}\}$ be a basis of a synchrony subspace. Each
basis element corresponds to a conjugacy class of the partition
$\bowtie$ and is an $n\times 1$ vector. Therefore, the basis
elements have the following forms:
\begin{eqnarray*}
v_{1} & = & [1,0,   \ldots    , 0]^{t}\\
 & \vdots &  \\
v_{\alpha_{1}} & = & [\underbrace{0,\ldots ,0}_{\alpha_{1}-1}, 1, 0,  \ldots  , 0]^{t}\\
v_{\alpha_{1}+1} & = & [\underbrace{0,\ldots ,0}_{\alpha_{1}}, \underbrace{1, 1}_{2}, 0, \ldots , 0]^{t}\\
 & \vdots &  \\
v_{k} & = & [\underbrace{0,\ldots ,0}_{\sum_{i=1}^{n-1}i\alpha_{i}}, \underbrace{1,   \ldots   , 1}_{n}]^{t}\\
\end{eqnarray*}

Since each block $A_{st}$ is a homogeneous block matrix, i.e.,
the sum of each row is identical, we can express the image of each
basis element using a linear combination of a basis with the sum
of the first row of each $A_{st}$ being used as a coefficient such as:

\begin{eqnarray*}
Av_{1} & = & \left( \begin{array}{c} \sum_{j=1}^{1}a^{11}_{1j} \\
\sum_{j=1}^{1}a^{21}_{1j} \\
\vdots \\
\sum_{j=1}^{1}a^{k1}_{1j}
\end{array} \right)\\
& = &
\sum_{j=1}^{1}a^{11}_{1j}v_{1}+\sum_{j=1}^{1}a^{21}_{1j}v_{2}+\cdots+\sum_{j=1}^{1}a^{k1}_{1j}v_{k}\\
Av_{2} & = & \left( \begin{array}{c} \sum_{j=1}^{1}a^{12}_{1j} \\
\sum_{j=1}^{1}a^{22}_{1j} \\
\vdots \\
\sum_{j=1}^{1}a^{k2}_{1j}
\end{array} \right)\\
& = &
\sum_{j=1}^{1}a^{12}_{1j}v_{1}+\sum_{j=1}^{1}a^{22}_{1j}v_{2}+\cdots+\sum_{j=1}^{1}a^{k2}_{1j}v_{k}\\
 & \vdots &  \\
Av_{k} & = & \left( \begin{array}{c} \sum_{j=1}^{n}a^{1k}_{1j} \\
\sum_{j=1}^{n}a^{2k}_{1j} \\
\vdots \\
\sum_{j=1}^{n}a^{kk}_{1j}
\end{array} \right)\\
& = &
\sum_{j=1}^{n}a^{1k}_{1j}v_{1}+\sum_{j=1}^{n}a^{2k}_{1j}v_{2}+\cdots+\sum_{j=1}^{n}a^{kk}_{1j}v_{k}\\
\end{eqnarray*}

Therefore a $k\times k$ matrix $A_{\bowtie}$, which is the adjacency matrix $A$ restricted on $\triangle_{\bowtie}$, is written as

\begin{displaymath}
A_{\bowtie}=\left( \begin{array}{c|c|c||c||c|c|c}
\sum_{j=1}^{1}a^{11}_{1j} & \cdots &
\sum_{j=1}^{1}a^{1\alpha_{1}}_{1j}
& \cdots & \sum_{j=1}^{n}a^{1(\alpha_{1}+\cdots +\alpha_{n-1}+1)}_{1j} & \cdots & \sum_{j=1}^{n}a^{1k}_{1j}\\
\hline \sum_{j=1}^{1}a^{21}_{1j} & \cdots &
\sum_{j=1}^{1}a^{2\alpha_{1}}_{1j}
& \cdots & \sum_{j=1}^{n}a^{2(\alpha_{1}+\cdots +\alpha_{n-1}+1)}_{1j} & \cdots & \sum_{j=1}^{n}a^{2k}_{1j}\\
\hline \vdots & \ddots & \vdots
& \ddots & \vdots & \ddots & \vdots \\
\hline \sum_{j=1}^{1}a^{k1}_{1j} & \cdots &
\sum_{j=1}^{1}a^{k\alpha_{1}}_{1j}
& \cdots &
\sum_{j=1}^{n}a^{k(\alpha_{1}+\cdots+\alpha_{n-1}+1)}_{1j} &
\cdots & \sum_{j=1}^{n}a^{kk}_{1j}
\end{array} \right)
\end{displaymath}
and this is the adjacency matrix of the quotient network
corresponding to $\bowtie$.
\end{proof}

\section{Computer algorithms}
\label{sec:algorithm}
\subsection{Balanced equivalence relations}
The above combinatorial properties of adjacency matrices leads to
a computer algorithm which determines all balanced equivalence
relations and adjacency matrices $A_{\bowtie}$ of associated
quotient networks $\mathcal{G}/_{\bowtie}$ for a given coupled
network $\mathcal{G}$.

We enumerate all possible equivalence relations $\bowtie$ of
$n$-cells, and test which are balanced. If the top lattice node
is known in advance, e.g. using the algorithm in \cite{Aldis} or
\cite{BelykhHasler2011}, only equivalence relations which refine
this top node need be tested. To test if each $\bowtie$ is
balanced, we construct an $n\times k$ matrix, where $k$ is the
number of equivalence classes of $\bowtie$, from the $n\times n$
adjacency matrix of $\mathcal{G}$. All rows in each equivalence
class are identical for a balanced equivalence relation. Finally,
for balanced equivalence relations, we can construct adjacency
matrices $A_{\bowtie}$ of the corresponding quotient networks
$\mathcal{G}/_{\bowtie}$.

\noindent\underline{\textbf{Step 1:}} For a given $n$-cell coupled cell network
$\mathcal{G}=(\mathcal{C},\mathcal{E},\sim_{C},\sim_{E})$, we
express the corresponding $n\times n$ adjacency matrix $A$ as
\begin{displaymath}
A=[C_{1}\cdots C_{n}]
\end{displaymath}
where $C_{i}\in\mathbb{R}^{n\times 1 }, i=1,\ldots ,n$ are column
vectors. Let $\overline{C}_{p}$ denote the $\bowtie$-equivalence
classes on $\mathcal{C}$ where $p=1,\ldots , k$. For example, if
$\mathcal{C}=\{\{1,3,5\},\{2\},\{4\}\}$ then
$\overline{C}_{1}=\{1,3,5\}$, $\overline{C}_{2}=\{2\}$, and
$\overline{C}_{3}=\{4\}$. Note that
$\sum_{p=1}^{k}|\overline{C}_{p}|=n$. Let $\overline{C}_{p1}$ be
the first element of each equivalence class. We assume that
$\overline{C}_{11}<\overline{C}_{21}<\cdots <\overline{C}_{k1}$,
where these cell numbers are used as indices for row vectors in
Step $3$.

We generate a new $n\times k$ matrix $\widetilde{A}_{\bowtie}$
with columns
\begin{displaymath}
\widetilde{C}_{\bowtie_{p}}=\sum_{j\in\overline{C}_{p}}C_{j}\quad\textrm{for}\quad
p=1,\ldots ,k
\end{displaymath}
for all possible equivalence relations $\bowtie$.

Let $\widetilde{R}_{\bowtie_{i}}\in\mathbb{R}^{1\times k}$, where
$i=1,\ldots ,n$, denote the row vectors of this new $n\times k$
matrix $\widetilde{A}_{\bowtie}$. Therefore,
\begin{displaymath}
\widetilde{A}_{\bowtie}=\left( \begin{array}{c} \widetilde{R}_{\bowtie_{1}}  \\
 \vdots\\
 \widetilde{R}_{\bowtie_{n}}
\end{array} \right)=[\widetilde{C}_{\bowtie_{1}}\cdots \widetilde{C}_{\bowtie_{k}}]
\end{displaymath}

\noindent\underline{\textbf{Step 2:}} Now we determine which
equivalence relations $\bowtie$ are balanced. An
equivalence relation $\bowtie$ on $\mathcal{C}$ is balanced if
and only if for all $p=1,\ldots ,k$ we have:
\begin{equation}
\label{equ:row-valid-condition}
\widetilde{R}_{\bowtie_{l}}=\widetilde{R}_{\bowtie_{m}}
\quad\forall l,m\in\overline{C}_{p}
\end{equation}

\noindent\underline{\textbf{Step 3:}} If the above condition~\ref{equ:row-valid-condition} is satisfied, the $k\times k$
adjacency matrix of the quotient network $A_{\bowtie}$
corresponding to a balanced equivalence relation $\bowtie$ is
given by:
\begin{displaymath}
A_{\bowtie}=\left( \begin{array}{c} R_{\bowtie_{1}}  \\
 \vdots\\
 R_{\bowtie_{k}}
\end{array} \right)
\end{displaymath}
where $R_{\bowtie_{i}}\in\mathbb{R}^{1\times k}$,
$i=1,\ldots ,k$ are representative row vectors in each equivalence class.

\begin{example}
Consider the homogeneous network $\mathcal{G}_{5}$ in Table~\ref{tab:coupled_cell_networks} with the corresponding adjacency matrix shown in Figure \ref{fig:example1}.

\begin{figure}[h!]
\begin{center}
\begin{tabular}{cc}
\multirow{5}{*}{\includegraphics[scale=0.2]{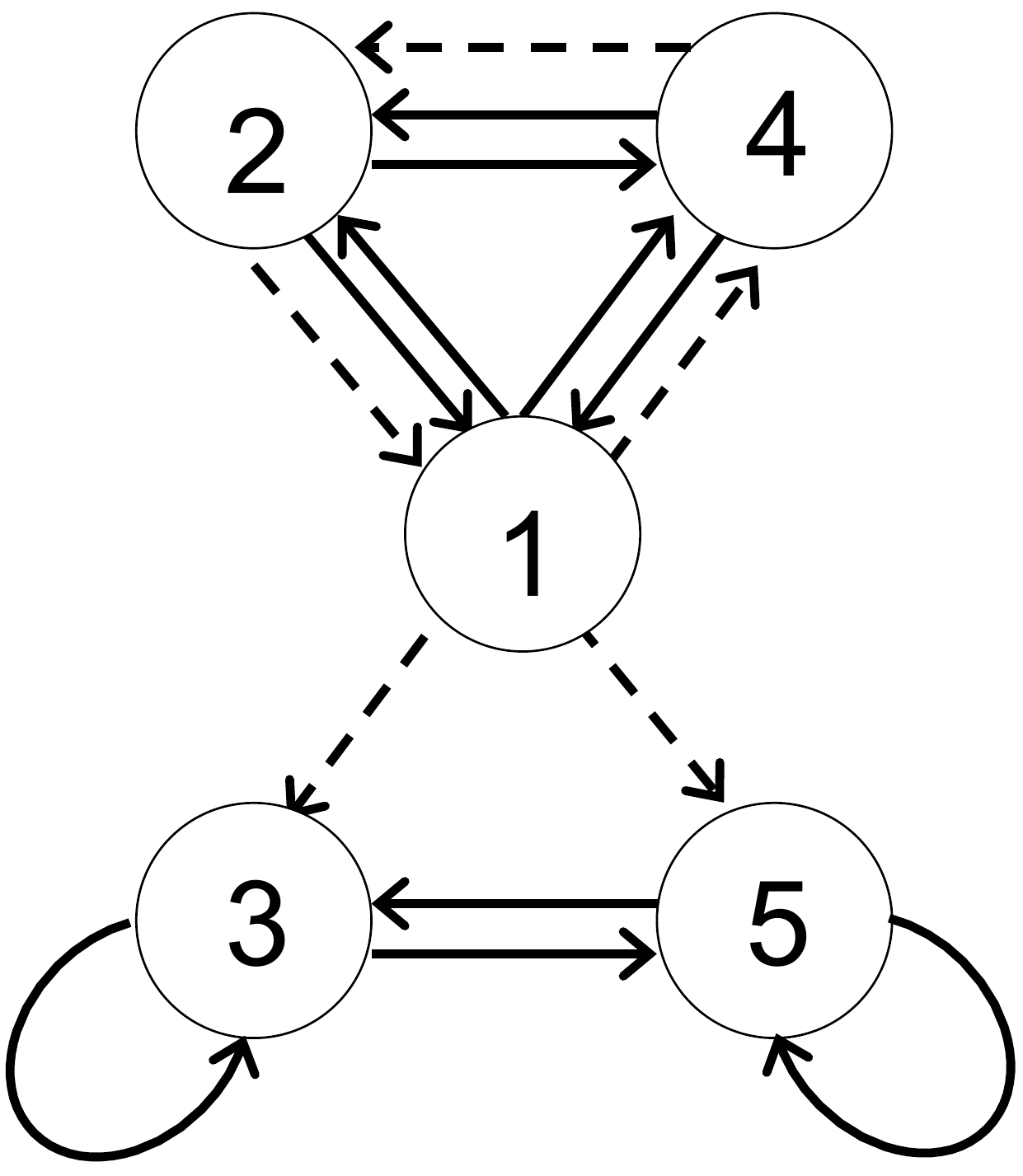}} &
\multirow{5}{*}{$\left(\begin{array}{ccccc}
0 & e_{1}+e_{2} & 0 & e_{1}   & 0 \\
e_{1} & 0   & 0 & e_{1}+e_{2} & 0 \\
e_{2} & 0   & e_{1} & 0   & e_{1} \\
e_{1}+e_{2} & e_{1} & 0 & 0 & 0 \\
e_{2} & 0   & e_{1} & 0   & e_{1} 
\end{array}\right)$}\\
& \\
& \\
& \\
& \\
& \\
&
\end{tabular}
\end{center}
\caption{Homogeneous network $\mathcal{G}_{5}$ in Table~\ref{tab:coupled_cell_networks} with the corresponding symbolic adjacency matrix.}
\label{fig:example1}
\end{figure}

We determine if the equivalence relation (i.e., partition of cells) $\bowtie = (135)(24)$ is balanced or not by the matrix computation. There are two equivalence classes $\overline{C}_{1}=\{1,3,5\}$ and $\overline{C}_{2}=\{2,4\}$. We generate a new $5\times 2$ matrix $\widetilde{A}_{\bowtie}$ by adding columns $1$, $3$, $5$ and $2$, $4$ such that
\begin{displaymath}
\widetilde{A}_{\bowtie}=
\left(\begin{array}{cc}
0 & 2e_{1}+e_{2} \\
e_{1} & e_{1}+e_{2} \\
2e_{1}+e_{2} & 0    \\
e_{1}+e_{2} & e_{1}  \\
2e_{1}+e_{2} & 0  
\end{array}\right)
\end{displaymath}

The equivalence relation $\bowtie$ is balanced if and only if
\begin{displaymath}
[0\quad 2e_{1}+e_{2}]=[2e_{1}+e_{2}\quad 0]=[2e_{1}+e_{2}\quad 0]\quad\textrm{and}\quad [e_{1}\quad e_{1}+e_{2}]=[e_{1}+e_{2}\quad e_{1}].
\end{displaymath}
However, this does not hold. Thus the equivalence relation $\bowtie=(135)(24)$ is not balanced.

On the other hand, let $\bowtie = (124)(3)(5)$. There are three equivalence classes $\overline{C}_{1}=\{1,2,4\}$,  $\overline{C}_{2}=\{3\}$ and $\overline{C}_{3}=\{5\}$. We generate a new $5\times 3$ matrix $\widetilde{A}_{\bowtie}$ by adding columns $1$, $2$, $4$ such that
\begin{displaymath}
\widetilde{A}_{\bowtie}=
\left(\begin{array}{ccc}
2e_{1}+e_{2} & 0 & 0 \\
2e_{1}+e_{2} & 0 & 0\\
e_{2} & e_{1} & e_{1} \\
2e_{1}+e_{2}  & 0 & 0 \\
e_{2} & e_{1} & e_{1}  
\end{array}\right)
\end{displaymath}

The equivalence relation $\bowtie$ is balanced if and only if
\begin{displaymath}
[2e_{1}+e_{2}\quad 0\quad 0]=[2e_{1}+e_{2}\quad 0\quad 0]=[2e_{1}+e_{2}\quad 0\quad 0].
\end{displaymath}
This is satisfied. Thus the equivalence relation $\bowtie=(124)(3)(5)$ is balanced. As a result, the quotient network $\mathcal{G}/_{\bowtie}$ corresponding to the balanced equivalence relation $\bowtie=(124)(3)(5)$ and the associated $3\times 3$ adjacency matrix $A_{\bowtie}$ are given in Figure \ref{fig:example1-quotient}.
\begin{figure}[h!]
\begin{center}
\begin{tabular}{cc}
$A_{\bowtie}$ & $\mathcal{G}/_{\bowtie}$\\
\hline
& \\
\multirow{4}{*}{$\left( \begin{array}{ccc} 
2e_{1}+e_{2} & 0 & 0 \\
e_{2} & e_{1} & e_{1} \\
e_{2} & e_{1} & e_{1} 
\end{array} \right)$} &
\multirow{4}{*}{\includegraphics[scale=0.2]{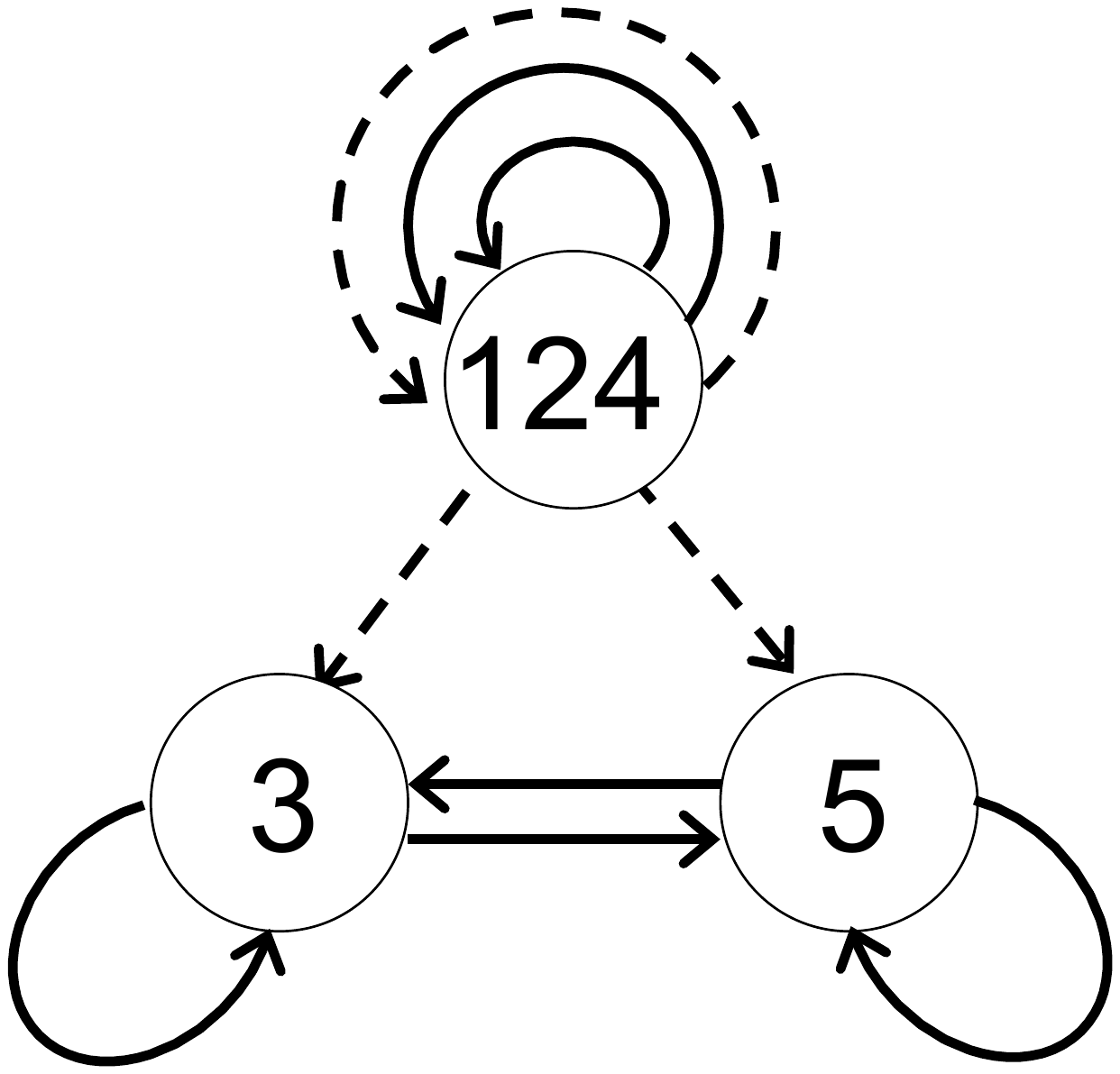}}\\
& \\
& \\
& \\
&
\end{tabular}
\end{center}
\caption{The quotient network $\mathcal{G}/_{\bowtie}$ corresponding to a balanced equivalence relation $\bowtie=(124)(3)(5)$ and the associated $3\times 3$ adjacency matrix $A_{\bowtie}$.}
\label{fig:example1-quotient}
\end{figure}
\label{ex:balanced}
\end{example}
\\
\\
The algorithm as shown above describes the graph $\mathcal{G}$ with a single adjacency matrix $A$ containing symbolic entries for the different arrow types. We now discuss an alternative representation using separate integer matrices for each arrow type, which for most programming languages is more practical to implement. The following definition is a variation of Definition $5.2$ in \cite{Aguiar-minimal} for homogeneous networks.

\begin{definition}
\label{def:adjacency-matrix2}
Let $\mathcal{G}=(\mathcal{C},\mathcal{E},\sim_{C},\sim_{E})$ be an $n$-cell coupled cell network with $l$ cell-types and $m$ arrow-types with $[c_{1}]_{C},\ldots,[c_{l}]_{C}$, the $\sim_{C}$-equivalence classes for cells and $[e_{1}]_{E},\ldots,[e_{m}]_{E}$, the $\sim_{E}$-equivalence classes for arrows. We define the \textit{adjacency matrix} of $\mathcal{G}$ with respect to $[e_{k}]_{E}$, for $k=1,\ldots, m$ to be the $n\times n$ matrix $M_{(\mathcal{G},k)}$. The $(i,j)$-entry corresponds to the number of arrows of types  $[e_{k}]_{E}$ from cell $j$ to cell $i$.
\end{definition}

Notice by construction we have:

\begin{displaymath}
A = \sum^m_{k=1} e_k M_{(\mathcal{G},k)}.
\end{displaymath}

Therefore the above algorithm procedure can now be applied to each of the $m$ arrow type specific matrices individually, and if it holds for all of them, it also holds for $A$ as well. We now repeat Example~\ref{ex:balanced} to demonstrate this.

\begin{example}
For the coupled cell network $\mathcal{G}$, two adjacency matrices $M_{(\mathcal{G},1)}=(m^{1}_{ij})$ (solid) and $M_{(\mathcal{G},2)}=(m^{2}_{ij})$ (dashed) for two different arrow types are defined as in Figure~\ref{fig:example2}.

\begin{figure}[h!]
\begin{center}
\begin{tabular}{cc}
\multirow{5}{*}{\includegraphics[scale=0.2]{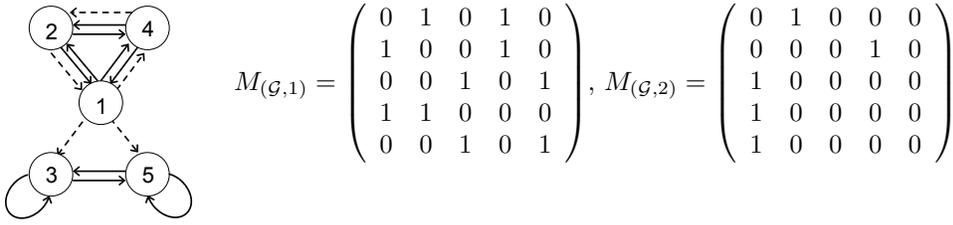}} &
\multirow{5}{*}{
$M_{(\mathcal{G},1)}=\left(\begin{array}{ccccc}
0 & 1 & 0 & 1 & 0 \\
1 & 0 & 0 & 1 & 0 \\
0 & 0 & 1 & 0 & 1 \\
1 & 1 & 0 & 0 & 0 \\
0 & 0 & 1 & 0 & 1 
\end{array}\right)$, 
$M_{(\mathcal{G},2)}=\left(\begin{array}{ccccc}
0 & 1 & 0 & 0 & 0 \\
0 & 0 & 0 & 1 & 0 \\
1 & 0 & 0 & 0 & 0 \\
1 & 0 & 0 & 0 & 0 \\
1 & 0 & 0 & 0 & 0 
\end{array}\right)$}\\
& \\
& \\
& \\
& \\
& \\
&
\end{tabular}
\end{center}
\caption{Homogeneous network $\mathcal{G}_{5}$ in Table~\ref{tab:coupled_cell_networks} with two adjacency matrices $M_{(\mathcal{G},1)}=(m^{1}_{ij})$ (solid) and $M_{(\mathcal{G},2)}=(m^{2}_{ij})$ (dashed) for two different arrow types.}
\label{fig:example2}
\end{figure}

Using these arrow type specific adjacency matrices, we determine if the equivalence relation $\bowtie=(135)(24)$ is balanced. There are two equivalence classes $\overline{C}_{1}=\{1,3,5\}$ and $\overline{C}_{2}=\{2,4\}$. We generate new $5\times 2$ matrices by adding vectors in columns $1$, $3$, $5$ and $2$, $4$:
\begin{displaymath}
\widetilde{M}_{{(\mathcal{G},1)}_{\bowtie}}=
\left(\begin{array}{cc}
0 & 2 \\
1 & 1 \\
2 & 0 \\
1 & 1 \\
2 & 0  
\end{array}\right)
, \quad
\widetilde{M}_{{(\mathcal{G},2)}_{\bowtie}}=
\left(\begin{array}{cc}
0 & 1 \\
0 & 1 \\
1 & 0 \\
1 & 0 \\
1 & 0  
\end{array}\right)
\end{displaymath}

The equivalence relation $\bowtie$ is balanced if and only if for each arrow specific $5 \times 2$ matrix rows $1$, $3$, and $5$ are equal, and rows $2$ are $4$ equal.
However, this does not hold. Thus the equivalence relation $\bowtie=(135)(24)$ is not balanced.

On the other hand, let $\bowtie=(124)(3)(5)$. There are three equivalence classes $\overline{C}_{1}=\{1,2,4\}$,  $\overline{C}_{2}=\{3\}$ and $\overline{C}_{3}=\{5\}$. We generate new $5\times 3$ matrices by adding vectors in columns $1$, $2$, $4$:
\begin{displaymath}
\widetilde{M}_{{(\mathcal{G},1)}_{\bowtie}}=
\left(\begin{array}{ccc}
2 & 0 & 0 \\
2 & 0 & 0 \\
0 & 1 & 1 \\
2 & 0 & 0 \\
0 & 1 & 1  
\end{array}\right)
, \quad
\widetilde{M}_{{(\mathcal{G},2)}_{\bowtie}}=
\left(\begin{array}{ccc}
1 & 0 & 0 \\
1 & 0 & 0\\
1 & 0 & 0 \\
1 & 0 & 0 \\
1 & 0 & 0  
\end{array}\right)
\end{displaymath}

The equivalence relation $\bowtie$ is balanced if and only if for each arrow specific $5 \times 2$ matrix rows $1$, $2$, and $4$ are equal.
This is satisfied. Thus the equivalence relation $\bowtie=(124)(3)(5)$ is balanced. As a result, the quotient network $\mathcal{G}/_{\bowtie}$ corresponding to a balanced equivalence relation $\bowtie=(124)(3)(5)$ and the associated $3\times 3$ arrow type specific adjacency matrices are given in Figure \ref{fig:example2-quotient}.
\begin{figure}[h!]
\begin{center}
\begin{tabular}{cc}
$A_{\bowtie}$ & $\mathcal{G}/_{\bowtie}$\\
\hline
& \\
\multirow{4}{*}{$M_{\bowtie_{(\mathcal{G},1)}}=\left( \begin{array}{ccc} 
2 & 0 & 0 \\
0 & 1 & 1 \\
0 & 1 & 1 
\end{array} \right)$,
$M_{\bowtie_{(\mathcal{G},2)}}=\left( \begin{array}{ccc} 
1 & 0 & 0 \\
1 & 0 & 0 \\
1 & 0 & 0 
\end{array} \right)$} &
\multirow{4}{*}{\includegraphics[scale=0.2]{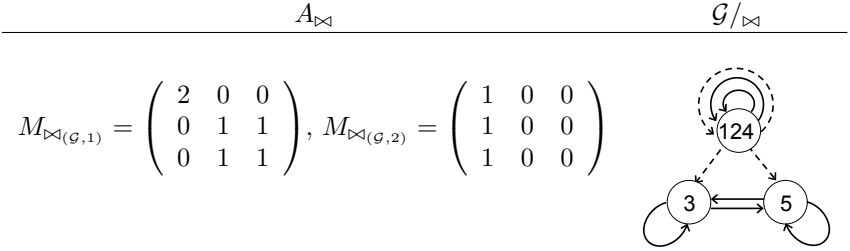}}\\
& \\
& \\
& \\
& \\
&
\end{tabular}
\end{center}
\caption{The quotient network $\mathcal{G}/_{\bowtie}$ corresponding to a balanced equivalence relation $\bowtie=(124)(3)(5)$ and the associated $3\times 3$ arrow type specific adjacency matrices.}
\label{fig:example2-quotient}
\end{figure}
\label{ex:balanced-alternative}

\end{example}

\subsection{Lattice of balanced equivalence relations}
Using the above computer algorithm, we can determine all balanced
equivalence relations and corresponding quotient networks for a
given coupled cell network. Now, using the refinement relation, we construct a complete lattice of balanced equivalence relations for a given coupled cell network.

Let $p$ be the total number of balanced equivalence relations of a given coupled cell network. We aim to compute a $p \times p$ adjacency matrix $L = (l_{ij})$ for the lattice with entries $1$ where $\bowtie_i$ is covered by $\bowtie_j$, and $0$ otherwise.

\noindent\underline{\textbf{Step 1:}} Without loss of generality, order the $p$ balanced equivalence relations by increasing rank (number of equivalence classes). This ensures that the top element is first and the bottom element last, and that the matrix $L$ will be lower triangular.

\noindent\underline{\textbf{Step 2:}}
Construct $p\times p$ matrix $B=(b_{ij})$ with entries $1$ where $\bowtie_i \prec \bowtie_j$ ($\bowtie_i$ refines $\bowtie_j$) and $0$ otherwise. This is almost the desired adjacency matrix, but it includes extra edges since refinement is not as strict as covering.

\noindent\underline{\textbf{Step 3:}}
Calculate $p \times p$ matrix $T = (t_{ij}) = B^2$. Non-zero entries $t_{ij}$ indicate nodes $i$ and $j$ are connected by a path of length two via some intermediate third node $k$, thus $\bowtie_i \prec \bowtie_k \prec \bowtie_j$, meaning $\bowtie_i$ is not covered by $\bowtie_j$. We can assume $k$ is distinct from $i$ and $j$ since the diagonal entries of $B$ are zero.

\noindent\underline{\textbf{Step 4:}}
Construct $p\times p$ matrix $L=(l_{ij})$ using $l_{ij} = 1$ if $b_{ij} = 1$ and $t_{ij} = 0$, and $l_{ij} = 0$ otherwise.

For larger lattices computing the full matrix $T$ in step 3 is increasingly time consuming, and only a fraction of the values are needed in step 4. For computational efficiency, only where $b_{ij} = 1$ do we need to check if $t_{ij} = 0$. We do this by considering the existence of a two step path between lattice nodes $i$ and $j$ via node $k$ (i.e. $b_{ik} =1$ and $b_{kj} = 1$), and as a further optimization only those nodes $k$ with $\text{rank}(i) < \text{rank}(k) < \text{rank}(j)$ need be considered.

The matrix $L$ is the adjacency matrix for the $p$-node lattice, and defines the set of edges. Lattices are by convention drawn as diagrams with an up/down orientation with the top lattice element higher than the bottom lattice element. Additionally we require lattices nodes of the same rank (number of equivalence classes) to be shown at the same height.

\begin{example}
Consider the five-cell homogeneous network $\mathcal{G}_{5}$ in Table~\ref{tab:coupled_cell_networks}. This network has $5$ balanced equivalence relations as shown in Figure \ref{fig:g5}, in rank order.

\begin{figure}[h]
\begin{center}
\begin{tabular}{cl}
$\mathcal{G}_{5}$ & Balanced equivalence relations \\
\hline
\multirow{7}{*}{\includegraphics[scale=0.18]{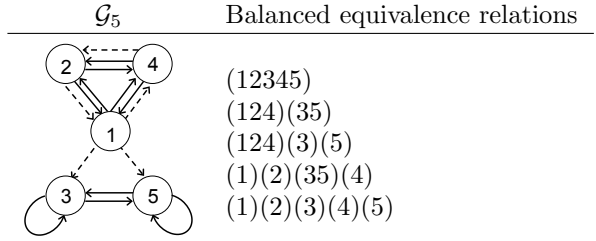}} & \\
& $(12345)$ \\
 & $(124)(35)$ \\
 & $(124)(3)(5)$ \\
 & $(1)(2)(35)(4)$ \\
 & $(1)(2)(3)(4)(5)$ \\\\
\end{tabular}
\end{center}
\caption{Five-cell homogeneous network $\mathcal{G}_{5}$ in Table~\ref{tab:coupled_cell_networks} and all possible $5$ balanced equivalence relations.}
\label{fig:g5}
\end{figure}

We construct the $5\times 5$ matrix $B$, which represents refinement relations between the $5$ balanced equivalence relations. Then $T=B^{2}$ is a simple matrix multiplication, which is used to remove unwanted edges from matrix $B$ to give $L$. Table \ref{tab:adj-lattice} shows these three matrices.  Matrix $L$ is used as the adjacency matrix when drawing the lattice, with vertical positions dictated by the lattice node ranks (Figure~\ref{fig:lattice-g5}, right).
\\
\\

 \begin{table}[h!]
 \begin{center}
 \begin{tabular}{ccc}
 $B$ & $T=B^{2}$ & $L$ \\
 \hline
 \\
 $\left( \begin{array}{ccccc} 
0 & 0 & 0 & 0 & 0\\
1 & 0 & 0 & 0 & 0 \\
1 & 1 & 0 & 0 & 0\\
1 & 1 & 0 & 0 & 0\\
1 & 1 & 1 & 1 & 0 
\end{array} \right)$
&
 $\left( \begin{array}{ccccc} 
0 & 0 & 0 & 0 & 0\\
0 & 0 & 0 & 0 & 0 \\
1 & 0 & 0 & 0 & 0\\
1 & 0 & 0 & 0 & 0\\
3 & 2 & 0 & 0 & 0 
\end{array} \right)$
&
$\left( \begin{array}{ccccc} 
0 & 0 & 0 & 0 & 0\\
1 & 0 & 0 & 0 & 0 \\
0 & 1 & 0 & 0 & 0\\
0 & 1 & 0 & 0 & 0\\
0 & 0 & 1 & 1 & 0 
\end{array} \right)$\\
& &
\end{tabular}
\end{center}
\caption{Three $5\times 5$ matrices, $B$, $T$, and $L$ defined in the lattice algorithm for the network $\mathcal{G}_{5}$ in Table \ref{tab:coupled_cell_networks}.
}  
\label{tab:adj-lattice}
\end{table}

\begin{figure}[h!]
\begin{center}
\begin{tabular}{lc}
Covering relations & Lattice of balanced equivalence relations \\
\hline
 & \multirow{7}{*}{\includegraphics[scale=0.2]{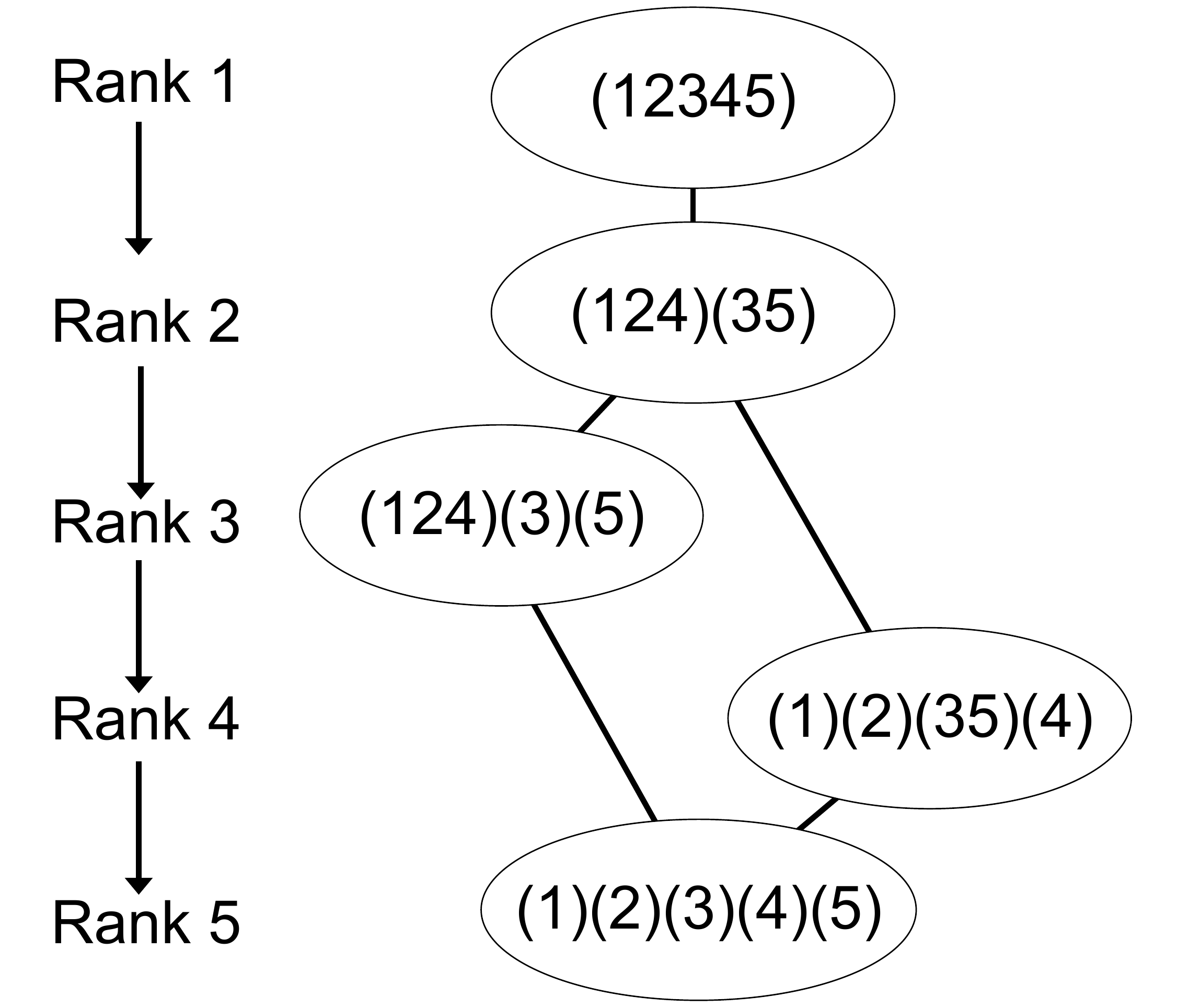}}  \\
$(124)(35)$       $<$ $(12345)$ & \\
&  \\
$(124)(3)(5)$     $<$ $(124)(35)$ &  \\
$(1)(2)(35) (4)$  $<$ $(124)(35)$ &  \\
& \\
$(1)(2)(3)(4)(5)$ $<$ $(124)(3)(5)$ & \\
$(1)(2)(3)(4)(5)$ $<$ $(1)(2)(35)(4)$ & \\
& \\
 & \\
 \end{tabular}
 \end{center}
 \caption{The covering relations and the lattice of balanced equivalence relations of the five-cell homogeneous network $\mathcal{G}_{5}$ in Table~\ref{tab:coupled_cell_networks}.}
 \label{fig:lattice-g5}
 \end{figure}

\end{example}

Computing all the balanced equivalence relations of a network size $n$ scales with the number of equivalence relations, given by the Bell number, and is thus combinatorial with the network size $n$. On a recent computer (2008 Apple Mac Pro) using a single CPU using the brute force approach with a single edge type, small networks take less than a second to compute, $10$ nodes about $20$ seconds, $11$ nodes about $2$ minutes, $12$ nodes about $15$ minutes, $13$ nodes under $2$ hours, and $14$ nodes about $12$ hours (with variation depending on the network topology).

We have also implemented the algorithm in \cite{BelykhHasler2011}, and generalized this to consider multiple arrow types.
The result is equivalent to a single-phase simplification of the algorithm in \cite{Aldis}, but less complicated to implement (see appendix).
Where we have computed the full list of balanced coloring and identified the unique minimal balanced coloring, the results agree. As described above, this offers a shortcut when computing all the balanced equivalence relations, although for highly symmetric networks this optimization has limited benefit -- and for regular networks offers no improvement. However, the time saving can be dramatic especially for random networks. As an example, a bidirectionally coupled chain (where the end nodes do not have self coupling) of up to $20$ nodes takes under a second, $30$ nodes takes about $20$ seconds, and $40$ nodes about $10$ minutes.

The time (and the memory requirements) needed to compute the lattice scales quadratically with the number of lattice nodes. In the worst case of a fully connected network all equivalence relations are balanced, giving the largest possible lattice and the longest compute time, taking about a second for the $877$ node lattice ($n=7$), $20$ seconds for the $4140$ node lattice ($n=8$), and $10$ minutes for the $21147$ node lattice ($n=9$).

In short, in our current algorithm implementation computing the balanced colorings of \emph{regular} networks more than $15$ nodes is impractical, although inhomogeneous networks are much easier to deal with. Additionally, the computations could in principle be run in parallel across multiple CPU cores, giving a potential linear speed up.

\section{Examples}
\label{sec:examples}
The lattice of partial synchronies computed by the algorithm shown tells us about the existence of all possible partial synchronies determined by the given network structure. In this section, we select example topics from synchronized chaos and coupled neuron models, and demonstrate how a symbolic adjacency matrix can be defined for each example problem, and construct a complete lattice of all possible partial synchronies derived from the given networks structure. Some dynamical properties such as the stability of possible partial synchronies depends on the specific form of the given vector field. We demonstrate the numerical analysis of the stability of partial synchronies for the topic of synchronized chaos.
 
\subsection{Coupled identical R\"{o}ssler systems}
\label{subsec:rossler}
We consider a bidirectional ring of six diffusively coupled R\"{o}ssler systems $\mathbf{u}_{i}=(x_{i},y_{i},z_{i})\in\mathbb{R}^{3}$ for $i=1,2,\ldots,6$:
\begin{eqnarray*}
\dot{x}_{i} &=& -(y_{i}+z_{i})+\epsilon(x_{i-1}-2x_{i}+x_{i+1}), \\
\dot{y}_{i} &=& x_{i}+ay_{i}, \\
\dot{z}_{i} &=& b+(x_{i}-c)z_{i},
\end{eqnarray*}
with periodic boundary conditions $x_{0}=x_{6}$ and $x_{7}=x_{1}$.

Since this is a regular network, the adjacency matrix consists of non-negative integers and the admissible vector field is defined by a single map $\mathbf{f}$, which is realized by the above defined system. Table \ref{tab:adj_and_system} shows the adjacency matrix and the associated coupled cell system. The complete lattice of balanced equivalence relations of this network is given in Figure \ref{fig:lattice_rossler}.

\begin{table}[h]
\begin{center}
\begin{tabular}{ccl}
adjacency matrix & \hspace{1cm} & coupled cell system\\
\hline
& & \\
\multirow{5}{*}{
$A=\left(\begin{array}{rrrrrr}
0 & 1 & 0 & 0 & 0 & 1\\
1 & 0 & 1 & 0 & 0 & 0\\
0 & 1 & 0 & 1 & 0 & 0\\
0 & 0 & 1 & 0 & 1 & 0\\
0 & 0 & 0 & 1 & 0 & 1\\
1 & 0 & 0 & 0 & 1 & 0
\end{array}\right)$} & \hspace{1cm} &
\multirow{6}{*}{$\begin{array}{rl}
\dot{\mathbf{u}}_{1} &= \mathbf{f}(\mathbf{u}_{1}, \overline{\mathbf{u}_{6}, \mathbf{u}_{2}})\\
\dot{\mathbf{u}}_{2} &= \mathbf{f}(\mathbf{u}_{2}, \overline{\mathbf{u}_{1}, \mathbf{u}_{3}})\\
\dot{\mathbf{u}}_{3} &= \mathbf{f}(\mathbf{u}_{3}, \overline{\mathbf{u}_{2}, \mathbf{u}_{4}})\\
\dot{\mathbf{u}}_{4} &= \mathbf{f}(\mathbf{u}_{4}, \overline{\mathbf{u}_{3}, \mathbf{u}_{5}})\\
\dot{\mathbf{u}}_{5} &= \mathbf{f}(\mathbf{u}_{5}, \overline{\mathbf{u}_{4}, \mathbf{u}_{6}})\\
\dot{\mathbf{u}}_{6} &= \mathbf{f}(\mathbf{u}_{6}, \overline{\mathbf{u}_{5}, \mathbf{u}_{1}})
\end{array}$} \\
& &\\
& &\\
& &\\
& & \\
& & \\
& &
\end{tabular}
\end{center}
\caption{Adjacency matrix which represents interactions among R\"{o}ssler systems and the associated coupled cell system. The overline indicates that influence from coupling cells to that cell are identical, i.e., $f(x_{i},\overline{x_{j},x_{k}})$ means $f(x_{i},x_{j},x_{k})=f(x_{i},x_{k},x_{j})$.}
\label{tab:adj_and_system}
\end{table}

\begin{figure}[h]
\begin{center}
\includegraphics[scale=0.3]{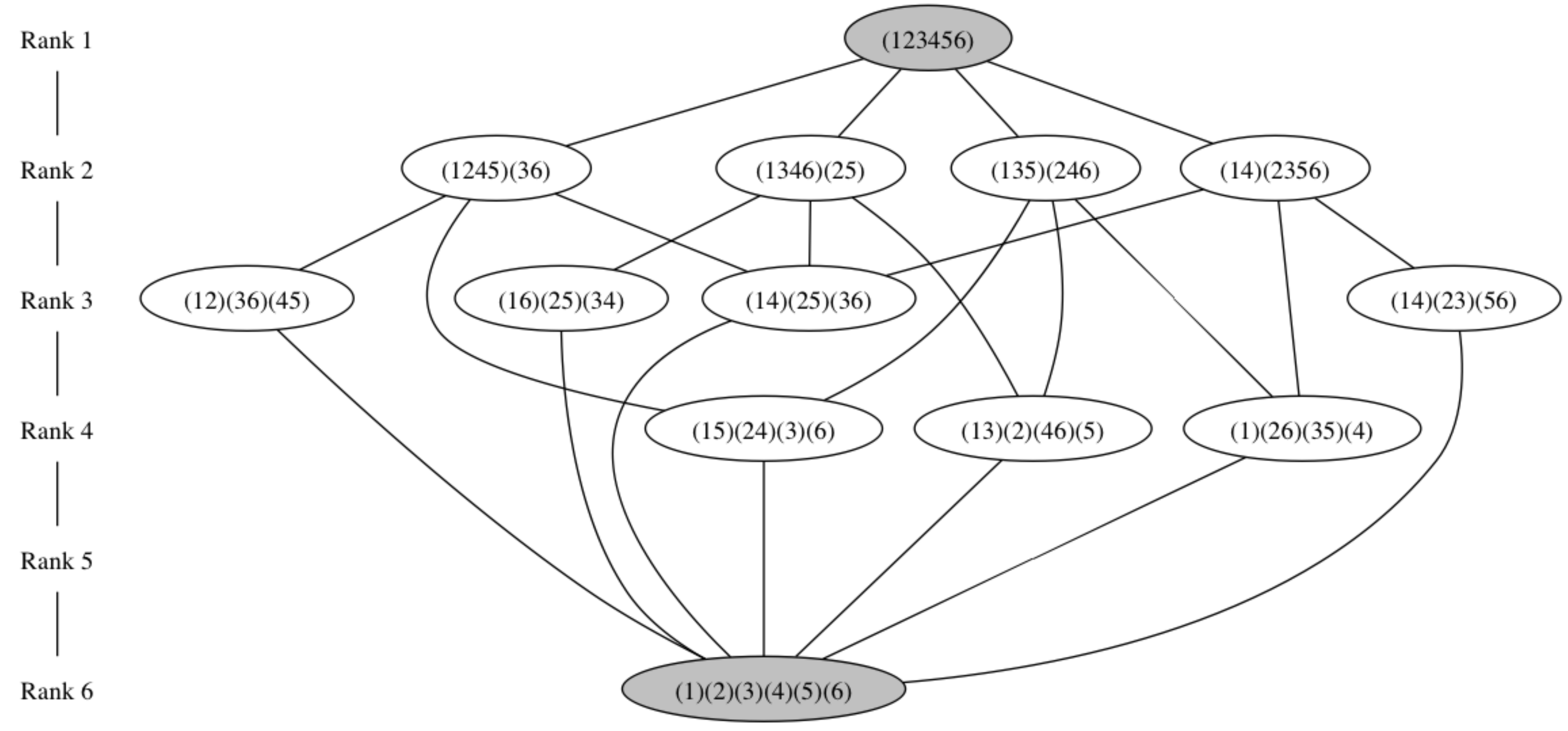}
\end{center}
\caption{Lattice of balanced equivalence relations of bidirectional ring of six diffusively coupled R\"{o}ssler systems. Some partial synchronies are symmetrically related. For example the balanced equivalence relations $(1245)(36)$, $(1346)(25)$, and $(14)(2356)$ are permutation symmetric and they give the same pattern of partial synchrony $(\alpha,\alpha,\beta,\alpha,\alpha,\beta)$ in a bidirectional ring.}
\label{fig:lattice_rossler}
\end{figure}

In the numerical analysis, we take parameter values $a=0.2$, $b=0.2$, $c=5.7$, and vary the coupling parameter $\epsilon$ for the stability analysis of synchrony subspaces of this specific vector field. The full synchrony subspace $\triangle=\{\mathbf{u}_{1}=\cdots=\mathbf{u}_{6}\}$ is globally stable from $\epsilon\approx 0.2$ to $\epsilon\approx 1.0$, and attracts all trajectories starting from randomly chosen initial conditions. With the loss of stability of  full synchronization above $\epsilon\approx 1.0$, the synchrony subspace $\triangle_{\bowtie}=\{\mathbf{u}_{1}=\mathbf{u}_{3}=\mathbf{u}_{5}, \mathbf{u}_{2}=\mathbf{u}_{4}=\mathbf{u}_{6}\}$ becomes globally stable, and attracts all trajectories. Figure \ref{fig:six_rossler_difference} ($a$) shows the behaviour of $x_{1}-x_{2}$, which is the difference between the first internal variables of the R\"{o}ssler systems $\mathbf{u}_{1}$ and $\mathbf{u}_{2}$, when changing the parameter $\epsilon$. This figure shows the gain and loss of stability of the full synchrony subspace. Figure \ref{fig:six_rossler_difference} ($b$) shows the behaviour of $x_{1}-x_{3}$, where the corresponding variables $\mathbf{u}_{1},\mathbf{u}_{3}$ remain synchronized above $\epsilon\approx 1.0$. Only the behaviors of $x_{1}-x_{2}$ and $x_{1}-x_{3}$ are illustrated in Figure \ref{fig:six_rossler_difference}, qualitatively the same behaviors are observed for the other pairwise comparisons.

R\"{o}ssler systems with $x$-component coupling, as shown in this example, are known to exhibit a phenomena called short wavelength bifurcations \cite{Heagy1995} in which the synchronous chaotic state loses its stability with an increase of coupling strength.
The desynchronization behaviour of three diffusively coupled R\"{o}ssler systems with Neumann boundary conditions (a bidirectional chain with self-coupling of the end cells) was analysed in \cite{Belykh2000}. They showed that the singularity of the individual R\"{o}ssler systems and the use of $x$-coupling impose the existence of equilibria that lie outside of the fully synchronous subspace for any coupling strength, and proposed a direct link to the mechanism of desynchronization.

\begin{figure}[h!]
\begin{center}
\begin{tabular}{c}
\includegraphics[scale=0.38]{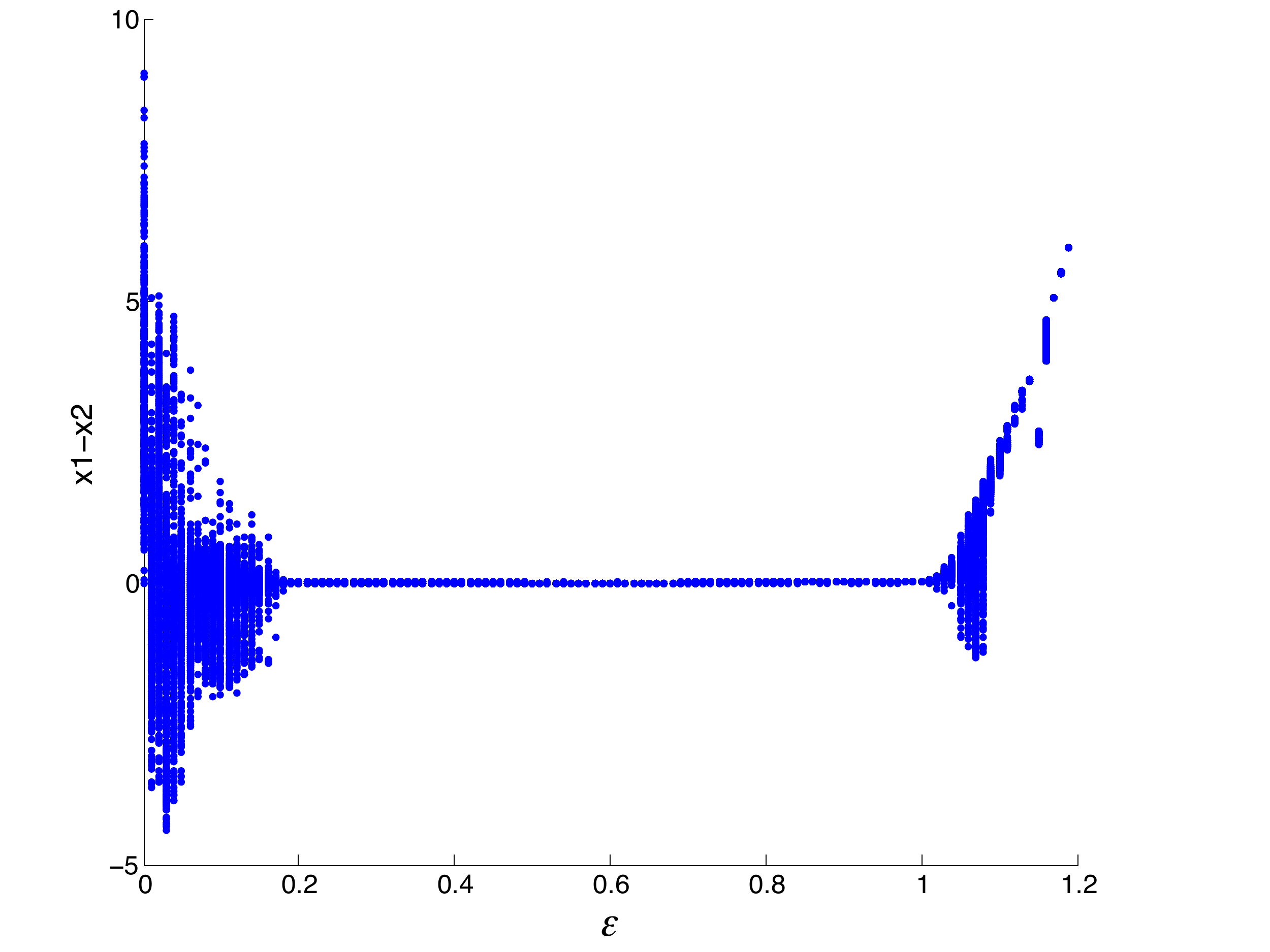}  \\
(a) \\
\\
\includegraphics[scale=0.38]{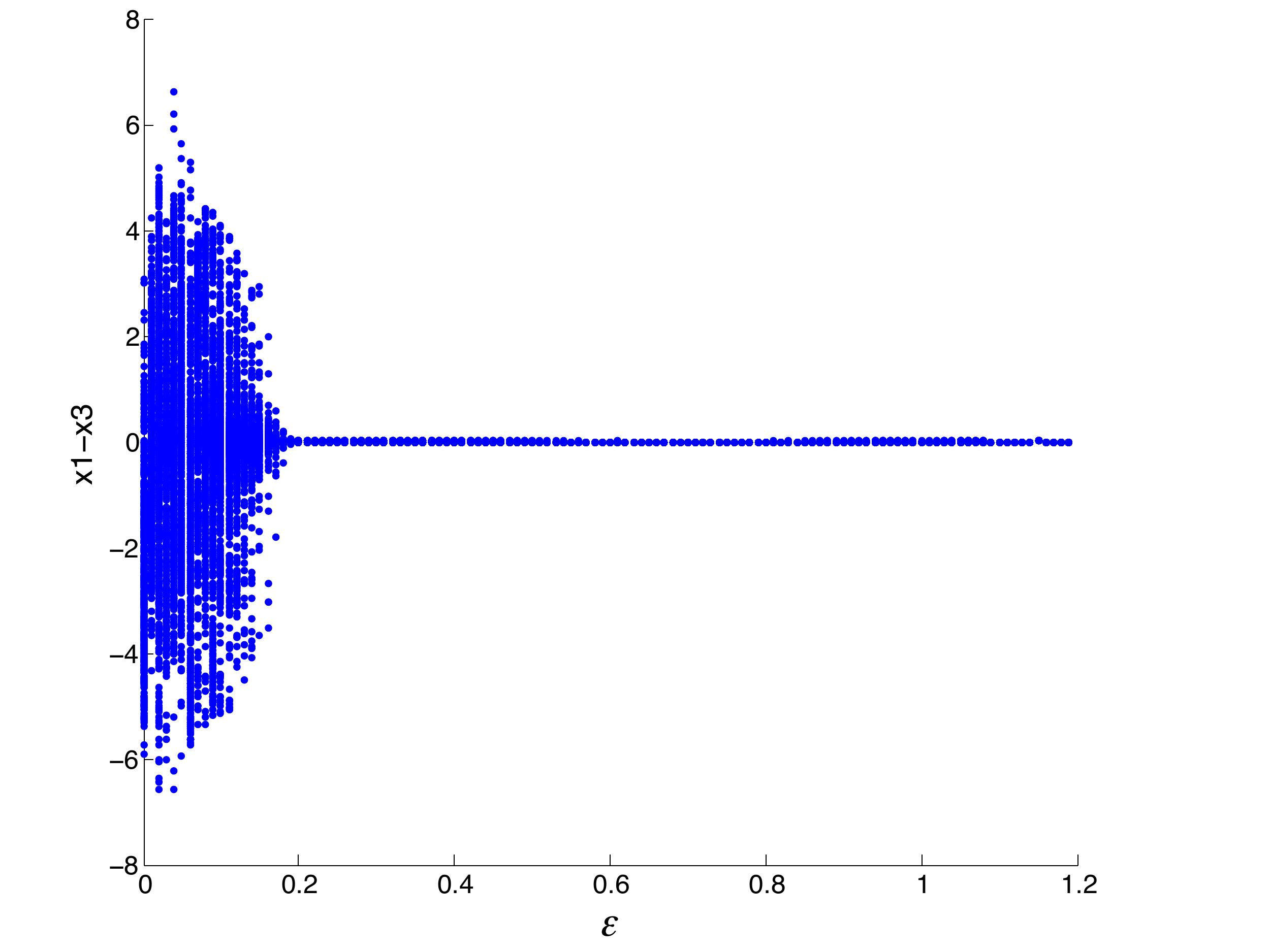} \\
(b)
\end{tabular}
\end{center}
\caption{The stability analysis of synchrony subspaces of a ring of six diffusively coupled R\"{o}ssler systems. (a) The behaviour of $x_{1}-x_{2}$ when changing the coupling parameter $\epsilon$. This difference is chosen as a representative difference of variables associated with the full synchrony subspace $\triangle=\{\mathbf{u}_{1}=\cdots=\mathbf{u}_{6}\}$. The fully synchronous subspace $\triangle$ is stable between $\epsilon\approx 0.2$ and $\epsilon\approx 1.0$. (b) The behaviour of $x_{1}-x_{3}$ when changing the coupling parameter $\epsilon$, chosen as a representative difference of the variables associated with the synchrony subspace $\triangle_{\bowtie}=\{\mathbf{u}_{1}=\mathbf{u}_{3}=\mathbf{u}_{5}, \mathbf{u}_{2}=\mathbf{u}_{4}=\mathbf{u}_{6}\}$. This partial synchrony subspace is observed above $\epsilon\approx 1.0$ (when the fully synchrony  subspace looses stability). The other parameter values are $a=0.2$, $b=0.2$, $c=5.7$. For each fixed parameter value $\epsilon$ (in steps of $0.01$ for $0\leq\epsilon\leq 1.2$), $500$ different initial conditions are generated and the difference $x_{1}-x_{2}$ or $x_1-x_3$ is plotted using the state variables after $40000$ iterates with the time step $h=0.01$. }
\label{fig:six_rossler_difference}
\end{figure}

A lattice-like hierarchy of synchrony subspaces of diffusively coupled identical systems was also discussed in \cite{Belykh2000}. For the chain network associated with Neumann boundary conditions, they hypothesized a clustering type hierarchy structure based on the number of nodes $n$ and its divisors, which we have verified for up to $n=15$ (see supplementary material). A related result for a linear chain with feedback in \cite{Stewart} gives an explicit lattice construction.

\subsection{Coupled Lorenz systems with heterogeneous coupling}
\label{subsec:lorenz}
In the next example, we demonstrate the symbolic adjacency matrix can be interpreted not only as a network structure, but also as different coupling strengths of identical individual systems. Consider cluster synchronization in an ensemble of five globally coupled Lorenz systems $\mathbf{u}_{i}=(x_{i},y_{i},z_{i})\in\mathbf{R}^{3}$ for $i=1,\cdots,5$ with heterogeneous coupling:
\begin{eqnarray*}
\dot{x}_{i} &=& \sigma(y_{i}-x_{i})+\frac{1}{N_{i}}\sum_{j=1}^{5}g_{ij}(x_{i}-x_{j})\\
\dot{y}_{i} &=& x_{i}(\rho-z_{i})-y_{i}\\
\dot{z}_{i} &=& x_{i}y_{i}-\beta z_{i}
\end{eqnarray*}

\noindent
where $\sigma=10$, $\rho=28$, $\beta=8/3$, $N_{i}=\sum_{j=1}^{5}g_{ij}$, and the coupling matrix $G=(g_{ij})$ is defined as

\begin{displaymath}
G=\left(\begin{array}{rrrrr}
0 & 1 & 1 & 2 & 4\\
1 & 0 & 1 & 2 & 4\\
1 & 1 & 0 & 1 & 5\\
1 & 2 & 3 & 0 & 2\\
3 & 2 & 1 & 2 & 0
\end{array}\right)
\end{displaymath}

We may regard an integer value $g_{ij}$ to be a weight from Lorenz system $j$ to Lorenz system $i$, and different non-zero integer values to be different weights from the corresponding systems. Note that $N_{i}=8$ for all $i=1,\ldots,5$. Then symbolically, the above coupling matrix $G$ can be represented with five different symbols, $a$, $b$, $c$, $d$, and $e$. Table \ref{tab:adj_and_system_lorenz} shows the symbolic adjacency matrix and the associated coupled cell system.

\begin{table}[h]
\begin{center}
\begin{tabular}{ccl}
symbolic adjacency matrix & \hspace{1cm} & coupled cell system\\
\hline
& & \\
\multirow{5}{*}{
$G=\left(\begin{array}{rrrrr}
0 & a & a & b & d\\
a & 0 & a & b & d\\
a & a & 0 & a & e\\
a & b & c & 0 & b\\
c & b & a & b & 0
\end{array}\right)$} & \hspace{1cm} &
\multirow{5}{*}{$\begin{array}{rl}
\dot{\mathbf{u}}_{1} &= \mathbf{f}(\mathbf{u}_{1}, \overline{\mathbf{u}_{2}, \mathbf{u}_{3}},\mathbf{u}_{4}, \mathbf{u}_{5})\\
\dot{\mathbf{u}}_{2} &= \mathbf{f}(\mathbf{u}_{2}, \overline{\mathbf{u}_{1}, \mathbf{u}_{3}},\mathbf{u}_{4}, \mathbf{u}_{5})\\
\dot{\mathbf{u}}_{3} &= \mathbf{g}(\mathbf{u}_{3}, \overline{\mathbf{u}_{1}, \mathbf{u}_{2},\mathbf{u}_{4}}, \mathbf{u}_{5})\\
\dot{\mathbf{u}}_{4} &= \mathbf{h}(\mathbf{u}_{4}, \mathbf{u}_{1}, \overline{\mathbf{u}_{2}, \mathbf{u}_{5}},\mathbf{u}_{3})\\
\dot{\mathbf{u}}_{5} &= \mathbf{h}(\mathbf{u}_{5}, \mathbf{u}_{1}, \overline{\mathbf{u}_{2}, \mathbf{u}_{4}},\mathbf{u}_{3})
\end{array}$} \\
& &\\
& &\\
& &\\
& & \\
& &
\end{tabular}
\end{center}
\caption{Symbolic adjacency matrix for network in Figure~\ref{fig:coupled_lorenz} which represents different weights on arrows, and the associated coupled cell system.}
\label{tab:adj_and_system_lorenz}
\end{table}

\begin{figure}[h!]
\begin{center}
\begin{tabular}{ccc}
\includegraphics[scale=0.3]{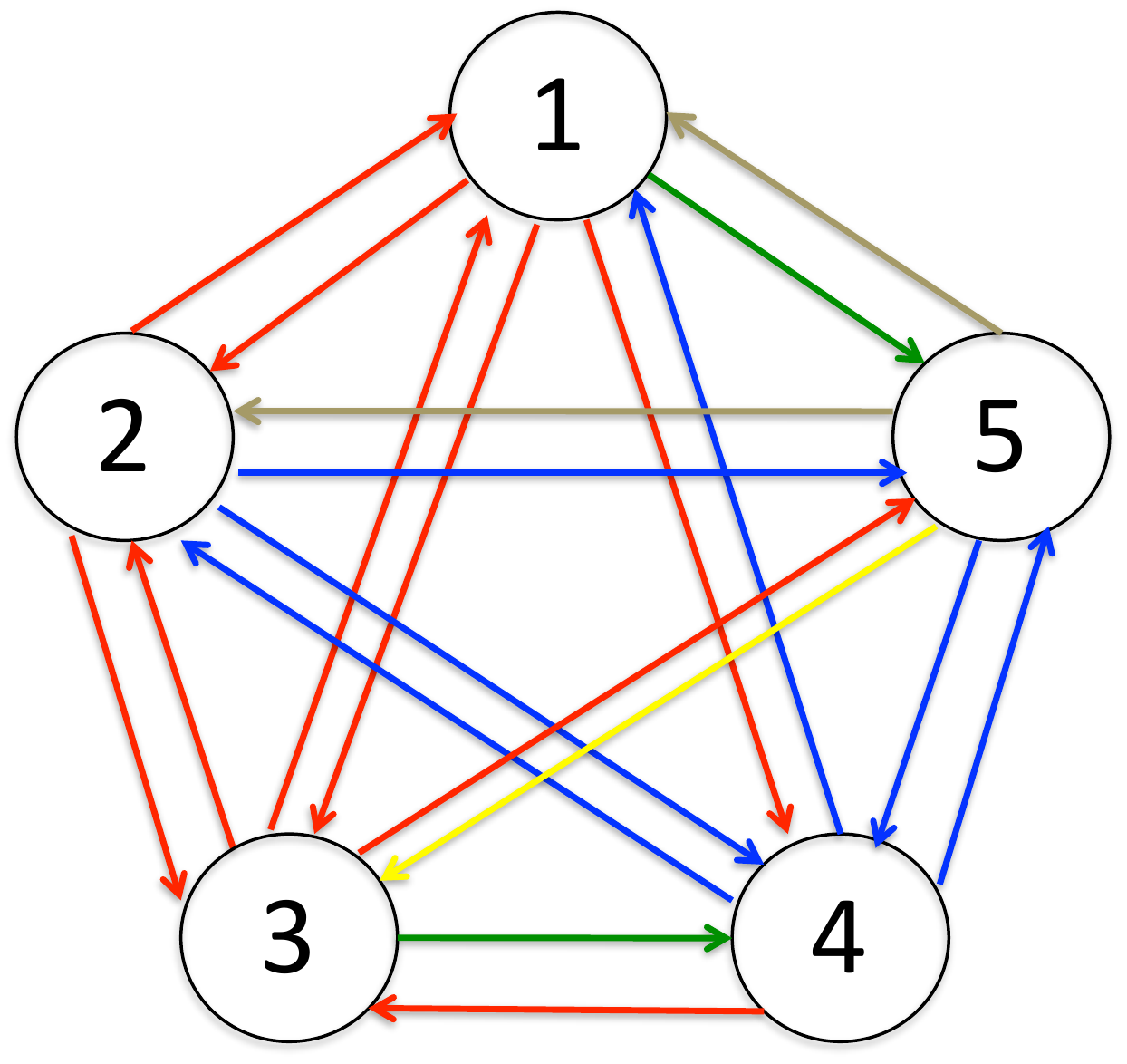} & \hspace{2cm} &
\includegraphics[scale=0.25]{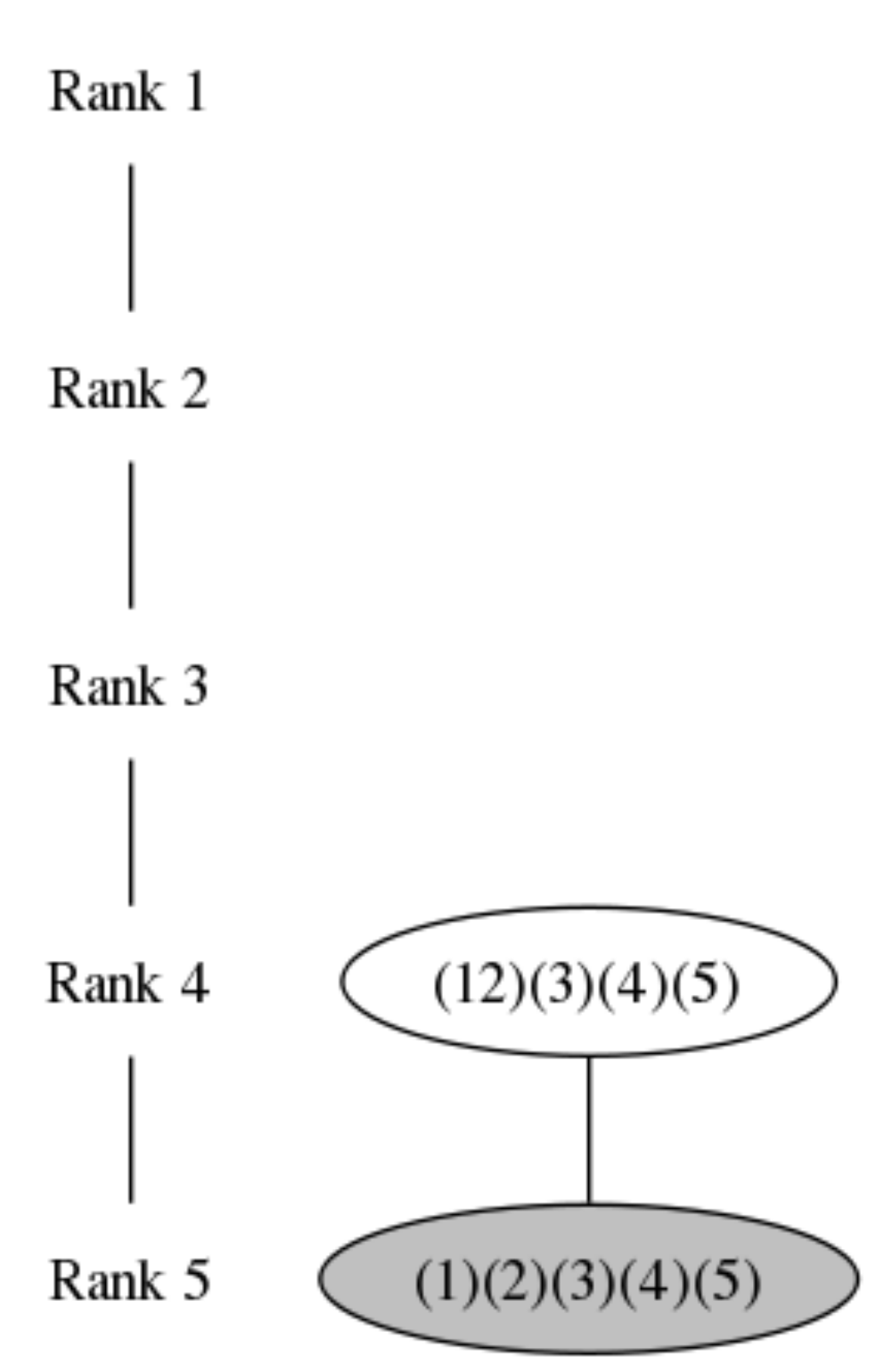}
\end{tabular}
\end{center}
\caption{Five globally coupled Lorenz systems and the associated lattice of balanced equivalence relations, which has a unique non-trivial balanced equivalence relation $(12)(3)(4)(5)$.}
\label{fig:coupled_lorenz}
\end{figure}

Figure \ref{fig:coupled_lorenz} shows five globally coupled Lorenz systems with different weighted arrows represented in different colors and the associated lattice of balanced equivalence relations of the network. If all weights are identical, the corresponding lattice of partial synchronies is the same as the partition lattice of $5$ elements with the Bell number $B_{5}=52$ lattice points. However, in this example with non-identical weights, there is only one non-trivial balanced equivalence relation given by $(12)(3)(4)(5)$, which is found to be unstable by numerical analysis.

We note the intriguing phenomenon termed \textit{bubbling} \cite{Ashwin1994} is related to the stability of synchrony subspaces. When the dynamics on the synchrony subspace is a chaotic attractor, small perturbations along the transverse direction of the synchrony subspace can induce intermittent bursting for some systems. This bubbling phenomenon is observed in synchrony subspaces corresponding to balanced coloring (see the example system ($14.1$) in \cite{Golubitsky-and-Stewart-2006}).

\subsection{Coupled neurons on a random network}
The Aldis \cite{Aldis} or Belykh and Hasler \cite{BelykhHasler2011} algorithm
finds the minimal balanced coloring, which is the balanced equivalence relation with the minimal number of colors (i.e. the minimal number of the synchronized clusters), and thus the top lattice node. Belykh and Hasler demonstrated this using a coupled identical Hindmarsh-Rose model \cite{HR} with $30$ neurons generated by randomly choosing bidirectional identical couplings between any two nodes with a small probability. Even though this network has only one type of cell and one type of coupling, it is not regular. This network has no apparent symmetry using the circular layout (Figure~\ref{fig:BH_thirty_network}(a)) which hides a local reflectional symmetry (Figure~\ref{fig:BH_thirty_network}(b)). The minimal balanced coloring has $23$ colors (i.e. $23$ synchronized clusters). Our algorithm shows the lattice of balanced equivalence relations contains only two lattice points, the trivial bottom lattice node (all distinct) and this one non-trivial balanced coloring. Modifying this network by adding random edges or rewiring existing edges can generate a more complex lattice, for example ten lattice nodes with a minimal balanced coloring of $18$ clusters (Figure~\ref{fig:BH_thirty_network}(c)). The Python code in the supplementary materials finds the lattices for both of these networks. A systematic exploration of the lattices of random networks, such as those generated by rewiring or the Watts-Strogatz model \cite{WattsStrogatz1998}, is an area for future work.

\begin{figure}[h]
\begin{center}
\begin{tabular}{cc}
\multirow{2}{*}[2.2cm]{\includegraphics[scale=0.3]{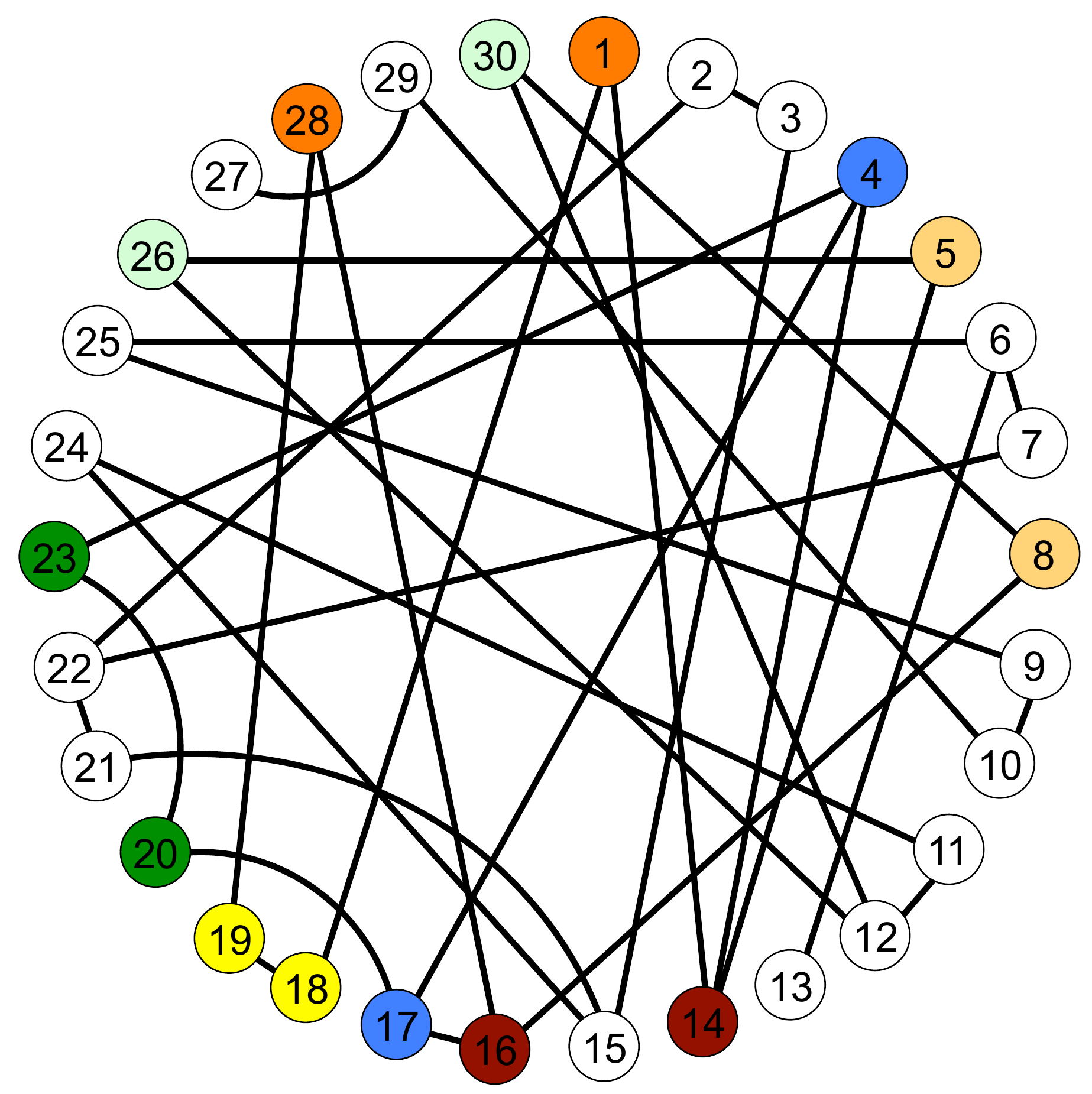}} &
\includegraphics[scale=0.3]{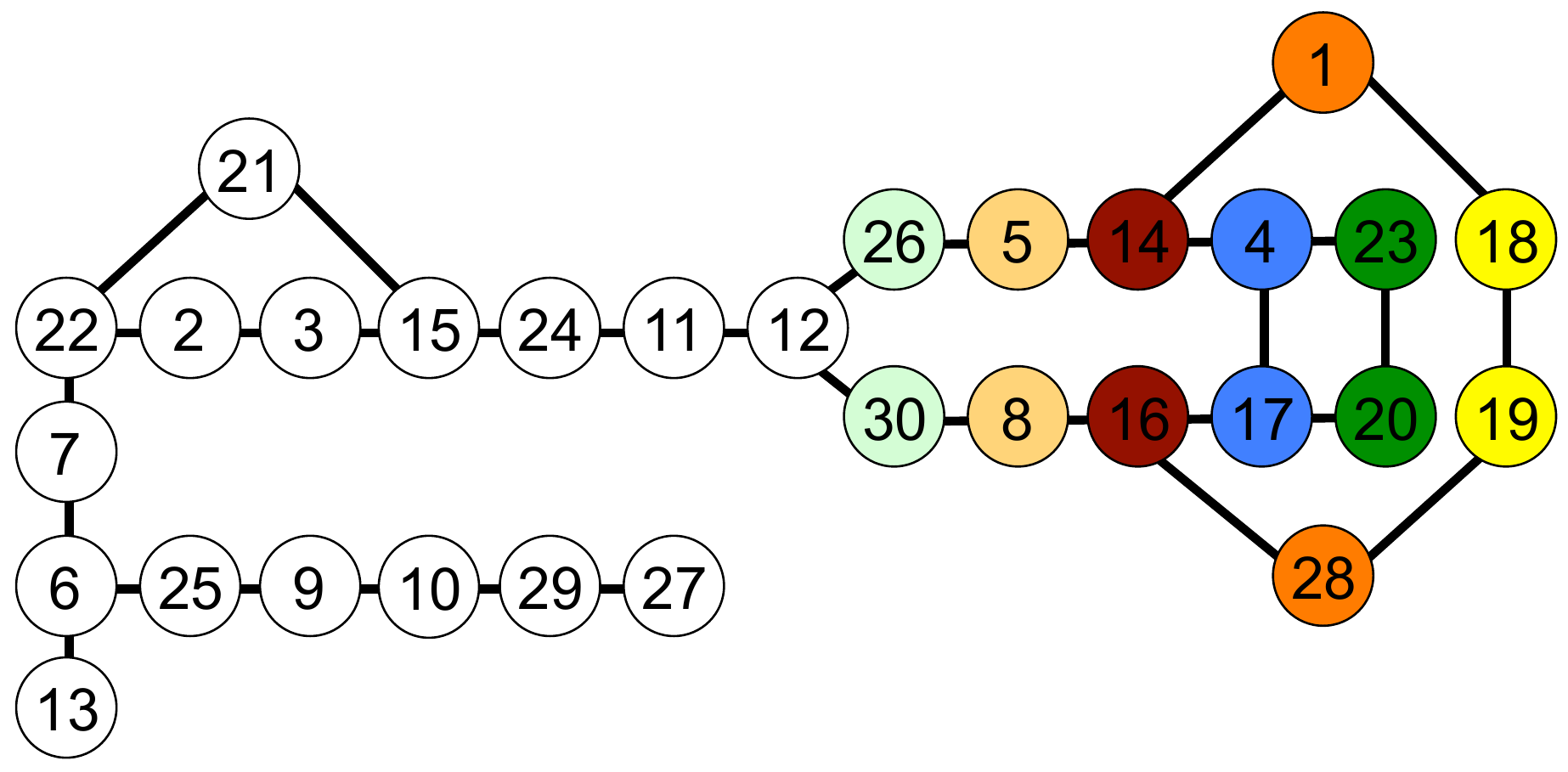} \\
& (b) \\
& \includegraphics[scale=0.3]{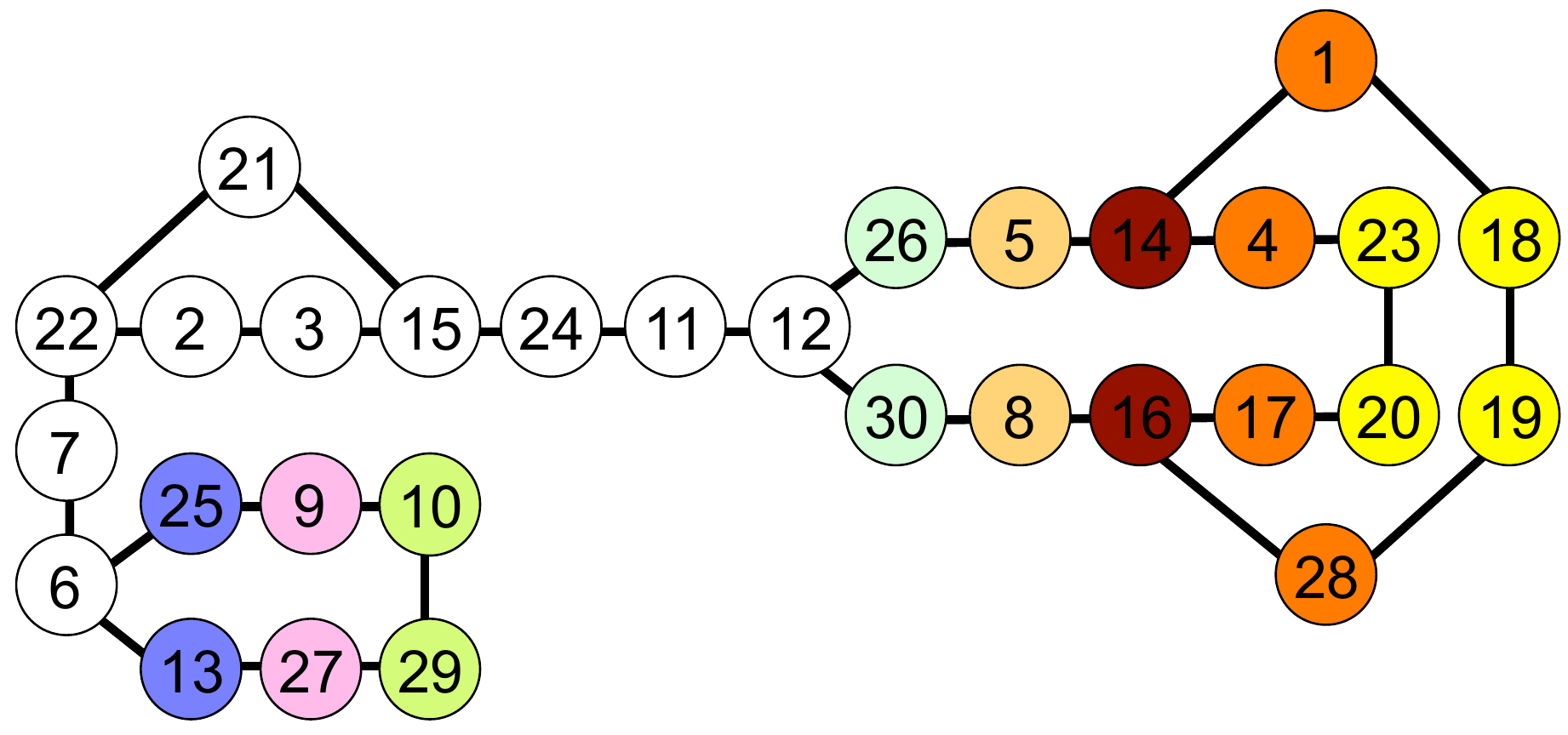} \\
(a) & (c)
\end{tabular}
\end{center}
\caption{Thirty neuron networks with bidirectional coupling. Colored nodes represent clusters in the minimal balanced coloring, with uncolored (white) nodes for distinct uni-clusters. Here (a) and (b) show the same network, from Figure 3(b) in \cite{BelykhHasler2011}, with $23$ clusters. The final figure (c) shows a rewired network connecting nodes $13$ and $27$ in place of $4$ and $17$, resulting in a richer lattice of ten nodes, with $18$ clusters in the minimal balanced coloring. Removing the coupling between $4$ and $17$ increases the local symmetry, as does coupling nodes $13$ and $27$.}
\label{fig:BH_thirty_network}
\end{figure}

\subsection{Excitatory/Inhibitory coupled neurons}
The network of inhibitory coupled FitzHugh-Nagumo model neurons shown in Figure~\ref{fig:HR_nine}(a) is discussed in \cite{Rabinovich2000,Rabinovich2001}.
Table~\ref{tab:adj_and_system_HR} shows the the $9\times 9$ adjacency matrix $C=(c_{ij})$ with integer entries and the associated coupled cell system.
A very similar neural network topology (deleting the arrow from neuron $1$ to $5$) is studied in \cite{Casado2003} using a discrete map instead of ODEs. We remark that our algorithm shows both network structures have the same lattice of $27$ balanced equivalence relations (see lattice generated by the Python code in the supplementary materials). The top lattice node is the synchronized cluster pattern $(19)(2378)(46)(5)$ which was discussed in \cite{Casado2003}.

\begin{figure}[h]
\begin{center}
\begin{tabular}{ccc}
\includegraphics[scale=0.2]{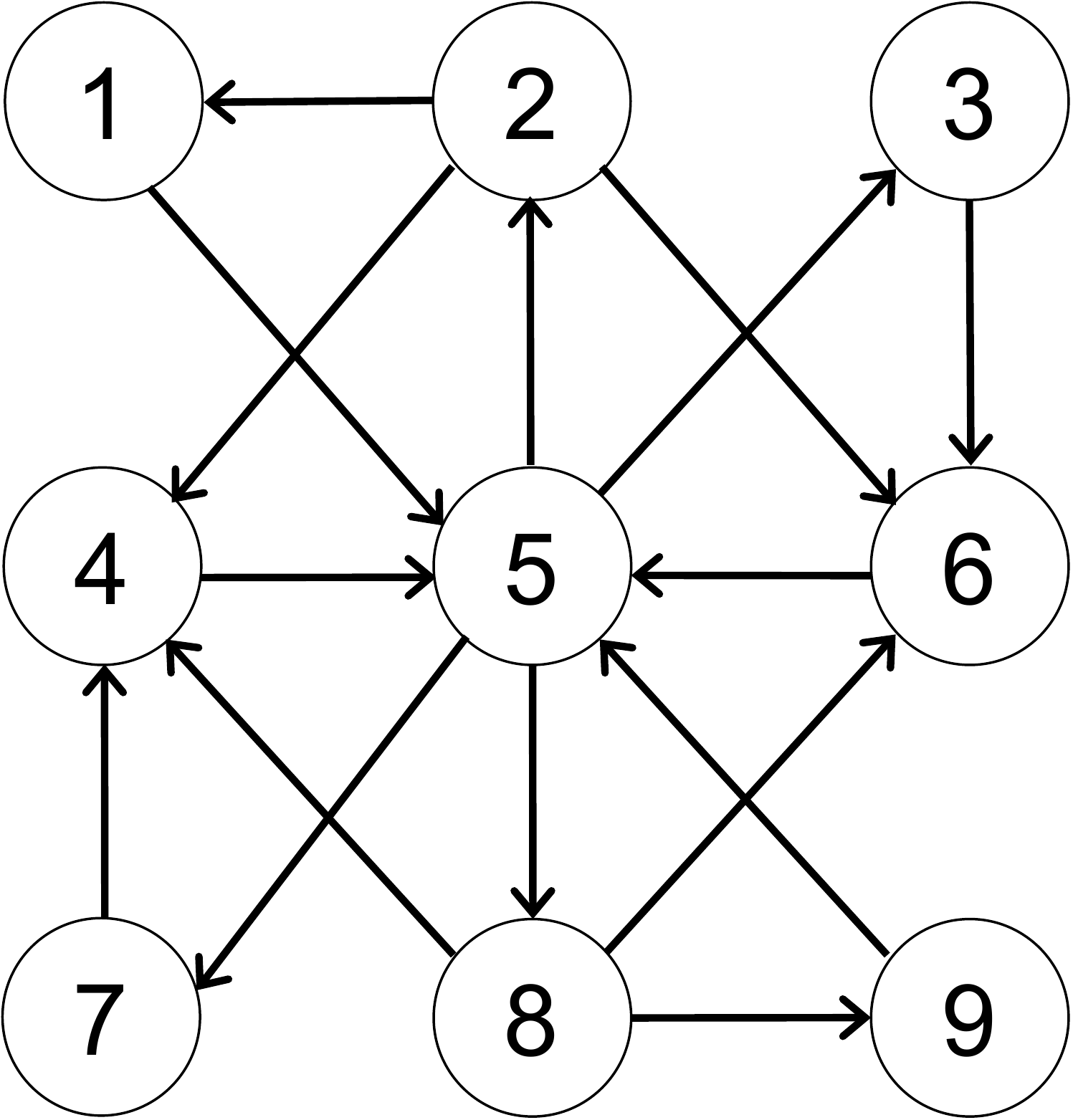} & \quad &
\includegraphics[scale=0.2]{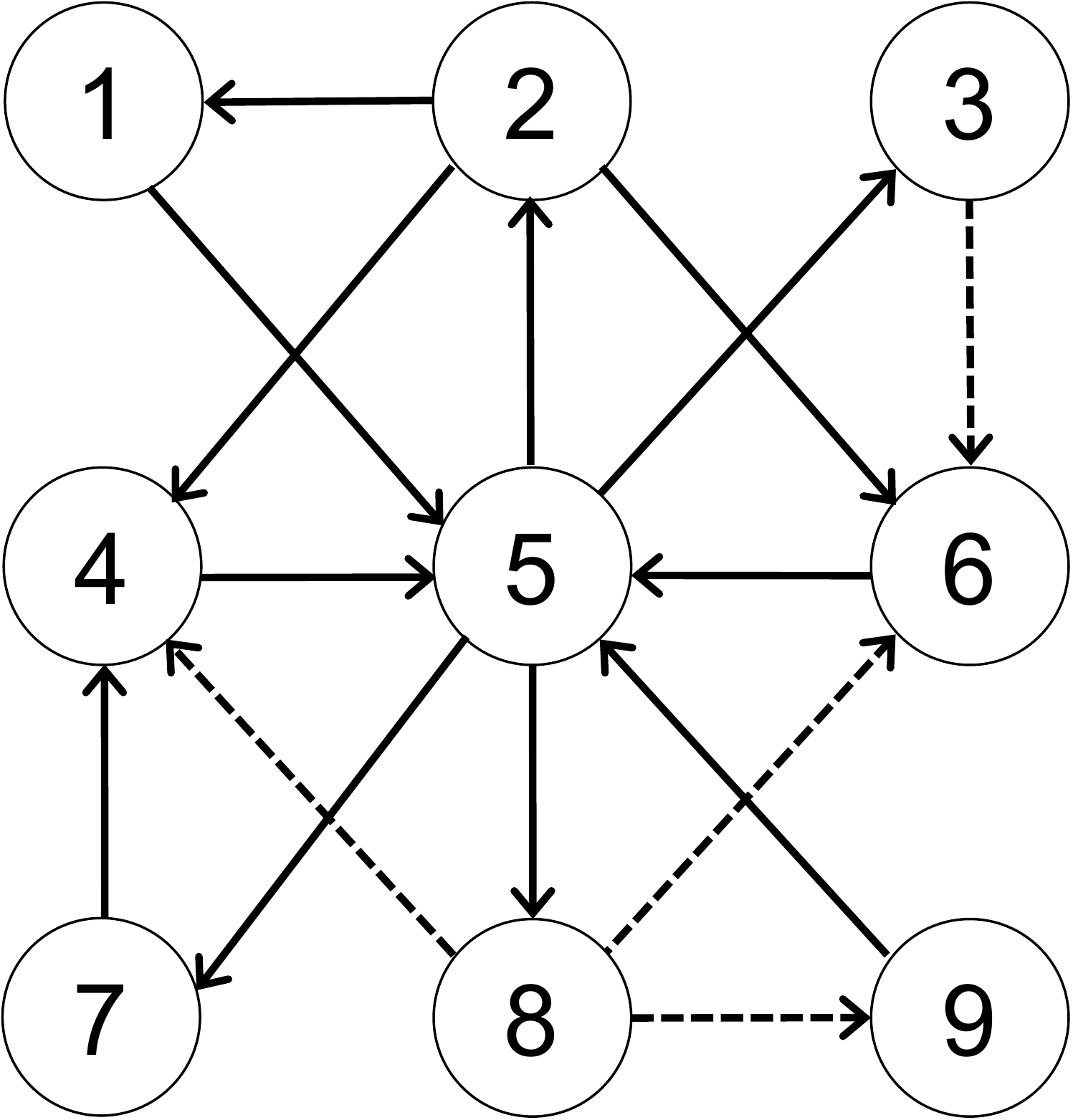} \\
(a) & & (b)
\end{tabular}
\end{center}
\caption{Nine neurons connected with (a) one coupling type as in \cite{Rabinovich2000,Rabinovich2001}, (b) two coupling types (excitatory and inhibitory). Solid arrows represent inhibitory coupling, dashed lines excitatory. In (b) the arrows from nodes $3$ and $8$ are excitatory.}
\label{fig:HR_nine}
\end{figure}

This network structure was studied as a \textit{winnerless competition network} \citep{Rabinovich2000,Rabinovich2001}, in which cluster states (unstable saddle states) are connected along a heteroclinic orbit. Such cluster states can correspond to balanced polydiagonals (invariant subspaces), which are defined by balanced equivalence relations \cite{Ashwin2007,Ashwin2011}. Thus the lattice of balanced equivalence relations might potentially be used to elucidate the possible robust heteroclinic cycles.

\begin{table}[h!]
\begin{center}
\begin{tabular}{ccl}
adjacency matrix & \hspace{1cm} & coupled cell system\\
\hline
& & \\
\multirow{9}{*}{
$C=\left(\begin{array}{rrrrrrrrr}
0 & 1 & 0 & 0 & 0 & 0 & 0 & 0 & 0\\
0 & 0 & 0 & 0 & 1 & 0 & 0 & 0 & 0\\
0 & 0 & 0 & 0 & 1 & 0 & 0 & 0 & 0\\
0 & 1 & 0 & 0 & 0 & 0 & 1 & 1 & 0\\
1 & 0 & 0 & 1 & 0 & 1 & 0 & 0 & 1\\
0 & 1 & 1 & 0 & 0 & 0 & 0 & 1 & 0\\
0 & 0 & 0 & 0 & 1 & 0 & 0 & 0 & 0\\
0 & 0 & 0 & 0 & 1 & 0 & 0 & 0 & 0\\
0 & 0 & 0 & 0 & 0 & 0 & 0 & 1 & 0
\end{array}\right)$} & \hspace{1cm} &
\multirow{9}{*}{$\begin{array}{rl}
\dot{\mathbf{u}}_{1} &= \mathbf{f}(\mathbf{u}_{1}, \mathbf{u}_{2})\\
\dot{\mathbf{u}}_{2} &= \mathbf{f}(\mathbf{u}_{2}, \mathbf{u}_{5})\\
\dot{\mathbf{u}}_{3} &= \mathbf{f}(\mathbf{u}_{3}, \mathbf{u}_{5})\\
\dot{\mathbf{u}}_{4} &= \mathbf{g}(\mathbf{u}_{4}, \overline{\mathbf{u}_{2}, \mathbf{u}_{7},\mathbf{u}_{8}})\\
\dot{\mathbf{u}}_{5} &= \mathbf{h}(\mathbf{u}_{5}, \overline{\mathbf{u}_{1},\mathbf{u}_{4}, \mathbf{u}_{6},\mathbf{u}_{9}}) \\
\dot{\mathbf{u}}_{6} &= \mathbf{g}(\mathbf{u}_{6}, \overline{\mathbf{u}_{2}, \mathbf{u}_{3}, \mathbf{u}_{8}}) \\
\dot{\mathbf{u}}_{7} &= \mathbf{f}(\mathbf{u}_{7}, \mathbf{u}_{5})\\
\dot{\mathbf{u}}_{8} &= \mathbf{f}(\mathbf{u}_{8}, \mathbf{u}_{5})\\
\dot{\mathbf{u}}_{9} &= \mathbf{f}(\mathbf{u}_{9}, \mathbf{u}_{8})\\
\end{array}$} \\
& &\\
& &\\
& &\\
& & \\
& & \\
& & \\
& & \\
& & \\
& &
\end{tabular}
\end{center}
\caption{Adjacency matrix which represents the network topology of $9$ coupled neurons in Figure~\ref{fig:HR_nine}(a), and the associated coupled cell system. Note that this network is not regular since it has three input equivalence classes.
Note also that the single coupling type constrains the linearized external couplings of the three maps $\mathbf{f}$, $\mathbf{g}$ and $\mathbf{h}$ to be the same.
}
\label{tab:adj_and_system_HR}
\end{table}

To demonstrate multiple arrow types, we modified the previous example to consider two coupling types, excitatory or inhibitory, as in Figure~\ref{fig:HR_nine}(b).
By changing the outputs of neurons $3$ and $8$ to be excitatory (i.e. changing four couplings), the number of balanced equivalence relations decreased from $27$ to $15$ (see lattices generated by the Python code in the supplementary materials). Table~\ref{tab:adj_and_system_HR_nonidentical} shows the corresponding symbolic adjacency matrix and the associated coupled cell system.

\begin{table}[h]
\begin{center}
\begin{tabular}{ccl}
symbolic adjacency matrix & \hspace{1cm} & coupled cell system\\
\hline
& & \\
\multirow{9}{*}{
$C=\left(\begin{array}{rrrrrrrrr}
0 & a & 0 & 0 & 0 & 0 & 0 & 0 & 0\\
0 & 0 & 0 & 0 & a & 0 & 0 & 0 & 0\\
0 & 0 & 0 & 0 & a & 0 & 0 & 0 & 0\\
0 & a & 0 & 0 & 0 & 0 & a & b & 0\\
a & 0 & 0 & a & 0 & a & 0 & 0 & a\\
0 & a & b & 0 & 0 & 0 & 0 & b & 0\\
0 & 0 & 0 & 0 & a & 0 & 0 & 0 & 0\\
0 & 0 & 0 & 0 & a & 0 & 0 & 0 & 0\\
0 & 0 & 0 & 0 & 0 & 0 & 0 & b & 0
\end{array}\right)$} & \hspace{1cm} &
\multirow{9}{*}{$\begin{array}{rl}
\dot{\mathbf{u}}_{1} &= \mathbf{f}(\mathbf{u}_{1}, \mathbf{u}_{2})\\
\dot{\mathbf{u}}_{2} &= \mathbf{f}(\mathbf{u}_{2}, \mathbf{u}_{5})\\
\dot{\mathbf{u}}_{3} &= \mathbf{f}(\mathbf{u}_{3}, \mathbf{u}_{5})\\
\dot{\mathbf{u}}_{4} &= \mathbf{k}(\mathbf{u}_{4}, \mathbf{u}_{8},\overline{\mathbf{u}_{2}, \mathbf{u}_{7}})\\
\dot{\mathbf{u}}_{5} &= \mathbf{l}(\mathbf{u}_{5}, \overline{\mathbf{u}_{1},\mathbf{u}_{4}, \mathbf{u}_{6},\mathbf{u}_{9}}) \\
\dot{\mathbf{u}}_{6} &= \mathbf{m}(\mathbf{u}_{6}, \mathbf{u}_{2}, \overline{\mathbf{u}_{3}, \mathbf{u}_{8}}) \\
\dot{\mathbf{u}}_{7} &= \mathbf{f}(\mathbf{u}_{7}, \mathbf{u}_{5})\\
\dot{\mathbf{u}}_{8} &= \mathbf{f}(\mathbf{u}_{8}, \mathbf{u}_{5})\\
\dot{\mathbf{u}}_{9} &= \mathbf{n}(\mathbf{u}_{9}, \mathbf{u}_{8})\\
\end{array}$} \\
& & \\
& & \\
& & \\
& & \\
& & \\
& & \\
& & \\
& & \\
& & 
\end{tabular}
\end{center}
\caption{Symbolic adjacency matrix which represents the network topology of $9$ coupled neurons with excitatory (symbol `a') and inhibitory (symbol `b') coupling in Figure~\ref{fig:HR_nine}(b), and the associated coupled cell system.}
\label{tab:adj_and_system_HR_nonidentical}
\end{table} 

\section{Conclusions}
\label{sec:conclusions}
Networks in real world applications are inhomogeneous. Individual systems in a network play different roles and they interact with each other in various ways. For example even in a simplified representation of gene regulatory networks, genes or proteins can interact either by activation or inhibition. We encoded different types of interaction in such inhomogeneous networks using a symbolic adjacency matrix. We considered possible partial synchronies of a given network, where the network elements can be grouped into clusters whose dynamics are self-synchronous. We are particularly interested in partial synchronies which are solely determined by the structure (topology) of the network, rather than the specific dynamics (such as function forms and parameter values). Such robust patterns of synchrony are associated with balanced equivalence relations, which can be determined by a matrix computation on the symbolic adjacency matrix. These symbolic adjacency matrices can alternatively be expressed as a linear combination of integer entry adjacency matrices for each coupling type, and the matrix computation applied to each arrow type matrix individually (as in the provided Python program). The later is simpler to implement as most programming languages or software packages do not support symbolic matrices. Using the Symbolic Python library (http://www.sympy.org/) our example program can be modified to work on symbolic matrices (not shown as the extra dependency complicates installation for no practical benefit).

The symbolic adjacency matrix therefore specifies the set of balanced equivalence relations for a network. From these the refinement relation gives a complete lattice. Rather than obtaining the lattice in this way by exhaustive computation, for the special case of regular networks with simple eigenvalues, Kamei \cite{Kamei,Kamei-part1,Kamei-part2} showed how to construct the lattice from building blocks related to the eigenvector/eigenvalues of the adjacency matrix, and use it  to predict the existence of codimension-one steady-state bifurcation branches from the fully synchronous state, and to classify synchrony-breaking bifurcation behaviors. Results in \cite{Elmhirst-Golubitsky,Leite,Aguiar-2007,Aguiar,Golubitsky-and-Lauterbach} also relate the eigenvalues of the Jacobian of a coupled cell system with the eigenvalues of the adjacency matrix of a homogeneous network for synchrony-breaking bifurcation analysis.
The connections between the algebraic properties of the lattice and the network dynamics are potentially of wide interest.

In general however, such theoretical approaches for the explicit construction of the lattice do not yet exist, leaving the ``brute force'' approach of calculating all the possible balanced equivalence relations as the only currently viable route. We hope that use of this algorithm will facilitate further theoretical work, and stimulate investigation linking lattice properties and synchronous dynamics, and ultimately links between network structure and dynamics.

\section*{Acknowledgments}
H. Kamei thanks Prof. Ian Stewart for his supervision while the foundations of this work was carried out at the University of Warwick. 
 
\bibliographystyle{siam}
\addcontentsline{toc}{section}{References}
\bibliography{hiroko_bibliography} 

\newpage
\section{Appendix - Algorithm to find the top lattice node}

The algorithm we use is a generalization of Belykh and Hasler~\citep{BelykhHasler2011} to consider multiple arrow types, or equivalently a simplification of the Aldis~\citep{Aldis} algorithm where phase one is eliminated at the cost of counting absent arrow type/tail-node color combinations. This simplification makes it easier to implement than the original Aldis algorithm, yet it is still more than fast enough for our needs. The nine-cell network in Figure~\ref{fig:HR_nine}(b) with two arrow types is used as an example:

\subsubsection*{Step 0}
Start by assigning the same color (node class) to each node, here shown in red. If multiple node types are considered as in Aldis, then each node type would be allocated a unique color. Aldis also classifies the arrows but we skip with that.

\subsubsection*{Step 1}
In each step compute the ``input driven refinement'' by tallying the inputs to each node according to the color of the node the input is from, \emph{and} the arrow type. After tabulation, unique input combinations give the next node partition.

See Figure~\ref{fig:appendix-gbh1}. Here we have one node color (red), and two arrow types (solid and dashed), so for each node there are two input counts (solid from red, dashed from red).
Here Aldis would also compute the same two input counts per node.

\begin{figure}[h!]
\begin{center}
\begin{tabular}{lc}
\multirow{3}{*}[-0.2cm]{\includegraphics[height=3cm]{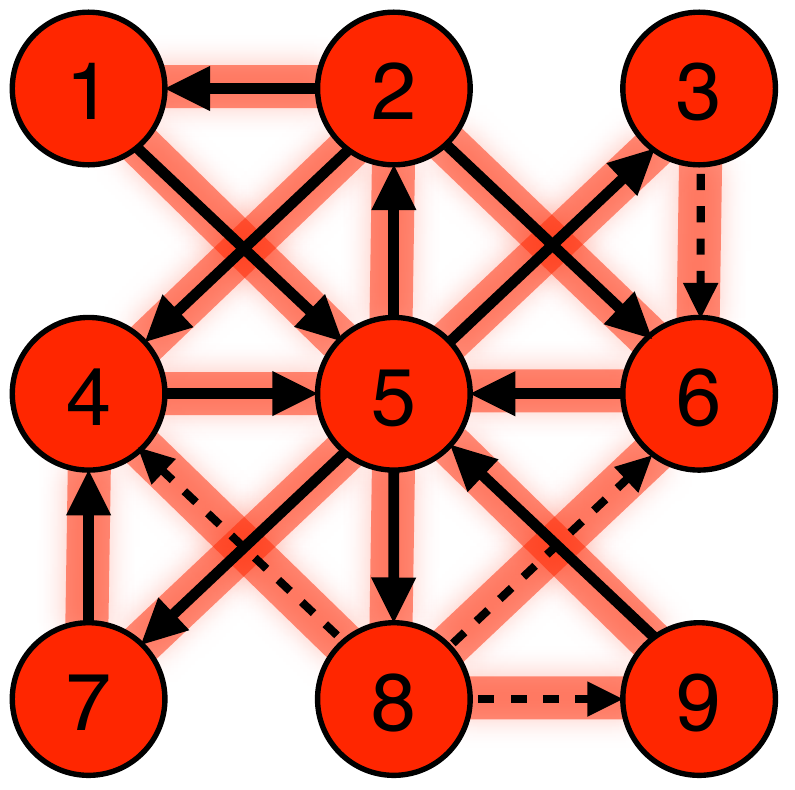}} & Old partition: (123456789) \\
 &
{\includegraphics[height=2.5cm]{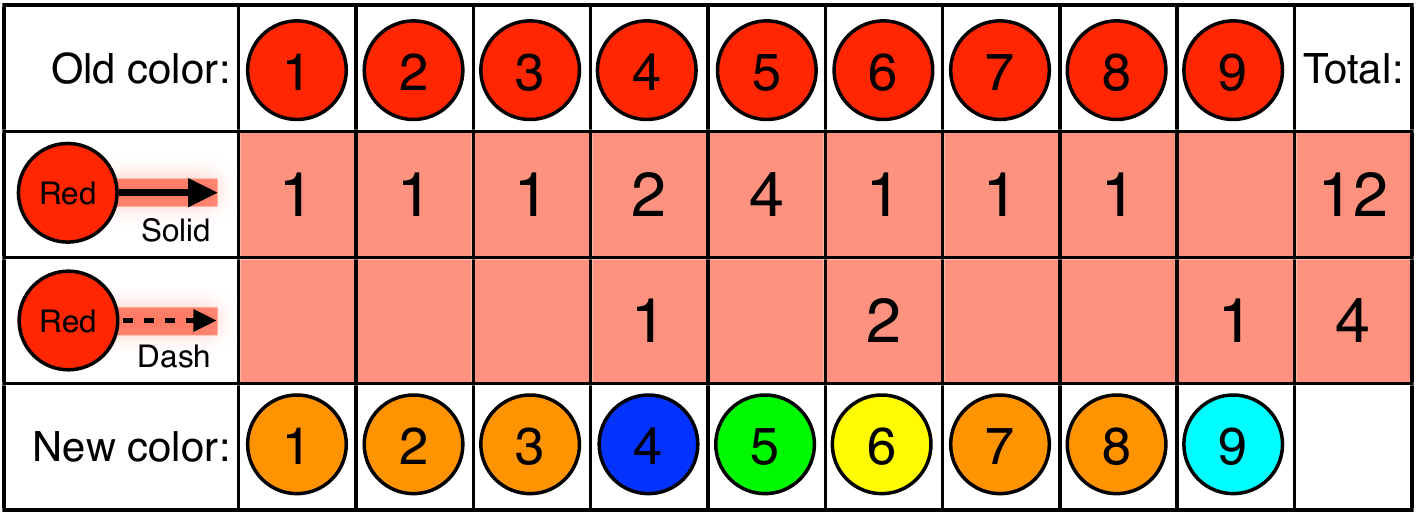}} \\
& New partition: (12378)(4)(5)(6)(9) \\
\end{tabular}
\end{center}
\caption{Finding Top Lattice Node (Step 1), input driven refinement of partition $(123456789)$.}
\label{fig:appendix-gbh1}
\end{figure}

We observe five unique input combinations, and so assign them five colors as the ``input driven refinement'' of the (trivial) input partition. For instance, nodes with one solid input from a red node only have been assigned the new partition color orange.

\subsection*{Step 2}

There are now five node colors (shown here as orange, blue, green, yellow and cyan). Thus with two arrow types (solid and dashed), we consider ten input types ($5{\times}2=10$; solid from orange, $\ldots$, solid from cyan, dashed from orange, $\ldots$, dashed from cyan). See Figure~\ref{fig:appendix-gbh2}.


\begin{figure}[h!]
\begin{center}
\begin{tabular}{lc}
& Old partition: (12378)(4)(5)(6)(9) \\
\multirow{1}{*}[6.5cm]{\includegraphics[height=3cm]{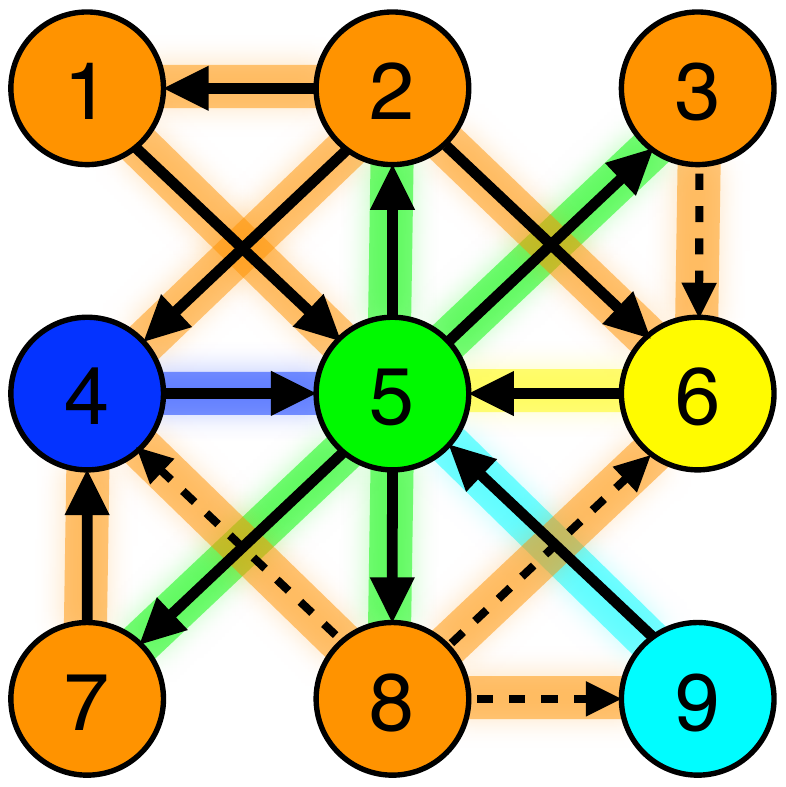}} &
{\includegraphics[height=7.5cm]{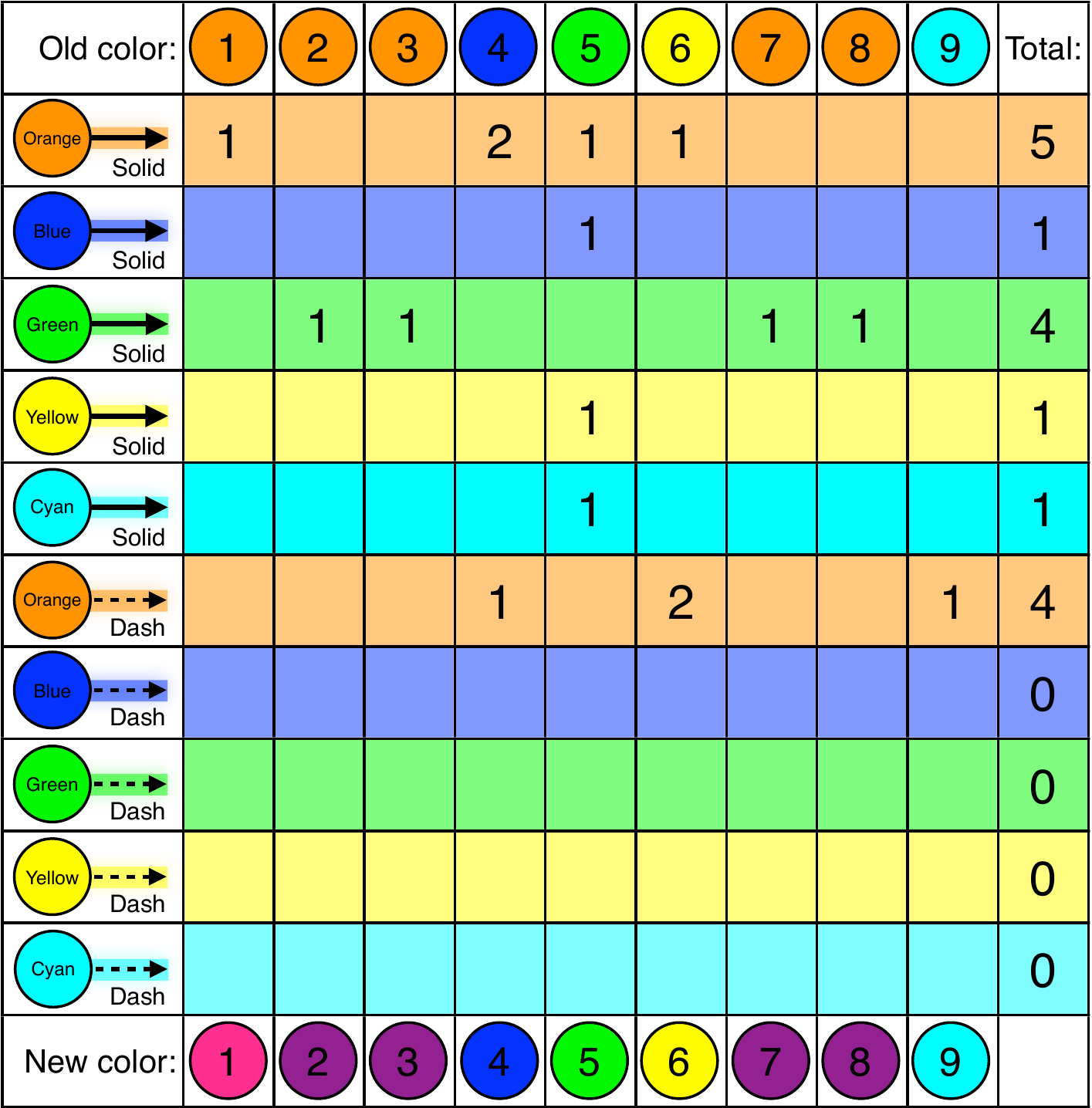}} \\
& New partition: (1)(2378)(4)(5)(6)(9) \\
\end{tabular}
\end{center}
\caption{Finding Top Lattice Node (Step 2), refinement of partition $(12378)(4)(5)(6)(9)$.}
\label{fig:appendix-gbh2}
\end{figure}

We observe six unique input combinations, giving six colors in the new node partition. The colors shown are arbitrary, and in the implementation are simply integers assigned incrementally. For this example we have reused the colors blue, green, yellow and cyan since those node groupings are unchanged. The former orange nodes have now been divided into pink and purple nodes.

Notice that of the ten possible input types tabulated here, the last four are absent. The Aldis algorithm avoids counting these.

\subsection*{Step 3}

There are now six node colors (pink, purple, blue, green, yellow and cyan), so with two arrow types (solid and dashed) we consider twelve input types ($6{\times}2=12$; solid from pink, $\ldots$, solid from cyan, dashed from pink, $\ldots$, dashed from cyan). See Figure~\ref{fig:appendix-gbh3}. At this iteration the partition of nodes is unchanged, and the algorithm halts. This gives the top lattice node.

\begin{figure}[h!]
\begin{center}
\begin{tabular}{lc}
& Old partition: (1)(2378)(4)(5)(6)(9) \\
\multirow{1}{*}[7.75cm]{\includegraphics[height=3cm]{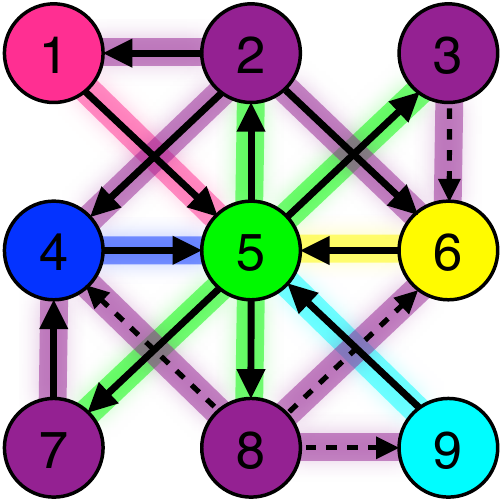}} &
{\includegraphics[height=8.75cm]{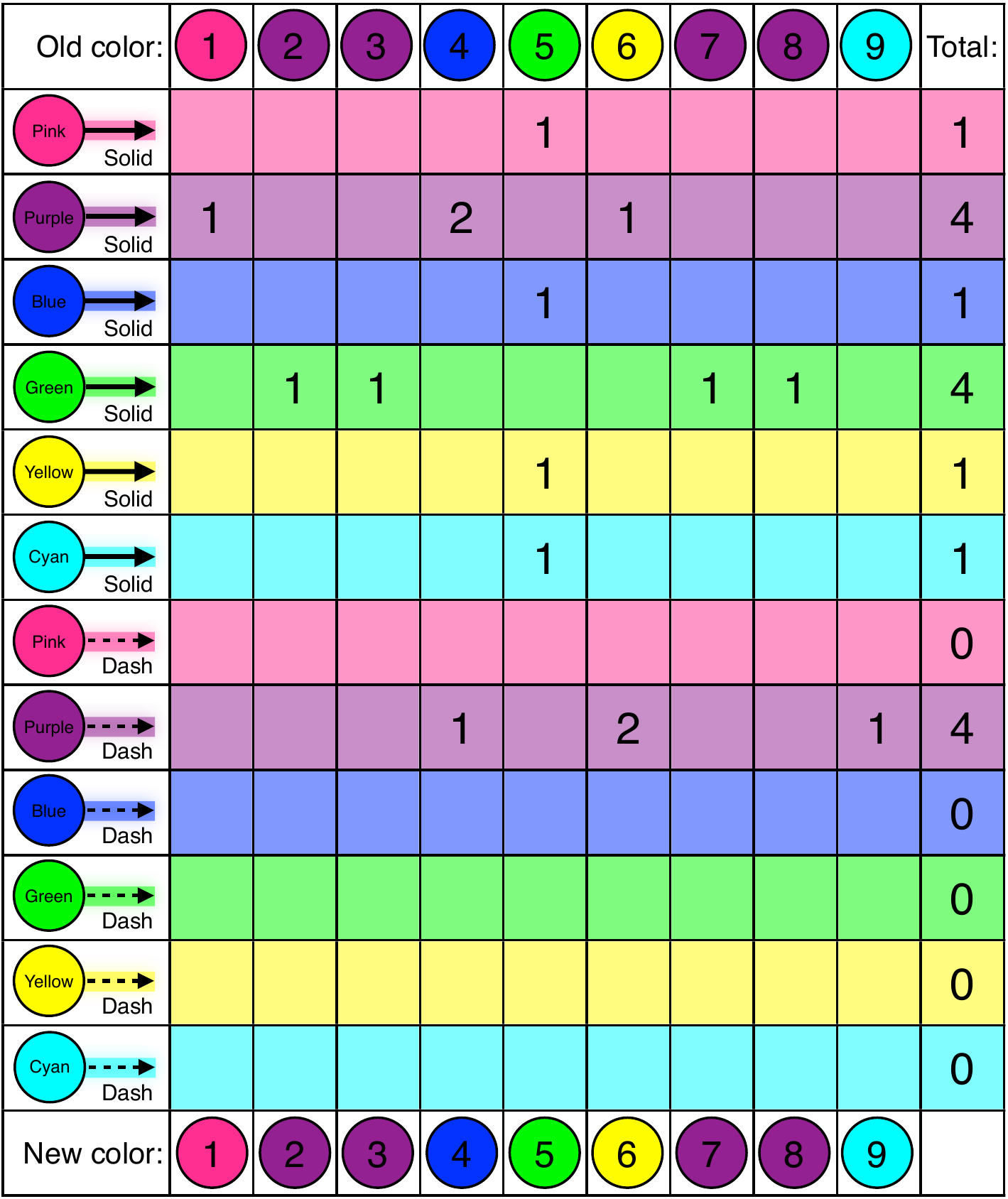}} \\
& New partition: (1)(2378)(4)(5)(6)(9) \\
\end{tabular}
\end{center}
\caption{Finding Top Lattice Node (Step 3), halts at partition $(1)(2378)(4)(5)(6)(9)$.}
\label{fig:appendix-gbh3}
\end{figure}

As in the previous step, some of the possible input types tabulated do not occur (five out of twelve), and the Aldis algorithm avoids counting these.

\subsubsection*{Remark}
When comparing the tables in steps 2 and 3, reusing the same color for node groups preserved between iterations highlights that the blue, green, yellow and cyan rows are unchanged. In this example at step 3 only the new pink and purple rows need be calculated (replacing the orange rows in step 2). This suggests a possible speed optimization when finding the top lattice node.

\subsubsection*{Differences between Aldis (2008) and our implementation}
In the above we have noted that the Aldis algorithm avoids computing the zero rows present in our tally tables. This is done by an additional phase in each iteration which tracks the arrow type and tail node color combinations as arrow equivalence classes, shown schematically in Figure~\ref{fig:appendix-aldis-tree}. This figure shows ancestry trees of the node partitions (left) and the observed combinations of arrow type (solid or dashed) with tail node color (right). By tracking the observed arrow type and tail-node color combinations explicitly, absent potential combinations need not be counted (i.e. dashed arrows from blue, green, yellow, cyan or pink nodes).

\begin{figure}[h!]
\begin{center}
\includegraphics[height=3.5cm]{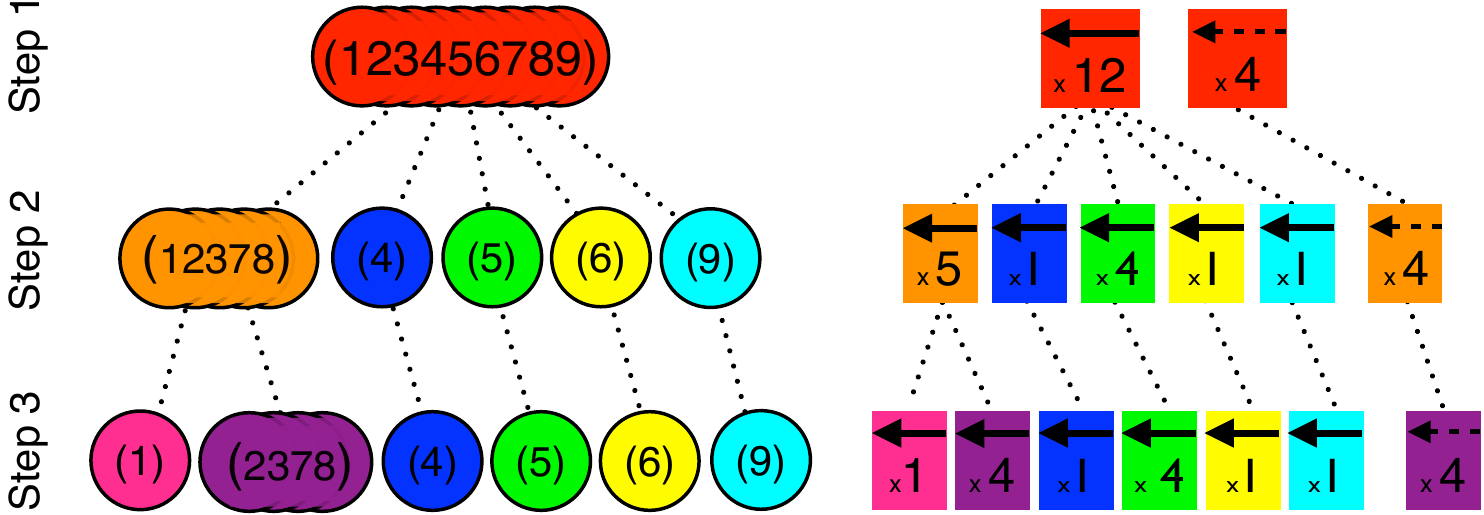}
\end{center}
\caption{Tree of node and arrow partitions as used in the Aldis algorithm. The node partitions on the left are $(123456789)$ at step 1 (all red), $(12378)(4)(5)(6)(9)$ at step two (five colors), and finally $(1)(2378)(4)(5)(6)(9)$ at step 3 (six colors). The arrow partitions are shown in the disjoint tree on the right, using our notation combining the arrow type (solid or dashed) and tail-node partition color. The numbers in each box indicate the number of arrows of that type from nodes of that color, for example in step 1 there are 12 solid arrows from red nodes.}
\label{fig:appendix-aldis-tree}
\end{figure}

The solid arrow tree and the dashed arrow trees in Figure~\ref{fig:appendix-aldis-tree} (right) are both sub-trees of the node partition tree (left). Our approach can be viewed as implicitly using the full node partition tree for each arrow type, at the cost of including redundant zero branches. This is a tradeoff between algorithmic complexity (Aldis) versus additional memory and computational overhead (not noticeable on the graph sizes considered). We expect the approach of Aldis to be most beneficial in large networks with many arrow types.

\end{document}